\documentclass[11pt,a4paper,reqno]{amsart}

\usepackage{amssymb,amsmath}
\usepackage{amsthm}
\usepackage{latexsym}
\usepackage[colorlinks=true,linktocpage=true,pagebackref=false, citecolor=black,linkcolor=black]{hyperref}
\usepackage{graphics}
\usepackage{euscript}
\usepackage{mathrsfs}
\usepackage{pdfsync}
\usepackage{tikz}
\usepackage{pgf}
\usetikzlibrary{arrows,decorations.pathmorphing,backgrounds,positioning,fit,matrix}
\usepackage{comment}
\usepackage{multirow}
\usepackage{xcolor}
\usepackage{subcaption}

\newtheorem{thm}{Theorem}[section]
\newtheorem{prop}[thm]{Proposition}
\newtheorem{cor}[thm]{Corollary}
\newtheorem{lem}[thm]{Lemma}

\theoremstyle{definition}

\newtheorem{defines}[thm]{Definitions}

\theoremstyle{remark}
\newtheorem{remark}[thm]{Remark}

\newtheorem{example}[thm]{Example}

 \newcommand{\N}{{\mathbb N}}
\newcommand{\Z}{{\mathbb Z}} \newcommand{\R}{{\mathbb R}}

\newcommand{\pol}{{\EuScript K}}

\newcommand{\p}{{\EuScript P}}

\newcommand{\Pp}{{\EuScript P}}
\newcommand{\Ss}{{\EuScript S}}

\newcommand{\Qq}{{\EuScript Q}}

\newcommand{\Vv}{{\EuScript V}}
\newcommand{\Ww}{{\EuScript W}}

\newcommand{\Rr}{{\EuScript R}}

\newcommand{\Bb}{{\EuScript B}}

\newcommand{\Tt}{{\EuScript T}}
\newcommand{\Hh}{{\EuScript H}}

\newcommand{\Ff}{{\EuScript F}}

\newcommand{\Cc}{{\EuScript C}}

\newcommand{\Ee}{{\EuScript E}}
\newcommand{\Gg}{{\EuScript G}}
\newcommand{\tildebaja}{{\raise.17ex\hbox{$\scriptstyle\sim$}}}

\newcommand{\Int}{\operatorname{Int}}

\newcommand{\pp}{\operatorname{p}}
\newcommand{\rr}{\operatorname{r}}

\newcommand{\cl}{\operatorname{Cl}}
\newcommand{\dist}{\operatorname{dist}}
\newcommand{\id}{\operatorname{id}}

\newcommand{\conv}[2]{{\vec{\EuScript C}_{#2}(#1)}}

\newcommand{\Span}{\text{Span}} 
\newcommand{\ven}{\vec{\tt e}_n} 
\newcommand{\vem}{\vec{\tt e}_{n-1}} 
\newcommand{\vcon}[1]{{\vec{{\EuScript C}}^{\sf v}_{#1}}} 
\newcommand{\acon}[2]{{{\EuScript C}^{\sf v}_{#1}(#2)}} 
\newcommand{\orig}{0} 
\newcommand{\origs}{\{0\}} 
\newcommand{\vspan}[2]{{#1#2}} 
\newcommand{\vspanp}[2]{{#1#2^{\hskip0.4mm+}}} 
\newcommand{\setg}[1]{{\mathfrak #1}} 
\newcommand{\sets}[1]{{\EuScript S}({#1})} 
\newcommand{\setst}[1]{{\EuScript S^*}({#1})} 
\newcommand{\seta}[1]{{\EuScript A}({#1})} 
\newcommand{\polq}[1]{Q_{#1}} 
\newcommand{\setd}[1]{{\EuScript A}^d({#1})} 

\newcommand{\expn}{\ell}

\newcommand{\x}{{\tt x}} \newcommand{\y}{{\tt y}}
\newcommand{\z}{{\tt z}} \renewcommand{\t}{{\tt t}}
 \renewcommand{\u}{{\tt u}} 
 \newcommand{\vv}{{\tt v}}

\newcommand{\ol}{\overline}
\newcommand{\qq}[2]{\langle #1,#2\rangle}
\newcommand{\veps}{\varepsilon}

\numberwithin{equation}{section}

\begin{document}
\title[Unbounded convex polyhedra as polynomial images of Euclidean spaces]{Unbounded convex polyhedra as polynomial\\ images of Euclidean spaces}
\author{Jos\'e F. Fernando}
\address{Departamento de \'Algebra, Facultad de Ciencias Matem\'aticas, Universidad Complutense de Madrid, 28040 MADRID (SPAIN)}
\email{josefer@mat.ucm.es}
\author{J.M. Gamboa}
\address{Departamento de \'Algebra, Facultad de Ciencias Matem\'aticas, Universidad Complutense de Madrid, 28040 MADRID (SPAIN)}
\curraddr{}
\email{jmgamboa@mat.ucm.es}
\author{Carlos Ueno}
\address{Dipartimento di Matematica, Universit\`a degli Studi di Pisa, Largo Bruno Pontecorvo, 5, 56127 PISA (ITALY)}
\email{cuenjac@gmail.com}
\thanks{First and second authors are supported by Spanish MTM2014-55565-P and Grupos UCM 910444 whereas the third author is an external collaborator. Third author was supported by `Scuola Galileo Galilei' Research Grant at the Dipartimento di Matematica of the Universit\`a di Pisa while writing a part of this article.
}

\begin{abstract}
In a previous work we proved that each $n$-dimensional convex polyhedron $\pol\subset\R^n$ and its relative interior are regular images of $\R^n$. As the image of a non-constant polynomial map is an unbounded semialgebraic set, it is not possible to substitute regular maps by polynomial maps in the previous statement. In this work we determine constructively all unbounded $n$-dimensional convex polyhedra $\pol\subset\R^n$ that are polynomial images of $\R^n$. We also analyze for which of them the interior $\Int(\pol)$ is a polynomial image of $\R^n$. A discriminating object is the recession cone $\conv{\pol}{}$ of $\pol$. Namely, \em $\pol$ is a polynomial image of $\R^n$ if and only if $\conv{\pol}{}$ has dimension $n$\em. In addition, \em $\Int(\pol)$ is a polynomial image of $\R^n$ if and only if $\conv{\pol}{}$ has dimension $n$ and $\pol$ has no bounded faces of dimension $n-1$\em. A key result is an improvement of Pecker's elimination of inequalities to represent semialgebraic sets as projections of algebraic sets. Empirical approaches suggest us that there are `few' polynomial maps that have a concrete convex polyhedron as a polynomial image and that there are even fewer for which it is affordable to show that their images actually correspond to our given convex polyhedron. This search of a `needle in the haystack' justifies somehow the technicalities involved in our constructive proofs.
\end{abstract}

\date{17/03/2017}

\maketitle
\section{Introduction}\label{s1} 

A map $f:=(f_1,\ldots,f_m):\R^n\to\R^m$ is \em polynomial \em if its components $f_k\in\R[\x]:=\R[\x_1,\ldots,\x_n]$ are polynomials. Analogously, $f$ is \em regular \em if its components can be represented as quotients $f_k=\frac{g_k}{h_k}$ of two polynomials $g_k,h_k\in\R[\x]$ such that $h_k$ never vanishes on $\R^n$. By Tarski-Seidenberg's principle \cite[1.4]{bcr} the image of an either polynomial or regular map is a semialgebraic set. A subset $\Ss\subset\R^n$ is \em semialgebraic \em when it has a description by a finite boolean combination of polynomial equalities and inequalities. 

It is quite natural to wonder about for properties that a set in $\R^m$ must satisfy in order to be the image of a polynomial map $f:\R^n\to\R^m$. To our knowledge, this question was first posed by Gamboa in an Oberwolfach week \cite{g}. A related problem concerns the parameterization of semialgebraic sets of dimension $d$ using continuous semialgebraic maps whose domains are semialgebraic subsets of $\R^d$ satisfying certain nice properties \cite{grs}. The approach proposed by Gamboa in \cite{g} sacrifices injectivity but chooses the simplest possible domains (Euclidean spaces) and the simplest possible maps (polynomial and regular) to represent semialgebraic sets. The class of semialgebraic sets that can be represented as polynomial and regular images of Euclidean spaces (even sacrificing injectivity) is surely much smaller than the one consisting of the images under injective continuous semialgebraic maps of nice semialgebraic sets. Of course, more general domains than the Euclidean spaces can be considered and compact semialgebraic sets deserve special attention: balls, spheres, compact convex polyhedra, \dots For instance, in \cite{kps} the authors develop a computational study of images under polynomial maps $\phi:\R^3\to\R^2$ (and the corresponding convex hulls) of compact (principal) semialgebraic subsets $\{f\geq0\}\subset\R^3$, where $f\in\R[\x_1,\x_2,\x_3]$ (this includes for example the case of a $3$-dimensional ball). 

The effective representation of a subset $\Ss\subset\R^m$ as a polynomial or regular image of $\R^n$ reduces the study of certain classical problems in Real Geometry to its study in $\R^n$. Examples of such problems appear in Optimization, with the advantage of avoiding contour conditions and reducing optimization problems to the case of Euclidean spaces (see for instance \cite{nds,ps,sch,vs} for relevant tools concerning optimization of polynomial functions on $\R^n$) or in the search for Positivstellens\"atze certificates \cite{s}. These representations provide Positivstellensatz certificates for general semialgebraic sets, whenever we are able to represent them as regular or polynomial images of $\R^n$. Recall that classical Positivstellensatz certificates are stated only for closed basic semialgebraic sets. Further details are described carefully in \cite{fgu1,fu2}. 

If $\Ss$ is a non-compact locally compact semialgebraic set in $\R^n$, it admits a (semialgebraic) Alexandrov compactification by one point. In addition, there is a doubly exponential (in the number $n$ of variables describing $\Ss$) algorithm triangulating each compact semialgebraic set (see \cite[Ch.9,\S2]{bcr} and \cite{hrr}). Thus, locally compact semialgebraic sets can be considered as finite simplicial complexes (up to losing one vertex), but we remark that the known algorithm can produce a doubly exponential number of simplexes. The algorithms developed to show that certain semialgebraic sets with piecewise linear boundary are polynomial or regular images of $\R^n$ are constructive (including those provided in this article), but the degrees of the involved maps are very high; however, it would be interesting to estimate the smallest degree for which there is a suitable polynomial or regular map, and to compare its complexity with the doubly exponential one for the triangulations of semialgebraic sets.

So far we have found partial answers to the representation problem of semialgebraic sets as polynomial and regular images of Euclidean spaces \cite{fg1,fg2,f1,fu1}, but a full geometric characterization of these sets seems difficult to be obtained at present. On the other hand, we have also focused on finding large families of semialgebraic sets that can be expressed as either polynomial or regular images of $\R^n$, giving constructive methods to obtain explicit maps producing them \cite{f1,fg1,fgu1,fgu2,fu5,u2}. In particular, we have focused our interest in determining whether convex polyhedra, their interiors and the corresponding complements can be expressed as polynomial or regular images. We understand that these types of semialgebraic sets are the simplest among those with piecewise linear boundary, and their full study is the first natural step to understand which semialgebraic sets whose boundaries have `nice properties' are either polynomial or regular images of $\R^n$.

In \cite{fgu1} we proved that every $n$-dimensional convex polyhedron $\pol\subset\R^n$ and its interior are regular images of $\R^n$. This result cannot be extended directly to the polynomial case because the image of a non-constant polynomial map is an unbounded semialgebraic set. Our purpose in this work is to determine all $n$-dimensional convex polyhedra $\pol\subset\R^n$ such that $\pol$ and/or $\Int(\pol)$ are polynomial images of $\R^n$. Here, $\Int(\pol)$ refers to the relative interior of $\pol$ with respect to the affine subspace of $\R^n$ spanned by $\pol$, which coincides with the interior of $\pol$ as a topological manifold with boundary. For these unbounded convex polyhedra, their representations as polynomial images of Euclidean spaces provide a priori simpler Positivstellens\"atze certificates and optimization approaches that if we use regular maps because polynomial representations do not involve denominators.

In \cite{fg1,fg2,fu1} we found obstructions for a semialgebraic set of $\R^n$ to be a polynomial image of some $\R^n$. Two distinguished ones that are relevant to us here are the following:

\noindent {\bf Condition 1:} \em The projections of a polynomial image of a Euclidean space are either singletons or unbounded semialgebraic sets\em.

\noindent {\bf Condition 2:} \em If a semialgebraic set $\Ss\subset\R^n$ is a polynomial image of $\R^n$ and $Z$ is an irreducible component of dimension $n-1$ of the Zariski closure of $\cl(\Ss)\setminus\Ss$, then $Z\cap\cl(\Ss)$ is unbounded \cite[Cor. 3.4]{fg2}\em.

Let us translate the first condition for convex polyhedra in terms of the recession cone. Given a point $p$ in a convex polyhedron $\pol\subset\R^n$, the set of vectors $\vec{v}\in\R^n$ such that the ray with origin $p$ and direction $\vec{v}$ is contained in $\pol$ is called the \em recession cone \em of $\pol$ (see \cite[Ch.1]{z} and \cite[II.\S8]{r}). This set does not depend on the chosen point $p$. We will see later in Proposition~\ref{ini1} that if the dimension of the recession cone $\conv{\pol}{}$ of a convex polyhedron $\pol$ is strictly smaller than its dimension, then $\pol$ has bounded, non-singleton projections and neither $\pol$ nor $\Int(\pol)$ are polynomial images of an Euclidean space.

On the other hand, translating the second condition to our polyhedral setting turns into the fact that if $\pol\subset\R^n$ is an $n$-dimensional convex polyhedron with a bounded face of dimension $n-1$, then $\Int(\pol)$ is not a polynomial image of $\R^n$.

Taking the previous obstructions in mind our main results in this work, which are the best possible ones, are the following:

\begin{thm}\label{main1}
Let $\pol\subset\R^n$ be an $n$-dimensional convex polyhedron whose recession cone is $n$-dimensional. Then $\pol$ is a polynomial image of $\R^n$ and $\Int(\pol)$ is a polynomial image of $\R^{n+1}$.
\end{thm}

\begin{thm}\label{main2}
Let $\pol\subset\R^n$ be an $n$-dimensional convex polyhedron without bounded facets and whose recession cone is $n$-dimensional. Then $\Int(\pol)$ is a polynomial image of $\R^n$.
\end{thm}

This means that for convex polyhedra, their interiors and the corresponding complements the known obstructions for the representability of general semialgebraic sets as polynomial images of Euclidean spaces are enough.

The proofs of Theorems \ref{main1} and \ref{main2} are rather technical and partly rely on {\it ad hoc} constructive arguments. With respect to the constructions we use to prove both results, it is difficult to determine how far from being 'optimal' they are. Even in the simplest non-trivial case of the open quadrant $\Qq:=\{\x>0,\y>0\}$ of $\R^2$, we have made several trials \cite{fg1,fgu2,fu5} to find the `best' possible representation of $\Qq$ as a polynomial image of $\R^2$. The criteria to measure the `goodness' of a representation are debatable, and we ourselves have been oscillating between the simplicity of the involved polynomial maps and the clearness of the proofs provided (an ideal situation would be to find examples where these two properties come together). A main difficulty, which permeates this work, is that our proofs are of constructive nature because we lack general principles that could provide a simpler and more direct existential approach to tackle the problems related to the representation of semialgebraic sets as polynomial images of Euclidean spaces. We point out here some obstacles that quickly arise when confronting them:
\begin{itemize}
\item The rigidity of polynomial maps hinders their manipulation in order to obtain the desired image sets. 
\item It is difficult to compute the image of an arbitrary polynomial map and, as far as we know, there are not feasible algorithms to achieve this.
\item The family of polynomial images do not behave nicely with respect to the usual set-theoretic operations or geometric constructions. 
\end{itemize}
We enlighten the latter fact with some examples.

\begin{example}[Convex hull of a polynomial image of $\R^n$] 
{\em The convex hull of a polynomial image of $\R^n$ needs not be a polynomial image of $\R^n$}. The semialgebraic set 
$$
\Ss:=\{\y>(\x+1)^2(\x-1)^2\}\subset\R^2 
$$
is a polynomial image of $\R^2$. Indeed, the upper half-plane $\Hh:=\{\y>0\}\subset\R^2$ is a polynomial image of $\R^2$ by \cite[Ex. 1.4, (iv)]{fg1} whereas $\Ss$ is the image of $\Hh$ via the polynomial map $\Hh\to\Ss,\ (x,y)\mapsto (x,y+(\x+1)^2(\x-1)^2)$.

The convex hull of $\Ss$ is the semialgebraic set 
$$
\Cc:=\{\y>(\x+1)^2(\x-1)^2\}\cup\{\y>0,-1<\x<1\},
$$
which is not a polynomial image of $\R^2$ by \cite[Thm. 3.8]{fg2}.
\end{example}

\begin{example}[Minkowski sum of polynomial images of $\R^n$] 
{\em The Minkowski sum $\Ss+\Tt$ of two polynomial images $\Ss$ and $\Tt$ of $\R^n$ needs not to be a polynomial image of $\R^n$}. We take the semialgebraic subsets 
$$
\Ss:=\{\x\ge 0,\y\ge 0,\x+\y\ge 1\}\quad\text{and}\quad\Tt:=\{\x>0,\y>0\}
$$
of $\R^2$, both of which are polynomial images of $\R^2$ by \cite[Thm. 5.1]{fg2} and \cite[Thm. 1.7]{fg1}. Their Minkowski sum is 
$$
\Ss+\Tt:=\{\x>0,\y>0,\x+\y> 1\},
$$
which is not a polynomial image of $\R^2$ by \cite[Cor. 3.4]{fg2}.
\end{example}

\begin{example}[Connected intersection of polynomial images of $\R^n$] 
{\em If the intersection $\Ss\cap\Tt$ of two polynomial images $\Ss$ and $\Tt$ of $\R^n$ is connected, then $\Ss\cap\Tt$ is not in general a polynomial image of $\R^n$}. The semialgebraic subsets $\Ss:=\{\x\le 1\}$ and $\Tt:=\{\x\ge-1\}$ of $\R^2$ are polynomial images of $\R^2$ whereas their intersection $\Ss\cap\Tt$, which is connected, is not a polynomial image of $\R^2$ because it does not satisfy Condition 1 above. Observe that $\Ss$, $\Tt$ and $\Ss\cap\Tt$ are convex semialgebraic sets.
\end{example}

\begin{example}[Connected union of polynomial images of $\R^n$] 
{\em If the union $\Ss\cup\Tt$ of two polynomial images $\Ss$ and $\Tt$ of $\R^n$ is connected, then $\Ss\cup\Tt$ is not in general a polynomial image of $\R^n$}. The semialgebraic subsets $\Ss:=\{\x\ge 0,\y\ge\x^2\}$ and $\Tt:=\{\y\ge0,\x\ge\y^2\}$ of $\R^2$ are polynomial images of $\R^2$ whereas their union $\Ss\cup\Tt$, which is connected, is not a polynomial image of $\R^2$ by \cite[Thm. 1.1]{fu1}. In fact, $\Ss$ and $\Tt$ are convex sets, but their union $\Ss\cup\Tt$ is not. 
\end{example}

We suspect that with the current knowledge it is difficult (or even plausibly impossible) to find two convex semialgebraic sets which are polynomial images of $\R^n$ whose union is convex but not a polynomial image of $\R^n$. The reason is the following: \em if two convex semialgebraic sets $\Ss$ and $\Tt$ satisfy all known obstructions to be polynomial images of $\R^n$ and their union $\Ss\cup\Tt$ is convex, then such union also satisfies all those known obstructions\em. So we have no known `a priori' tools to find such an example. In this regard, it would be relevant \em to determine whether the union of two convex polynomial images of $\R^n$ is also a polynomial image of $\R^n$ whenever such union is a convex set\em. A result of this nature will definitely help to determine all convex semialgebraic sets that are polynomial images of $\R^n$. However, at present we feel far from achieving this goal.

If we restrict our attention to the family ${\mathfrak F}$ of $n$-dimensional closed convex semialgebraic subsets of $\R^n$ with piecewise linear boundary that are polynomial images of $\R^n$, then $\Ss$ and $\Tt$ are by Theorem \ref{main1} $n$-dimensional convex polyhedra whose recession cone has dimension $n$. If the union $\Ss\cup\Tt$ is convex, then $\Ss\cup\Tt$ is again an $n$-dimensional convex polyhedron with recession cone of dimension $n$, so it is a polynomial image of $\R^n$ by Theorem \ref{main1}. Analogously, if we are interested in the family ${\mathfrak G}$ of $n$-dimensional open convex semialgebraic subsets of $\R^n$ with piecewise linear boundary that are polynomial images of $\R^n$, then $\Ss$ and $\Tt$ are, by Theorem \ref{main2}, $n$-dimensional convex polyhedra without bounded facets and whose recession cone has dimension $n$. If the union $\Ss\cup\Tt$ is convex, then $\Ss\cup\Tt$ is again an $n$-dimensional convex polyhedron without bounded facets and whose recession cone has dimension $n$. By Theorem \ref{main2} this union is a polynomial image of $\R^n$. 

In both cases above the result arises `a posteriori' because the union, if convex, of sets of either the family ${\mathfrak F}$ or ${\mathfrak G}$ is again a set of the family ${\mathfrak F}$ or ${\mathfrak G}$. We guess it is really difficult to develop a general strategy to prove `a priori' (without knowing the characterizations provided by Theorems \ref{main1} and \ref{main2}) that the union, if convex, of two convex semialgebraic sets with piecewise linear boundary that are polynomial images of $\R^n$ is again a polynomial image of $\R^n$. 

In order to circumvent these difficulties we have developed alternative strategies that rely on some constructions introduced in Pecker's work \cite{p}. The Tarski--Seidenberg principle on elimination of quantifiers can be also restated geometrically by saying that the projection of a semialgebraic set is again semialgebraic. An alternative converse problem, to find an algebraic set in $\R^{n+k}$ whose projection is a given semialgebraic subset of $\R^n$, is known as the \em problem of eliminating inequalities\em. Motzkin proved in \cite{m2} that this problem always has a solution for $k=1$. However, his solution is rather complicated and is generally a reducible algebraic set. In another direction Andradas--Gamboa proved in \cite{ag1,ag2} that if $\Ss\subset\R^n$ is a closed semialgebraic set whose Zariski closure is irreducible, then $\Ss$ is the projection of an irreducible algebraic set in some $\R^{n+k}$. In \cite{p} Pecker gives some improvements on both results: for the first one by finding a construction of an algebraic set in $\R^{n+1}$ that projects onto the given semialgebraic subset of $\R^n$, far simpler than the original construction of Motzkin; for the second one by characterizing the semialgebraic sets in $\R^n$ which are projections of a real variety in $\R^{n+1}$. In Section~\ref{s2} we modify Pecker's polynomials introduced in \cite[\S2]{p} to take advantage of them in order to prove both Theorems \ref{main1} and \ref{main2}.

To ease the presentation of the full picture of what is known \cite{fgu1,fu1,fu2,fu3,fu4,u2} about the representation of semialgebraic sets with piecewise linear boundary as either polynomial or regular images of some euclidean space $\R^m$ we introduce the following two invariants. Given a semialgebraic set $\Ss\subset\R^m$, we define
\begin{equation*}
\begin{split}
\pp(\Ss):&=\inf\{n\geq1:\exists \ f:\R^n\to\R^m\ \text{polynomial such that}\ f(\R^n)=\Ss\},\\
\rr(\Ss):&=\inf\{n\geq1:\exists \ f:\R^n\to\R^m\ \text{regular such that}\ f(\R^n)=\Ss\}.
\end{split}
\end{equation*}
The condition $\pp(\Ss):=+\infty$ expresses the non-representability of $\Ss$ as a polynomial image of some $\R^n$ whereas $\rr(\Ss):=+\infty$ has the analogous meaning for regular maps. The values of these invariants for the families of convex polyhedra and their complements are shown in Table~\ref{tabla}. Here, $\pol\subset\R^n$ represents an $n$-dimensional convex polyhedron and its complement $\Ss:=\R^n\setminus\pol$ is assumed to be connected. In addition, we write $\ol{S}:=\R^n\setminus\Int(\pol)$.
\begin{table}[ht]
$$
\renewcommand*{\arraystretch}{1.3}
\begin{array}{|c|c|c|c|c|c|c|c|c|c|}
\hline 
&\multicolumn{2}{c|}{\text{$\pol$ bounded}}&\multicolumn{4}{c|}{\text{$\pol$ unbounded}}\\[2pt] 
\cline{2-7} 
&n=1&n\geq2&n=1&\multicolumn{3}{c|}{n\geq2}\\[2pt] 
\hline
{\rm r}(\pol)&1&\multirow{2}{*}{$n$}&1&\multicolumn{3}{c|}{\multirow{2}{*}{$n$}}\\[2pt] 
\cline{1-2}\cline{4-4}
{\rm r}(\Int(\pol))&2&&2&\multicolumn{3}{c|}{}\\[2pt] 
\cline{1-1}\cline{2-7}
{\rm p}(\pol)&\multicolumn{2}{c|}{\multirow{2}{*}{$+\infty$}}&1&\multicolumn{3}{c|}{n, +\infty\ (\ast)}\\[2pt] 
\cline{1-1}\cline{4-7}
{\rm p}(\Int(\pol))&\multicolumn{2}{c|}{}&2&\multicolumn{3}{c|}{n, n+1, +\infty\ (\star)}\\[2pt] 
\cline{1-7}
{\rm r}(\Ss)&\multirow{4}{*}{$+\infty$}&\multirow{4}{*}{$n$}&2&\multicolumn{3}{c|}{\multirow{4}{*}{$n$}}\\[2pt] 
\cline{1-1}\cline{4-4}
{\rm r}(\ol{\Ss})&&&1&\multicolumn{3}{c|}{}\\[2pt] 
\cline{1-1}\cline{4-4}
{\rm p}(\Ss)&&&2&\multicolumn{3}{c|}{}\\[2pt] 
\cline{1-1}\cline{4-4} 
{\rm p}(\ol{\Ss})&&&1&\multicolumn{3}{c|}{}\\[2pt] 
\hline
\end{array}
$$
\caption{Full picture}\label{tabla}
\end{table}

Let us explain the (marked) cases in Table~\ref{tabla} which follow from this work:
\begin{itemize}
\item[$(\ast)$] $(n,\,+\infty)$: An $n$-dimensional convex polyhedron $\pol\subset\R^n$ has ${\rm p}(\pol)=n$ if and only if its recession cone $\conv{\pol}{}$ has dimension $n$. Otherwise, ${\rm p}(\pol)=+\infty$.
\item[$(\star)$] $(n,n+1,\,+\infty)$: If the recession cone $\conv{\pol}{}$ of an $n$-dimensional convex polyhedron $\pol$ has dimension $<n$, then ${\rm p}(\Int(\pol))=+\infty$. Otherwise, if $\pol$ has bounded facets, ${\rm p}(\Int(\pol))=n+1$ and if $\pol$ has no bounded facets, ${\rm p}(\Int(\pol))=n$.
\end{itemize}

\noindent{\bf Structure of the article.} The article is organized as follows. In Section~\ref{s1b} we introduce some basic notions, notations and tools that will be employed along the article. In Section~\ref{s2} we analyze further properties of Pecker's polynomials and we introduce some variations that fit the situation we need. In Section~\ref{s3} we prove Theorem~\ref{main1} whereas Theorem~\ref{main2} is proved in Section~\ref{s4}. We end this article with an appendix that collects some useful inequalities for positive real numbers.
 

\section{Preliminaries and basic tools}\label{s1b}
We proceed first to establish some basic concepts, notations and results. This section can be considered as a sort of toolkit, where diverse techniques and auxiliary tools that will be needed later are introduced.

\subsection{Basic notation} Points in the Euclidean space $\R^n$ are denoted with the letters $x$, $y$, $z$, $p$, $q$, \dots and vectors by $\vec{v}$, $\vec{w}$, \dots Given two points $p,q\in\R^n$, $\overrightarrow{pq}$ represents the vector from $p$ to $q$ and $\overline{pq}$ the segment joining them. Given an affine subspace $W\subset\R^n$, we use an overlying arrow $\vec{W}$ to refer to the corresponding linear subspace. This notation is extended in the following way: Given a finite union of affine subspaces $X:=X_1\cup\cdots\cup X_r$, we will denote $\vec{X}$ the union of the linear subspaces $\vec{X}_i$, so that
$$
\vec{X}:=\vec{X}_1\cup\cdots\cup\vec{X}_r.
$$
The vectors of the standard basis of $\R^n$ are denoted $\vec{\tt e}_i=(0,\dots,0,\overset{(i)}{1},0,\dots,0)$ for $i=1,\ldots,n$.

An \em affine hyperplane \em of $\R^n$ will usually be written as $H:=\{h=0\}$ using a non-zero linear equation $h$. It determines two \emph{closed half-spaces}
$$
H^+:=\{h\geq0\}\quad\text{and}\quad H^-:=\{h\leq0\}.
$$
In fact, these half-spaces depend on the linear equation $h$ chosen to define $H$. Whenever needed, we will clearly state the orientation that is being considered. 

An affine subspace $W$ of $\R^n$ is called \em vertical \em if $\vec{W}$ contains the vector $\ven$. Otherwise, we say that $W$ is \em non-vertical\em. In general, whenever an affine object or map is denoted with a symbol, we will use an overlying arrow on it to refer to its linear counterpart.

Given a set $X\subset\R^n$ and a set of vectors $\vec{V}\subset\R^n$, we define
$$
X+\vec{V}:=\{x+\vec{v}:\ x\in X, \vec{v}\in \vec{V}\}\subset\R^n.
$$
Whenever $X$ and $\vec{V}$ are convex sets, the set $X+\vec{V}$ is also convex. Given a set $X\subset\R^n$ and a vector $\vec{v}\in\R^n$, the \em cylinder of base $X$ in the direction $\vec{v}$ \em is defined as
$$
\vspan{X}{\vec{v}}:=\{x+\lambda\vec{v}:\,x\in X,\lambda\in\R\},
$$
and the \em positive cylinder of base $X$ in the direction $\vec{v}$ \em as
$$
\vspanp{X}{\vec{v}}:=\{x+\lambda\vec{v}:\,x\in X,\lambda\ge 0\}.
$$
We will use analogous notations $\vspan{\vec{X}}{\vec{v}}$ and $\vspanp{\vec{X}}{\vec{v}}$ when $\vec{X}$ is a set of vectors instead of a subset of $\R^n$. As special cases, the line through the point $p$ with direction $\vec{v}$ is written as $\vspan{p}{\vec{v}}$, whereas the ray with origin at $p$ and direction $\vec{v}$ is written as $\vspanp{p}{\vec{v}}$. Given $X_1,\ldots,X_m\subset\R^n$, we denote $\Span(X_1,\dots,X_m)$ the affine span of their union $\bigcup_{i=1}^mX_i$. 

\subsection{Convex polyhedra and recession cone}

A subset $\pol\subset\R^n$ is a \emph{convex polyhedron} if it can be described as a finite intersection of closed half-spaces. The dimension $\dim(\pol)$ of $\pol$ is the dimension of the smallest affine subspace of $\R^n$ that contains $\pol$ and $\Int(\pol)$ represents the relative interior of $\pol$ with respect to this subspace. If $\pol$ has non-empty interior there exists a unique minimal family $\{H_1,\ldots,H_m\}$ of affine hyperplanes such that $\pol=\bigcap_{i=1}^mH_i^+$. The \em facets \em or $(n-1)$-\em faces \em of $\pol$ are the intersections $\Ff_i:=H_i\cap\pol$ for $1\leq i\leq m$. Each facet $\Ff_i:=H_i\cap\bigcap_{j=1}^mH_j^+$ is a convex polyhedron contained in $H_i$. For $0\leq j\leq n-2$ we define inductively the $j$-\em faces \em of $\pol$ as the facets of the $(j+1)$-faces of $\pol$, which are again convex polyhedra. The $0$-faces are the \em vertices \em of $\pol$ and the $1$-faces are the \em edges \em of $\pol$. A face $\Ee$ of $\pol$ is \em vertical \em if the affine subspace of $\R^n$ spanned by $\Ee$ is vertical. Otherwise, we say that $\Ee$ is \em non-vertical\em. A convex polyhedron is \em non-degenerate \em if it has at least one vertex. Otherwise, it is called \em degenerate\em. For a detailed study of the main properties of convex sets we refer the reader to \cite{ber1,r,z}.

We associate to each convex polyhedron $\pol\subset\R^n$ its \em recession cone\em, see \cite[Ch.1]{z} and \cite[II.\S8]{r}. Fix a point $p\in\pol$ and denote $\conv{\pol}{}:=\{\vec{v}\in\R^n:\,\vspanp{p}{\vec{v}}\subset\pol\}$. Then \em $\conv{\pol}{}$ is a convex cone and it does not depend on the choice of $p$\em. The set $\conv{\pol}{}$ is called the \em recession cone \em of $\pol$. It holds $\conv{\pol}{}=\{\vec{\orig}\}$ if and only if $\pol$ is bounded. The recession cone of $\pol:=\bigcap_{i=1}^rH_i^+$ is 
$$
\conv{\pol}{}=\bigcap_{i=1}^r\conv{H_i^+}{}=\bigcap_{i=1}^r\vec{H_i}^+=\Big\{\sum_{i=1}^s\lambda_i\vec{v}_i:\ \lambda_i\geq0\Big\} 
$$
where the non-zero vectors $\vec{v}_1,\ldots,\vec{v}_s$ span the lines containing the unbounded edges of $\pol$. If a non-zero vector $\vec{v}\in\Int(\conv{\pol}{})$, then $\pol$ does not have facets parallel to $\vec{v}$.

If $\pol$ is non-degenerate we may write $\pol=\pol_0+\conv{\pol}{}$ where $\pol_0$ is the convex hull of the set of vertices of $\pol$. If $\p\subset\R^n$ is a non-degenerate convex polyhedron and $k\geq1$, then $\conv{\R^k\times\p}{}=\R^k\times\conv{\p}{}$. 

Recall that each degenerate convex polyhedron can be written as the product of a non-degenerate convex polyhedron times an Euclidean space. Besides, a convex polyhedron is degenerate if and only if it contains a line or, equivalently, if its recession cone contains a line. Consequently a convex polyhedron $\pol$ is non-degenerate if and only if all its faces are non-degenerate polyhedra.

The next result justifies the fact that the recession cone of a polyhedron plays an important role when we are trying to express it as a polynomial image of $\R^n$:

\begin{prop}\label{ini1}
If the dimension of the recession cone $\conv{\pol}{}$ of an $n$-dimensional convex polyhedron $\pol\subset\R^n$ is strictly smaller than $n$, then both $\pol$ and $\Int(\pol)$ have bounded non-singleton projections. Consequently, under the previous hypotheses both $\pol$ and $\Int(\pol)$ are not polynomial images of $\R^m$ for each $m\geq1$.
\end{prop}
\begin{proof}
We may assume $\conv{\pol}{}$ is contained in the hyperplane $\{\x_n=0\}$. Consider the projection $\eta:\R^n\to\R,\ x:=(x_1,\ldots,x_n)\mapsto x_n$. Suppose first that $\pol$ is non-degenerate. As $\dim(\pol)=n$, we can choose a set of points $\Ww:=\{p_1,\dots, p_k\}\subset\pol$ that contains all the vertices of $\pol$ and spans the whole space $\R^n$. Then $\pol=\pol_0'+\conv{\pol}{}$ where $\pol_0'$ is the convex hull of $\Ww$. As $\pol_0'$ is a compact polyhedron and has dimension $n$, the projection $\eta(\pol_0')$ is a non-trivial bounded interval. We have $\eta(\pol)=\eta(\pol_0')+\vec{\eta}(\conv{\pol}{})=\eta(\pol_0')$ because $\vec{\eta}(\conv{\pol}{})=\{\vec{0}\}$. Consequently, both $\eta(\pol)$ and $\eta(\Int(\pol))$ are bounded non-trivial intervals. 

Assume next that $\pol$ is degenerate and suppose $\pol=\R^k\times\pol'$ where $1\leq k<n$ and $\pol'\subset\R^{n-k}$ is a non-degenerate convex polyhedron of $\R^{n-k}$. Choose the notation $(\x_{k+1},\ldots,\x_n)$ for the coordinates of $\R^{n-k}$. As $\conv{\pol}{}=\R^k\times\conv{\pol'}{}$, we may assume $\conv{\pol'}{}\subset\{\x_n=0\}$. Let $\tau:\R^n\to\R^{n-k}$ denote the projection onto the last $n-k$ coordinates and let $\bar{\eta}:\R^{n-k}\to\R$ denote the projection onto the last coordinate, so that $\eta=\bar{\eta}\circ\tau$. We have
$$
\eta(\pol)=\eta(\R^k\times\pol')=(\bar{\eta}\circ\tau)(\R^k\times\pol')=\bar{\eta}(\pol').
$$ 
By the non-degenerate case $\bar{\eta}(\pol')$ and $\bar{\eta}(\Int(\pol'))$ are bounded intervals, as required. 
\end{proof}

Other results that follow from the use of the recession cone are the following.

\begin{lem}\label{dist}
Let $\pol\subset\R^n$ be a convex polyhedron and let $H:=\{h=0\}$ be a hyperplane of $\R^n$ such that $\pol\subset\{h>0\}$. Then $\dist(\pol,H)=\dist(p_0,H)$ for each point $p_0$ contained in one of the faces of $\pol$ of minimal dimension and in addition $\pol\subset\{h>\frac{h(p_0)}{2}\}$. 
\end{lem}
\begin{proof}
Assume first that $\pol$ is a non-degenerate convex polyhedron and write $\pol=\pol_0+\conv{\pol}{}$ where $\pol_0$ is the convex hull of the set $\Vv$ of vertices of $\pol$. As $\pol\subset\{h>0\}$, then $\mu:=\min\{h(p):\ p\in\Vv\}>0$ and $\vec{h}(\vec{v})\geq0$ for all $\vec{v}\in\conv{\pol}{}$. Observe that $h(q)\geq\mu$ for all $q\in\pol$ and $\dist(\pol,H)=\dist(p_0,H)$ where $p_0\in\Vv$ is a vertex such that $h(p_0)=\mu$. In addition, $\pol\subset\{h\geq h(p_0)\}\subset\{h>\frac{h(p_0)}{2}\}$. As the convex polyhedron $\pol$ is non-degenerate, $\{p_0\}$ is a face of $\pol$ of minimal dimension.

If $\pol$ is degenerate, we assume $\pol=\pol'\times\R^k$ where $\pol'\subset\R^{n-k}$ is a non-degenerate polyhedron. As $\pol\cap H=\varnothing$, we have $H=H'\times\R^k$ where $H':=\{h=0\}$ is a hyperplane of $\R^{n-k}$ and $n-k\geq1$. We abuse notation using the fact that the linear form $h$ only depends on the first $n-k$ variables. Applying the non-degenerate case to $\pol', H'$ and $h$ we find a vertex $q_0$ of $\pol'$ such that $\dist(\pol',H')=\dist(q_0,H')$. Observe that $\Ee:=\{q_0\}\times\R^k$ is a face of $\pol$ of minimal dimension and $h(p)=h(q_0,0)=h(q_0)$ for each $p\in\Ee$. The statement now follows straightforwardly.
\end{proof}

\begin{cor}\label{dist2}
Let $\pol\subset\R^n$ be a convex polyhedron and let $H_1:=\{h_1=0\}$ and $H_2:=\{h_2=0\}$ be hyperplanes of $\R^n$. Suppose that $\pol\cap H_1\subset\{h_2>0\}$. Then there exists $\veps>0$ such that $\pol\cap\{-\veps\leq h_1\leq\veps\}\subset\{h_2>0\}$.
\end{cor}
\begin{proof}
Define $\p:=\pol\cap\{h_2\leq0\}$. As $\pol\cap H_1\cap\{h_2\leq0\}=\varnothing$, we may assume $\p\subset\{h_1>0\}$. By Lemma~\ref{dist} there exists $\veps>0$ such that $\p\subset\{h_1>\veps\}$. Thus, $\pol\cap\{-\veps\leq h_1\leq\veps\}\subset\{h_2>0\}$, as required.
\end{proof}

\subsection{Vertical cones and convex polyhedra}

Along the article we will make frequent use of one particular direction in $\R^n$, the one given by the vector $\vec{\tt e}_{n}=(0,\dots,0,1)$. Set $x':=(x_1,\dots,x_{n-1})\in\R^{n-1}$ so that a point in $\R^n\equiv\R^{n-1}\times\R$ is written as $x:=(x',x_n)$. The \em vertical cone of radius $\delta>0$ \em is defined as 
$$
\vcon{\delta}:=\{(v',v_n)\in\R^n:\ \|v'\|\leq\delta v_n\}.
$$
Given a set $A\subset\R^n$ we define the \em vertical cone of radius $\delta>0$ over $A$ \em as
$$
\acon{\delta}{A}:=A+\vcon{\delta}=\{x+\vec{v}:\ x\in A,\ \vec{v}\in\vcon{\delta}\}.
$$
If $A$ is a convex set, then $\acon{\delta}{A}$ is also a convex set.

We establish now some results relating vertical cones and unbounded polyhedra.

\begin{lem}\label{prop:conpol}
Let $\pol\subset\R^n$ be a convex polyhedron such that $\ven\in\Int(\conv{\pol}{})$. Then there exists $\delta>0$ such that for each $p\in\R^n$ the inclusion $\acon{\delta}{\{p\}}\setminus\{p\}\subset\{p\}+\Int(\conv{\pol}{})$ holds.
\end{lem}
\begin{proof}
As $\ven:=(0,\ldots,0,1)\in\Int(\conv{\pol}{})$, there exists $\delta>0$ such that the ball $\Bb(\ven,\delta)$ of center $\ven$ and radius $\delta>0$ is contained in $\Int(\conv{\pol}{})$. As $\conv{\pol}{}$ is a cone with vertex $0$,
$$
\vcon{\delta}\setminus\origs\subset\{\lambda \vec{v}:\ \vec{v}\in\Bb(\ven,\delta),\ \lambda>0\}\subset\Int(\conv{\pol}{}).
$$ 
From this inclusion readily follows that $\acon{\delta}{\{p\}}\setminus\{p\}\subset\{p\}+\Int(\conv{\pol}{})$ for each $p\in\pol$.
\end{proof}

\begin{prop}\label{conpols}
Let $\pol\subset\R^n$ be a non-degenerate unbounded convex polyhedron. Assume $\pol\subset\{\x_n\ge 0\}$, the intersection $\Ee:=\{\x_n=0\}\cap\pol$ is a face of $\pol$ and the vector $\ven\in\Int(\conv{\pol}{})$. Then there exist positive numbers $\delta<\Delta$ such that $\acon{\delta}{\Ee}\subset\pol\subset\acon{\Delta}{\Ee}$.
\end{prop}
\begin{proof}
By Lemma~\ref{prop:conpol} we can choose $\delta>0$ such that $\acon{\delta}{\{p\}}\subset\{p\}+\conv{\pol}{}$ for each $p\in\Ee$, so that the inclusions $\acon{\delta}{\Ee}\subset\Ee+\conv{\pol}{}\subset\pol$ hold.

We prove next $\pol\subset\acon{\Delta}{\Ee}$ for $\Delta$ large enough. We may assume that $\orig\in\Int(\Ee)$. Observe first that for each $p\in\Ee$ we have $\{p\}\cup\{\x_n>0\}=\bigcup_{k\in\N}\acon{k}{\{p\}}$, so 
$$
\Ee\cup\{\x_n>0\}=\bigcup_{k\in\N}\acon{k}{\Ee}.
$$
Write $\conv{\pol}{}=\{\sum_{i=1}^s\lambda_i\vec{v}_i:\ \lambda_i\ge 0\}$ where the non-zero vectors $\vec{v}_1,\ldots,\vec{v}_s$ span the lines spanned by the unbounded edges of $\pol$. We may assume that the last coordinate of $\vec{v}_i$ is positive for $i=1,\ldots,r$ and identically zero for $\vec{v}_i$ with $i=r+1,\ldots,s$. Consequently, $\conv{\Ee}{}=\{\sum_{i=r+1}^s\lambda_i\vec{v}_i:\ \lambda_i\ge 0\}$. Pick $k_0\geq\delta$ such that: 
\begin{itemize}
\item[(1)] All the vertices of $\pol$ are contained in $\acon{k_0}{\Ee}$.
\item[(2)] The rays $\vspanp{\orig}{\vec{v}_i}\subset\acon{k_0}{\Ee}$ for $i=1,\ldots,s$. 
\end{itemize}
As $\acon{k_0}{\Ee}$ is convex (because $\Ee$ is convex) and $\origs+\conv{\pol}{}$ is the convex hull of the rays $\vspanp{\orig}{\vec{v}_i}$ for $i=1,\ldots,s$, we deduce that both sets $\origs+\conv{\pol}{}$ and the convex hull $\pol_0$ of the vertices of $\pol$ are contained in $\acon{k_0}{\Ee}$. Consequently,
$$ 
\pol=\pol_0+\conv{\pol}{}\subset\acon{k_0}{\Ee}
$$
and taking $\Delta:=k_0$ we have $\pol\subset\acon{\Delta}{\Ee}$, as required.
\end{proof}

\subsection{Projections of affine subspaces and convex polyhedra}

Given a hyperplane $H\subset\R^n$ and a vector $\vec{v}\in\R^n\setminus\vec{H}$, we denote by $\pi_{\vec{v}}:\R^n\to H$ the projection onto $H$ with direction $\vec{v}$. For each $X\subset\R^n$, the set $\pi_{\vec{v}}^{-1}(\pi_{\vec{v}}(X))$ coincides with $\vspan{X}{\vec{v}}$, so it does not depend on the chosen projection hyperplane $H$ but only on the vector $\vec{v}$. Write $x':=(x_1,\ldots,x_{n-1})$ and $x:=(x',x_n)$. We use often the vertical projection $\pi_{\ven}:\R^n\to\R^n,\ (x',x_n)\mapsto (x',0)$ onto the coordinate hyperplane $\{\x_n=0\}$ and we reserve the notation $\pi_n$ for this particular projection. 

\begin{prop}\label{proj}
Let $\pol\subset\{\x_n\geq0\}\subset\R^n$ be an unbounded convex polyhedron whose recession cone $\conv{\pol}{}$ has dimension $n$ and assume $\ven\in\Int(\conv{\pol}{})$. Then the restriction
$\rho:=\pi_n|_{\partial\pol}:\partial\pol\to\R^{n-1}\times\origs$ defines a semialgebraic homeomorphism.
\end{prop}
\begin{proof}
We prove first: \em $\rho$ is surjective\em. 

Pick a point $x:=(x',0)\in\R^{n-1}\times\{0\}$ and consider the ray $\vspanp{x}{\ven}$. Choose now $y\in\pol$. As $\ven\in\Int(\conv{\pol}{})$ and $\conv{\pol}{}$ has dimension $n$, there exists $\veps>0$ such that $\vec{w}:=\ven+\veps\overrightarrow{yx}
\in\conv{\pol}{}$. The ray $\vspanp{y}{\vec{w}}\subset\pol$ and 
$$
z:=y+\frac{1}{\veps}\vec{w}=x+\frac{1}{\veps}\ven\in\vspanp{y}{\vec{w}}\cap\vspanp{x}{\ven}\subset\pol\cap\vspanp{x}{\ven}.
$$
Consequently, $\vspanp{z}{\ven}\subset\pol\cap\vspanp{x}{\ven}\subset\{\x_n\geq0\}$, so there exists a point $p\in\partial\pol\cap\vspanp{x}{\ven}$, which satisfies $\pi_n(p)=x$. In addition, $\vspanp{x}{\ven}\cap\pol=\vspanp{p}{\ven}$.

We show next: \em $\rho$ is injective\em. It is enough to show: \em for each $x:=(x',0)\in\R^n$ the intersection $\vspan{x}{\ven}\cap\partial\pol$ is a singleton\em. 

We have already proved that $\vspan{x}{\ven}\cap\pol=\vspanp{p}{\ven}$ for some $p\in\partial\pol$. If the ray $\vspanp{p}{\ven}$ meets $\partial\pol$ in a point $y\neq p$, then either $\pol\cap\vspanp{p}{\ven}$ is a bounded interval or $\pol\cap\vspanp{p}{\ven}\subset\partial\pol$. As both situations are impossible because $\ven\in\Int(\conv{\pol}{})$, we conclude $\Int(\vspanp{p}{\ven})\subset\Int(\pol)$, so $\rho^{-1}(\rho(p))=\{p\}$.

To prove that $\rho$ is a homeomorphism, it is enough to check that it is a closed map and in fact it is sufficient that the restriction $\rho|_{\Ff}$ is a closed map for each facet $\Ff$ of $\pol$. Let $H$ be the hyperplane spanned by $\Ff$ and let us check that $\pi_n|_H$ is a closed map. As $\ven\in\Int(\conv{\pol}{})\subset\R^n\setminus\vec{H}$, the restriction $\pi_n|_H$ is an affine bijection and in particular a closed map, as required.
\end{proof}

Let us consider now a set $X\subset\R^n$ and a projection $\pi_{\vec{v}}:\R^n\to H$. The set $\pi_{\vec{v}}^{-1}(\pi_{\vec{v}}(X))=X\vec{v}$ contains $X$. If we consider now finitely many vectors $\vec{v}_1,\dots,\vec{v}_s$, the set $X':=\bigcap_{i=1}^sX\vec{v}_i$ also contains $X$. It seems natural to wonder under which conditions can we assert that $X'=X$. When $X$ is a finite union of affine subspaces of dimension $\le n-2$ we have the following result.

\begin{prop}\label{proj2}
Let $X:=\bigcup_{i=1}^mX_i\subset\R^n$ be a finite union of affine subspaces $X_i$ such that $d:=\dim(X)\le n-2$ and $X_i\nsubseteq X_j$ if $i\neq j$. Let $\Omega$ be a non-empty open subset of $\R^n\setminus\origs$. Then there exist finitely many vectors $\vec{v}_1,\ldots,\vec{v}_s\in\Omega$ such that $\bigcap_{i=1}^sX\vec{v}_i=X$. Besides, we can choose these vectors so that $\vec{v}_i\notin\bigcup_{j=1}^{i-1}
\vec{X}\vec{v}_j$ for $i=1,\dots,s$.
\end{prop}
\begin{proof}
As $X=\bigcup_{i=1}^mX_i$ and each $X_i$ is an affine subspace of $\R^n$ with $X_i\nsubseteq X_j$ if $i\neq j$, the affine subspaces $X_1, \dots, X_m$ are the irreducible components of $X$ as an algebraic subset of $\R^n$. Given $\vec{v}\in\R^n\setminus\{\vec{0}\}$, the set $X\vec{v}$ is also a finite union of affine subspaces of $\R^n$. For each irreducible component $X_i$ of $X$ the set $X_i\vec{v}$ is an affine subspace that either coincides with $X_i$ or has dimension $\dim(X_i)+1$ and contains $X_i$. If $p\in X_i\vec{v}\setminus X_i$, then $X_i\vec{v}=\Span(p,X_i)$. Set $\vec{X}:=\bigcup_{i=1}^m\vec{X}_i$. 

For $p\in\R^n\setminus X$ define $[p,X]:=\bigcup_{i=1}^m\Span(p,X_i)$. The set $\overrightarrow{[p,X]}$ denotes the union of the linear subspaces $\overrightarrow{\Span(p,X_i)}$ associated to the affine subspaces $\Span(p,X_i)$. We have $\dim([p,X])=\dim(\overrightarrow{[p,X]})\leq d+1$ and $p\notin X\vec{v}$ for each vector $\vec{v}\in\R^n\setminus\overrightarrow{[p,X]}$.

Pick $\vec{v}_1\in\Omega$ and let $Y_1^1,\dots,Y_1^s$ be the irreducible components of $Y_1:=X\vec{v}_1$. If each $Y_1^i\subset X$, we are done. Otherwise, assume $Y_1^1,\dots,Y_1^r$ are the irreducible components of $Y_1$ not contained in $X$ and pick $p_i\in Y_1^i\setminus X$ for $i=1,\ldots,r$. As $T_1:=\bigcup_{i=1}^r[p_i,X]$ is a finite union of affine subspaces of $\R^n$ whose dimensions are strictly smaller than $n$, there exists $\vec{v}_2\in\Omega\setminus(\vec{T}_1\cup\vec{Y}_1)$. We have $p_i\notin Y_2:=X\vec{v}_2$ for $i=1,\ldots,r$. Let $Z$ be an irreducible component of $Y_1\cap Y_2$ that is not contained in $X$. As $Z\subset Y_1$, there exists an irreducible component $Y_1^{i}$ of $Y_1$ not contained in $X$ such that $Z\subset Y_1^i$. In addition, $Z\subsetneq Y_1^i$ because $p_i\in Y_1^i\setminus Z$. Consequently, $\dim(Z)<\dim(Y_1^i)$ because $Z$ and $Y_1^i$ are affine subspaces. Thus, the dimension of every irreducible component of $Y_1\cap Y_2$ not contained in $X$ is strictly smaller than the dimension of some irreducible component of $Y_1$ that is not contained in $X$. We conclude $\dim((Y_1\cap Y_2)\setminus X)<\dim(Y_1\setminus X)$.

Next, for each irreducible component $Y_{12}^j$ of $Y_1\cap Y_2$ not contained in $X$ (and indexed with $j=1,\ldots,\ell$) we choose a point $q_j\in Y_{12}^j\setminus X$ and consider the set $T_2:=\bigcup_{j=1}^{\ell}[q_j,X]$. There exists $\vec{v}_3\in\Omega\setminus(\vec{T}_2\cup \vec{Y}_1\cup \vec{Y}_2)$ and we have $q_j\notin Y_3:=X\vec{v}_3$ for $j=1,\ldots,\ell$. The dimension of each irreducible component of $Y_1\cap Y_2\cap Y_3$ not contained in $X$ is strictly smaller than the dimension of some irreducible component of $Y_1\cap Y_2$ that is not contained in $X$. Again, this implies $\dim((Y_1\cap Y_2\cap Y_3)\setminus X)<\dim((Y_1\cap Y_2)\setminus X)$.

We repeat the process $s\leq d+3\leq n+1$ times to find $\vec{v}_1,\dots,\vec{v}_{s}\in\Omega$ such that 
$$
\vec{v}_i\notin \vec{Y}_1\cup\cdots\cup\vec{Y}_{i-1}
=\bigcup_{j=1}^{i-1}\vec{X}\vec{v}_j
$$ 
for $i=1,\ldots,s$ and all irreducible components of $\bigcap_{i=1}^sY_i:=\bigcap_{i=1}^sX
\vec{v}_i$ are contained in $X$. This holds because in each step $\dim((\bigcap_{i=1}^k Y_i)\setminus X)<\dim((\bigcap_{i=1}^{k-1} Y_i)\setminus X)$ for $k\geq2$. Consequently, $\bigcap_{i=1}^sX\vec{v}_i=X$, as required.
\end{proof}

\subsection{Separating hyperplanes in convex polyhedra}

Given two semialgebraic sets $\Ss_1,\Ss_2\subset\R^n$, we say that a hyperplane $B:=\{b=0\}\subset\R^n$ \em separates $\Ss_1$ and $\Ss_2$ \em if the semialgebraic sets $\Ss_i$ lie in the different half-spaces $\{b\geq0\}$ and $\{b\leq0\}$ determined by $B$ and $B\cap\Ss_i\subset\Ss_1\cap\Ss_2$ for $i=1,2$. Consequently, $\Ss_1\cap\Ss_2\subset B$ and $B\cap\Ss_i=\Ss_1\cap\Ss_2$ for $i=1,2$.

We are concerned here about hyperplanes that separate two adjacent facets of a convex polyhedron.

\begin{lem}\label{sep}
Let $\Ff_1$ and $\Ff_2$ be two non-parallel facets of a convex polyhedron $\pol\subset\R^n$. Let $H_i:=\{h_i=0\}$ be the hyperplane spanned by $\Ff_i$ and assume $\pol\subset\{h_1\geq0,h_2\geq0\}$. For each $\lambda>0$ denote $B_\lambda:=\{b_\lambda:=h_1-\lambda h_2=0\}$. Then $B_\lambda$ separates $\Ff_1$ and $\Ff_2$ and meets $\Int(\pol)$.
\end{lem}
\begin{proof}
Observe that $\Ff_1\subset\{b_\lambda\leq0\}$, $\Ff_2\subset\{b_\lambda\geq0\}$ and $B_\lambda\cap\Ff_i=\{h_1=0,h_2=0\}\cap\pol=\Ff_1\cap\Ff_2$ for $i=1,2$, so $B_\lambda$ separates $\Ff_1$ and $\Ff_2$. Let us check: $B_\lambda\cap\Int(\pol)\neq\varnothing$.

Pick $x_i\in\Int(\Ff_i)$. As $\Int(\ol{x_1x_2})\subset\Int(\pol)$, it is enough to check: $B_\lambda\cap\Int(\ol{x_1x_2})\neq\varnothing$. 

Set $\vec{v}=\overrightarrow{x_1x_2}$ and write each point $z\in\Int(\ol{x_1x_2})$ as 
$$
z=z_\mu:=x_1+\mu\vec{v}=x_2-(1-\mu)\vec{v}\in\Int(\ol{x_1x_2}).
$$
for some $0<\mu<1$. Observe that $h_1(x_1)=0$, $h_2(x_2)=0$, $\vec{h}_1(\vec{v})>0$ and $\vec{h}_2(\vec{v})<0$. All reduces to find a value $0<\mu<1$ such that $z_\mu\in B_\lambda$. To that end,
\begin{multline*}
0=b_{\lambda}(z_{\mu})=h_1(x_1+\mu\vec{v})-\lambda h_2(x_2-(1-\mu)\vec{v})=\mu\vec{h}_1(\vec{v})+\lambda(1-\mu)\vec{h}_2(\vec{v})\\
\leadsto\quad\mu:=\frac{-\lambda\vec{h}_2(\vec{v})}{\vec{h}_1(\vec{v})-\lambda\vec{h}_2(\vec{v})}.
\end{multline*}
As $0<\mu<1$, we have $z_\mu\in B_\lambda\cap\Int(\ol{x_1x_2})$, as required.
\end{proof}

We have denoted $\pi_n:\R^n\to\R^n,\ x:=(x_1,\ldots,x_n)\to(x_1,\ldots,x_{n-1},0)$ the orthogonal projection onto the hyperplane $\{\x_n=0\}$. Let us assume that a convex polyhedron $\pol$ is placed so that one of its facets $\Ff$ is vertical and $\ven\in\conv{\pol}{}$. The following result relates the projection of $\Int(\pol)$ under $\pi_n$ with the union of the projections under $\pi_n$ of the intersections of $\Int(\pol)$ with a family of separating hyperplanes between $\Ff$ and its adjacent facets. 
\begin{figure}[!ht]
\begin{center}
\begin{tikzpicture}[yscale=0.8]
\draw[draw=none,fill=gray!20,opacity=1]
(0,0)--(8,4)--(8,14)--(0,10);

\draw[draw=none,fill=gray!20,opacity=0.6]
(0,0)--(2,0)--(10,4)--(8,4);
\draw[thick,->](8,4)--(0,0);
\draw[thick,->](10,4)--(2,0);
\draw[thick, dashed] (8,4)--(10,4);
\draw[thick](4,2)--(6,2);
\draw[dashed](7,2.5)--(4,2)--(5,1.5);
\draw[draw=none,fill=gray!40,opacity=0.9]
(0,0.2)--(2,0.2)--(7,2.7)--(4,2.2)--cycle;
\draw[thin,dashed,<->]
(2,0.2)--(7,2.7)--(4,2.2)--(0,0.2);
\draw[thin,dashed,fill=gray!40,opacity=0.9]
(4,2.4)--(5,1.9)--(10,4.4)--(8,4.4)--cycle;
\draw (6.8,3.3) node[rotate=22]{$\pi_n(B_2\cap\Int(\pol))$};
\draw (3,1.1) node[rotate=22]{$\pi_n(B_1\cap\Int(\pol))$};

\draw[draw=none,fill=gray!40,opacity=0.6]
(8,8)--(10,8)--(10,14)--(8,14)--cycle;
\draw[thick,->](10,8)--(10,14);
\draw[thick,->,dashed](10,8)--(8,8)--(8,14);

\draw[draw=none,fill=gray!50,opacity=0.6]
(8,8)--(8,14)--(0,10)--(4,6)--cycle;
\draw[thick, dashed](4,6)--(8,8);
\draw[thick,->](4,6)--(0,10);

\draw[draw=none,fill=gray!50,opacity=0.6]
(4,6)--(6,6)--(10,8)--(8,8)--cycle;
\draw[thick] (4,6)--(6,6)--(10,8);
\draw[thick,dashed] (4,6)--(8,8)--(10,8);

\draw[draw=none,fill=black!100,opacity=0.6]
(4,6)--(8,8)--(10,9.5)--(5,7)--cycle;
\draw[gray](4,6)--(5,7)--(10,9.5);
\draw[gray,dashed](8,8)--(10,9.5);

\draw[draw=none,fill=black!100,opacity=0.6]
(4,6)--(7,6.5)--(3,10.5)--(0,10)--cycle;
\draw[gray](3,10.5)--(7,6.5);
\draw[gray,dashed](4,6)--(7,6.5);
\draw[dashed](4,6)--(6,7.5);

\draw[draw=none,fill=gray!80,opacity=0.6]
(0,10)--(4,6)--(6,6)--(2,10)--cycle;
\draw[thick,<->] (0,10)--(4,6)--(6,6)--(2,10);

\draw[draw=none,fill=gray!50,opacity=0.6]
(2,10)--(6,6)--(10,8)--(10,14)--cycle;
\draw[thick,<->](2,10)--(6,6)--(10,8)--(10,14);

\draw[dashed](4,6)--(5,7)(4,6)--(7,6.5);

\draw[very thick,->] (5,5.2)--(5,3.7);
\draw(5.4,4.45) node{$\pi_3$};
\draw(6,10.5) node{{\Large$\pol$}};
\draw(3,8) node{$\Ff_1$};
\draw(5,12) node{$\Ff$};
\draw(7.5,7.2) node{$\Ff_2$};
\draw(5,9) node{$B_1$};
\draw(7.3,8.6) node{$B_2$};
\end{tikzpicture}
\end{center}
\caption{$\pi_n(\Int(\pol))=\pi_n(B_1\cap
\Int(\pol))\cup\pi_n(B_2\cap\Int(\pol))$}\label{fig0}
\end{figure}

\begin{lem}\label{allplanos0}
Let $\pol\subset\R^n$ be an unbounded convex polyhedron and $\Ff$ one of its facets. Assume that $\Ff$ lies in the hyperplane $\{\x_{n-1}=0\}$ and the vector $\ven\in\conv{\pol}{}$. Let $\Ff_1,\ldots,\Ff_r$ be the non-vertical facets of $\pol$ and assume that all of them meet $\Ff$. Let $B_i$ be a hyperplane of $\R^n$ that separates $\Ff$ and $\Ff_i$ and meets $\Int(\pol)$. Then $\pi_n(\Int(\pol))=\bigcup_{i=1}^r\pi_n(B_i\cap\Int(\pol))$. Consequently, $\Int(\pol)\ven=\bigcup_{i=1}^r(B_i\cap\Int(\pol))\ven$.
\end{lem}
\begin{proof}
We prove first: 
\em\begin{equation}\label{planos}
\pi_n(\Ff_i)\setminus\partial\pi_n(\pol)
\subset\pi_n(B_i\cap\Int(\pol))
\end{equation}
for $i=1,\ldots,r$\em.

Take $x\in\pi_n(\Ff_i)\setminus\partial\pi_n(\pol)$. As $\Ff_i$ is non-vertical, $\vspan{x}{\ven}\cap\pol=\vspanp{p}{\ven}$ for some $p\in\Ff_i$. We claim: \em $\vspan{x}{\ven}\cap\partial\pol=\{p\}$\em.

Otherwise, $\vspanp{p}{\ven}\subset\partial\pol$ and $\{x\}=\pi_n(\vspan{x}{\ven})\subset\partial\pi_n(\pol)$. The latter inclusion follows because all the facets that contain $\vspanp{p}{\ven}$ are vertical, so their projections are contained in $\partial\pi_n(\pol)$, which is a contradiction. 

Let us check: \em $B_i$ is non-vertical\em. 

Otherwise, pick $q\in\Ff\cap\Ff_i\subset B_i$. As $\Ff\subset\{\x_{n-1}=0\}$, the ray $\vspanp{q}{\ven}\subset B_i\cap\Ff$. As $B_i$ separates $\Ff$ and $\Ff_i$, we have $\vspanp{q}{\ven}\subset B_i\cap\Ff\subset\Ff\cap\Ff_i$, so $\Ff_i$ should be vertical, which is a contradiction. 

The line $\vspan{x}{\ven}$ meets $B_i$ in a point $z$. We claim: \em $z\in\Int(\vspanp{p}{\ven})\subset\Int(\pol)$, so $x=\pi_n(z)\in\pi_n(B_i\cap\Int(\pol))$\em. 

As $\pi_n(\Ff)\subset\partial\pi_n(\pol)$ because $\Ff$ is vertical, $x\not\in\pi_n(\Ff)$. Consequently, $p\not\in B_i$ because otherwise $p\in B_i\cap\Ff_i\subset\Ff\cap\Ff_i$ and $x=\pi_n(p)\in\pi_n(\Ff)$, which is a contradiction.

Let $q\in\Ff\cap\Ff_i\subset B_i$ and let $b_i=0$ be a linear equation of $B_i$. As $B_i$ is non-vertical, we may assume $\vec{b}_i(\ven)>0$, so $\Int(\vspanp{q}{\ven})\subset\{b_i>0\}$ because $b_i(q)=0$. As $\vspanp{q}{\ven}\subset\Ff$, we deduce $\Ff\subset\{b_i\geq0\}$, so $\Ff_i\subset\{b_i\leq0\}$. As $p\in\Ff_i\setminus B_i$, we have $b_i(p)<0$. Write $z=p+\lambda\ven$, so
$$
0=b_i(z)=b_i(p)+\lambda\vec{b}_i(\ven)
\quad\leadsto\quad0<-b_i(p)=\lambda\vec{b}_i(\ven)
$$
and $\lambda>0$. Thus, $z\in\Int(\vspanp{p}{\ven})$, as claimed.

Notice that $\pi_n(\pol)=\bigcup_{i=1}^r\pi_n(\Ff_i)$. By \eqref{planos}
\begin{align*}
\pi_n(\Int(\pol))&=\pi_n(\pol)\setminus\partial\pi_n(\pol)
=\Big(\bigcup_{i=1}^r\pi_n(\Ff_i)\Big)
\setminus\partial\pi_n(\pol)\\
&=\bigcup_{i=1}^r(\pi_n(\Ff_i)\setminus\partial\pi_n(\pol))\subset
\bigcup_{i=1}^r\pi_n(B_i\cap\Int(\pol))\subset\pi_n(\Int(\pol)),
\end{align*}
so $\pi_n(\Int(\pol))=\bigcup_{i=1}^r\pi_n(B_i\cap\Int(\pol))$, as required.
\end{proof}

To illustrate the meaning of Lemma~\ref{allplanos0}, Figure~\ref{fig0} shows how the projection $\pi_3:\R^3\to\R^3$ acts on a polyhedron $\pol$ with two non-vertical facets $\Ff_1$, $\Ff_2$. These facets are separated from $\Ff$ by the hyperplanes $B_1$, $B_2$.

\subsection{Nonvertical hyperplanes and polynomial functions}

In many of our arguments non-vertical affine subspaces play a special role because of the way we place our polyhedra in space. If we consider a finite collection of non-vertical hyperplanes, it is intuitively clear that we can find a polynomial function $G\in\R[\x']:=\R[\x_1,\dots,\x_{n-1}]$ whose graph $\{\x_n=G\}$ lies `above' all these hyperplanes. In fact, we can say more.

\begin{prop}\label{polq2}
Let $\{H_i\}_{i=1}^k$ be a finite family of (non-vertical) hyperplanes with linear equations $H_i:=\{h_i=0\}$ oriented so that $\vec{h}_i(\ven)>0$. Then there exists $G\in\R[\x']$ such that $G>1$ on $\R^{n-1}$ and its graph $\Lambda:=\{\x_n=G\}\subset\R^n$ satisfies 
\begin{equation}\label{lven}
\vspanp{\Lambda}{\ven}=\{\x_n\geq G\}\subset\bigcap_{i=1}^k\{h_i>1\}.
\end{equation}
In particular, $H_i\cap\vspanp{\Lambda}{\ven}=\varnothing$ for $i=1,\dots,k$ and $\vspanp{\Lambda}{\ven}\subset\{\prod_{i=1}^kh_i>1\}$.
\end{prop}
\begin{proof}
Write $h_i(\x',\x_n)=\qq{a_i'}{\x'}+a_{in}\x_n+b_i$ where $a_i'\in\R^{n-1}$, $a_{in},b_i\in\R$ and $a_{in}=\vec{h}_i(\ven)>0$. Denote
$$
\rho_i(\x'):=-\frac{1}{a_{in}}(\qq{a_i'}{\x'}+b_i)
$$
and observe that $H_i=\{\x_n-\rho_i(\x')=0\}$. For each $i=1,\dots,k$ consider the polynomial
$$
G_i:=1+\frac{1}{a_{in}}+\frac{\rho^2_i+1}{2}\in\R[\x'].
$$
We have $G_i(x')\geq1+\frac{1}{a_{in}}+|\rho_i(x')|>1$ for each $x'\in\R^{n-1}$. Define $G:=\prod_{i=1}^kG_i\in\R[\x']$. It holds $G(x')\ge G_i(x')>1$ for each $i=1,\dots,k$ and $x'\in\R^{n-1}$. Let us check \eqref{lven}.

Pick $(x',x_n)\in\R^n$ such that $x_n\ge G(x)$. Then
\begin{align*}
h_i(x',x_n)&=\qq{a_i'}{x'}+a_{in}x_n+b_i\ge a_{in}x_n-|\qq{a_i'}{x'}+b_i|\ge a_{in}G(x)-|\qq{a_i'}{x'}+b_i|\\
&\ge a_{in}G_i(x)-|\qq{a_i'}{x'}+b_i|\geq a_{in}\big(1+\tfrac{1}{a_{in}}+|\rho_i(x')|\big)-|\qq{a_i'}{x'}+b_i|>1.
\end{align*}
Consequently, $\vspanp{\Lambda}{\ven}\subset\bigcap_{i=1}^k\{h_i>1\}\subset\{\prod_{i=1}^kh_i>1\}$ and $H_i\cap\vspanp{\Lambda}{\ven}=\varnothing$ for $i=1,\ldots,k$, as required.
\end{proof}
\begin{remark}
By including an extra hyperplane $H_0$ of equation $\x_n-b=0$ where $b\in\R$ we can find a corresponding polynomial $G(\x')$ satisfying the previous statement and such that $\{\x_n\ge G\}$ lies in $\{\x_n>b+1\}$.
\end{remark}

In order to construct polynomial maps $f:\R^n\to\R^n$ with polyhedral images we will resort to maps fixing pointwise finite collections of hyperplanes in $\R^n$. These maps will often leave vertical lines invariant. Under these hypotheses the following immediate but useful application of Bolzano's Theorem applies. Given a function $\psi:\R\to\R$ we write $\psi(\pm\infty):=\lim_{t\to\pm\infty}\psi(t)$ whenever the previous limit either exists or is equal to $\pm\infty$. 

\begin{lem}\label{fun}
Let $\psi:\R\to\R$ be a continuous function and let $-\infty< a< b\leq +\infty$ be such that $\psi(a)=a$ and $\psi(b)=b$. Then ${]a,b[}\subset\psi({]a,b[})$.
\end{lem}
\begin{cor}\label{fun2}
Let $f:=(f',f_n):\R^n\to\R^n$ be a continuous map and let $x:=(x',x_n)\in\R^n$ be such that $f(\vspan{x}{\ven})\subset\vspan{x}{\ven}$. Then
\begin{itemize}
\item[(i)] For each pair of points $p_1,p_2\in \vspan{x}{\ven}$ with $f(p_i)=p_i$, it holds $\Int(\overline{p_1p_2})\subset f(\Int(\overline{p_1p_2}))$.
\item[(ii)] Assume that $\psi_{x'}(t):=f_n(x',t)$ satisfies $\psi_{x'}(+\infty)=+\infty$. For each $p\in\vspan{x}{\ven}$ such that $f(p)=p$ we have $\Int(\vspanp{p}{\ven})\subset f(\Int(\vspanp{p}{\ven}))$.
\end{itemize}
\end{cor}

\section{Variations on Pecker's polynomials}\label{s2}

One main result in this section and the key to prove the main results of this article is Lemma~\ref{polq1}. In \ref{s3d} we present some of its consequences that will help us to establish a link between Pecker's results and Theorems \ref{main1} and \ref{main2}. Denote $x':=(x_1,\ldots,x_{n-1})$ so that each point $x:=(x_1,\dots,x_n)\in\R^n\equiv\R^{n-1}\times\R$ will be written in this section as $x:=(x',x_n)$. As before, $\R[\x']:=\R[\x_1,\ldots,\x_{n-1}]$, $\R[\x]:=\R[\x_1,\ldots,\x_n]$ and $\pi_n:\R^n\to\R^n,\ x:=(x',x_n)\mapsto(x',0)$.

\begin{defines}\label{adtup}
A tuple $\setg{g}:=(g_1,\ldots,g_m, g_{m+1})\in\R[\x']^{m+1}$ such that $g_{m+1}>1$ on $\R^{n-1}$ is called an \em admissible tuple of polynomials of length $m+1$\em. We associate to $\setg{g}$ the semialgebraic set $\seta{\setg{g}}:=\{g_1>0,\ldots,g_m>0,\x_n=0\}\subset\R^n$, which does not depend on $g_{m+1}$.
\end{defines}

\begin{lem}\label{polq1}
Let $\setg{g}:=(g_1,\ldots,g_m, g_{m+1})\in\R[\x']^{m+1}$ be an admissible tuple of length $m+1$. Then there exists a polynomial $\polq{\setg{g}}\in\R[\x]$ such that:
\begin{itemize}
\item[(i)] $\{\polq{\setg{g}}\leq0\}\subset
\vspan{\seta{\setg{g}}}{\ven}\cap\{|\x_n|>\max\{g_{m+1},\frac{g_{m+1}}{\sqrt{g_1\cdots g_m}}\}\}$.
\item[(ii)] For each $(x',0)\in\seta{\setg{g}}$ there exist a positive root $r\geq g_{m+1}(x')$ of the univariate polynomial $\polq{\setg{g}}(x',\t)$ and a value $t\geq r$ such that $\polq{\setg{g}}(x',t)=-1$.
\item[(iii)] The set $\sets{\setg{g}}:=\vspan{\{\polq{\setg{g}}\le 0, \x_n>0\}}{\ven}$ satisfies $\pi_n(\sets{\setg{g}})=\seta{\setg{g}}$. In addition, for each $(x',x_n)\in\sets{\setg{g}}$ there exist $r_n>0$ and $t_n\geq0$ such that $x_n=r_n+t_n$ and $\polq{\setg{g}}(x',r_n)=0$.
\end{itemize}
\end{lem}

\begin{figure}[!ht]
\begin{center}
\begin{tikzpicture}[xscale=2]

\draw[draw=none,fill=gray!20,opacity=1] (2,1) -- (5,1) -- (5,7) -- (2,7) -- (2,1);
\draw[line width=1pt] (2,1) -- (5,1);
\draw[line width=0.75pt] (2.01,7) .. controls (2.5,1) and (4.5,1) .. (4.99,7);
\draw[line width=1pt,fill=gray!80,opacity=0.5] (2.02,7) .. controls (2.5,3) and (4.5,3.5) .. (4.98,7);
\draw[line width=1pt] (2.05,7) .. controls (2.5,4.5) and (4.5,4.5) .. (4.95,7);
\draw[line width=0.75pt,fill=gray!20] (2.10,7) .. controls (2.5,6.5) and (4.5,6.5) .. (4.90,7);
\draw[line width=0.75pt] (1,4) parabola bend (5,3) (6,3.1);

\draw[dashed] (2,1) -- (2,7);
\draw[dashed] (5,1) -- (5,7);
\draw[dashed] (4.375,1) -- (4.375,7);
\draw[line width=0.75pt] (1,2) -- (6,2);

\draw[fill=white,draw,xscale=0.5] (4,1) circle (0.75mm);
\draw[fill=white,draw,xscale=0.5] (10,1) circle (0.75mm);
\draw[fill=black,draw,xscale=0.5] (8.75,1) circle (0.5mm);
\draw[fill=black,draw,xscale=0.5] (8.75,5) circle (0.5mm);
\draw[fill=black,draw,xscale=0.5] (8.75,5.65) circle (0.5mm);
\draw(3.5,0.5) node{\small$\seta{\setg{g}}$};
\draw(5.45,1.7) node{\small${\tt x}_n=1$};
\draw(5.6,2.7) node{\small${\tt x}_n=g_{m+1}$};
\draw(5.3,4.7) node{\small${\tt x}_n=\frac{g_{m+1}}{\sqrt{g_1\cdots g_m}}$};
\draw(3.5,6.1) node{\small$\polq{\setg{g}}\leq0$};
\draw(3.5,5.4) node{\tiny$\polq{\setg{g}}=-1$};
\draw(3.5,4.4) node{\tiny$\polq{\setg{g}}=0$};
\draw(4.15,5.1) node{\tiny$(x',r)$};
\draw(4.15,5.75) node{\tiny$(x',t)$};
\draw(4.15,0.75) node{\tiny$(x',0)$};

\end{tikzpicture}
\end{center}
\caption{$\{\polq{\setg{g}}\leq0\}\subset\vspan{\seta{\setg{g}}}{\ven}\cap\{{\tt x}_n>\max\{g_{m+1},\frac{g_{m+1}}{\sqrt{g_1\cdots g_m}}\}\}$}\label{fig1}
\end{figure}

Figure~\ref{fig1} sketches the graphical meaning of Lemma~\ref{polq1}. Its proof relies on Pecker's construction \cite[\S2]{p} that we recall next.

\subsection{Pecker's construction} 
Define:
\begin{equation}\label{ak}
a_k(\y_1,\dots,\y_{k+1}):=\y_{k+1}(\y_1+\cdots+\y_k)\in\Z[\y_1,\ldots,\y_{k+1}].
\end{equation}
If $y_i\geq0$ for $i=1,\ldots,k+1$, it holds $a_k(y_1,\ldots,y_{k+1})\geq0$. Consider Pecker's polynomials defined as follows:
\begin{align*}
P_1(\y_1,{\tt t})&:={\tt t}-\y_1,\\
P_{m+1}(\y_1,\dots,\y_{m+1},{\tt t})&:=P_m(a_1(\y_1,\y_2),\dots,a_m(\y_1,\dots,\y_{m+1}),({\tt t}-(\y_1+\cdots+\y_{m+1}))^2).
\end{align*}

\subsubsection{Basic properties of polynomials $P_m$}\label{pecker}
The previous polynomials satisfy the following properties \cite[Thm.1]{p}:
\begin{itemize}\em
\item[(i)] $P_m\in\Z[\y_1,\ldots,\y_m,\t]$ is a homogeneous polynomial of degree $2^{m-1}$.
\item[(ii)] If each $y_i\geq0$ and $P_m(y_1,\dots,y_m,t)=0$, then $0\le t\le 2\sum_{i=1}^m y_i$.
\item[(iii)] If all the $y_i$ are non-negative, the polynomial $P_m(y_1,\dots,y_m,{\tt t}^2)$ in the variable ${\tt t}$ has only real roots. 
\item[(iv)] If $P_m(y_1,\dots,y_m,{\tt t}^2)$ has a real root, then all the $y_i$ are non-negative.
\item[(v)] $P_m(\y_1,\dots,\y_{j-1},0,\y_{j+1},\dots,\y_m,{\tt t})=(P_{m-1}(\y_1,\dots,\y_{j-1},\y_{j+1},\dots,\y_m,{\tt t}))^2$.
\item[(vi)] $P_m(\y_1,\dots,\y_m,{\tt t}^2)$ is irreducible in $\R[\y_1,\ldots,\y_m,\t]$ and monic in each variable.
\end{itemize}

\subsubsection{Further properties of polynomials $P_m$}\label{peck1}
The polynomials $P_m$ satisfy in addition the following properties:
\begin{itemize}\em
\item[(i)] If each $y_i\geq0$ and $P_m(y_1,\dots,y_m,t)=0$, then 
$$
\Big(1-\sqrt{\tfrac{m-1}{m}}\Big)\sum_{i=1}^m y_i\le t\le\Big(1+\sqrt{\tfrac{m-1}{m}}\Big)\sum_{i=1}^m y_i.
$$
\item[(ii)] If $m\geq2$ and each $y_i>0$, then $P_m(y_1,\dots,y_m,0)>0$.
\item[(iii)] Define $A_1(\y_1):=\y_1$ and 
$$
A_{k+1}(\y_1,\ldots,\y_{k+1}):=A_k(a_1(\y_1,\y_2),\ldots,a_k(\y_1,\ldots,\y_{k+1}))\in\Z[\y_1,\ldots,\y_{k+1}].
$$
Then $A_m$ is a homogeneous polynomial of degree $2^{m-1}$ and there exists a homogeneous polynomial $B_{m-1}\in\Z[\y_1,\ldots,\y_{m-1}]$ of degree $2^{m-1}-m$ with non-negative coefficients such that $A_m=B_{m-1}\prod_{i=1}^m\y_i$. 
\item[(iv)] Given values $y_i\geq0$ for $i=1,\ldots,m$, there exists
$$
t_m\begin{cases}
=0&\text{if $m=1$,}\\
\ge\sum_{i=1}^my_i&\text{if $m\geq2$,}
\end{cases}
$$
such that $P_m(y_1,\ldots,y_m,t_m)=-A_m(y_1,\ldots,y_m)<0$.
\end{itemize}
\begin{proof}
(i) It $m=1$ the result is clearly true, so let us assume $m\geq2$. Denote $s_m:=\sum_{i=1}^my_i$, $r_m:=\sum_{1\le j<k\le m}y_jy_k$ and $q_m:=\sum_{i=1}^my^2_i$. As $s_m^2=q_m+2r_m$, notice that
\begin{align*}
q_m=\sum_{i=1}^m y_i^2=\frac{2}{m-1}\sum_{1\le j<k\le m}\frac{y_j^2+y_k^2}{2}\ge\frac{2}{m-1}\sum_{1\le j<k\le m}y_j y_k=\frac{2r_m}{m-1}.
\end{align*}
Consequently,
\begin{equation}\label{eq1}
s_m^2=q_m+2r_m\ge\frac{2m}{m-1}r_m\quad\leadsto\quad 2r_m\le\frac{m-1}{m}s_m^2.
\end{equation}
Let $t_0\in\R$ be such that $P_m(y_1,\dots,y_m,t_0)=0$. By definition
$$
P_m(y_1,\dots,y_m,\t)=P_{m-1}(a_1(y_1,y_2),\dots,a_{m-1}(y_1,\cdots,y_m),(\t-(y_1+\cdots+y_m))^2),
$$
so $u_0:=(t_0-s_m)^2$ is a root of the polynomial $P_{m-1}(a_1(y_1,y_2),\dots,a_{m-1}(y_1,\dots,y_m),{\tt u})$. As each $a_i(y_1,\dots,y_{i+1})\ge 0$ (see \eqref{ak}), we deduce from \ref{pecker}(ii)
\begin{equation*}
u_0=(t_0-s_m)^2\le 2\sum_{i=1}^{m-1}a_i(y_1,\dots,y_{i+1})=2\sum_{i=1}^{m-1}y_{i+1}(y_1+\cdots+y_i)=2\sum_{1\le j< k\le m} y_jy_k=2 r_m,
\end{equation*}
or equivalently,
$$
t_0^2-2s_mt_0+q_m\le 0.
$$
The previous condition is equivalent to
\begin{equation}\label{br}
s_m-\sqrt{2r_m}=s_m-\sqrt{s_m^2-q_m}\leq t_0\le s_m+\sqrt{s_m^2-q_m}=s_m+\sqrt{2r_m}.
\end{equation}
By \eqref{eq1} we have $\sqrt{2r_m}\leq\sqrt{\tfrac{m-1}{m}}s_m$ and by \eqref{br}
$$
\Big(1-\sqrt{\frac{m-1}{m}}\Big)s_m\leq s_m-\sqrt{2r_m}\leq t_0\le s_m+\sqrt{2r_m}\le\Big(1+\sqrt{\frac{m-1}{m}}\Big)s_m, 
$$
so the statement follows.

(ii) By (i) the polynomial $P_m(y_1,\dots,y_m,\t)\in\R[\t]$ has no real roots in the interval ${]-\infty,0]}$. By \ref{pecker}(i) \& (vi) $P_m(y_1,\dots,y_m,\t)$ is a monic polynomial of even degree, so
$$
\lim_{t\to-\infty}P_m(y_1,\dots,y_m,t)=+\infty.
$$
Consequently, $P_m(y_1,\dots,y_m,0)>0$.

(iii) We proceed by induction on $m$. For $A_1(\y_1)=\y_1$ and $A_2(\y_1,\y_2)=\y_1\y_2$ the statement is true by setting $B_0=B_1=1$. Assume the statement true for $m$. Then
\begin{equation*}
\begin{split}
A_{m+1}&(\y_1,\dots,\y_{m+1})=A_m(a_1(\y_1,\y_2),\dots,a_{m-1}(\y_1,\dots,\y_m),a_m(\y_1,\dots,\y_{m+1}))\\
&=\Big(\prod_{k=1}^ma_k(\y_1,\ldots,\y_{k+1})\Big)B_{m-1}(a_1(\y_1,\y_2),\dots,a_{m-1}(\y_1,\dots,\y_m))\\
&=\Big(\prod_{k=1}^{m+1}\y_k\Big)\Big(\prod_{k=2}^m(\y_1+\cdots+\y_k)\Big)
B_{m-1}(a_1(\y_1,\y_2),\dots,a_{m-1}(\y_1,\dots, \y_m))\\
&=\Big(\prod_{k=1}^{m+1}\y_k\Big)B_m(\y_1,\dots,\y_m)
\end{split}
\end{equation*}
where 
$$
B_m(\y_1,\dots,\y_m):=\Big(\prod_{k=2}^m(\y_1+\cdots+\y_k)\Big)B_{m-1}(a_1(\y_1,\y_2),\dots,a_{m-1}(\y_1,\dots,\y_m)).
$$
In addition, $A_{m+1}$ is by induction a homogeneous polynomial of degree $2^m$ (because $A_m$ is a homogeneous polynomial of degree $2^{m-1}$ and each $a_i$ is a homogeneous polynomial of degree $2$) and the equality 
$$
A_{m+1}(\y_1,\dots,\y_{m+1})=\y_1\cdots \y_{m+1}B_m(\y_1,\dots,\y_m)
$$
shows that $B_m(\y_1,\dots,\y_m)$ is a homogeneous polynomial of degree $2^m-(m+1)$. Besides $B_m\in\Z[\y_1,\ldots,\y_m]$ has non-negative coefficients by induction hypothesis because each $a_i\in\Z[\y_1,\ldots,\y_{i+1}]$ has non-negative coefficients.

(iv) We work by induction on $m$. For $m=1$ the polynomial $P_1(y_1,\t)=\t-y_1$ achieves the value $-A_1(y_1)=-y_1$ for $t_1:=0$. For $m=2$
$$
P_2(y_1,y_2,\t)=(\t-y_1-y_2)^2-y_1 y_2,
$$
and this polynomial attains the value $-A_2(y_1,y_2)=-A_1(a_1(y_1,y_2))=-y_1y_2$ for $t_2:=y_1+y_2$. 

Given $y_i\geq0$ for $i=1,\ldots,m$, consider the non-negative values 
$$
a_1(y_1,y_2),\ldots,a_m(y_1,\ldots,y_{m+1}). 
$$
Suppose by induction that there exists a real number 
$$
t_m'\geq\sum_{i=1}^ma_i(y_1,\dots,y_{i+1})
$$
such that 
\begin{equation}\label{claro}
P_m(a_1(y_1,y_2),\dots,a_m(y_1,\dots,y_{m+1}),t_m')=-A_m(a_1(y_1,y_2),\dots,a_m(y_1,\dots,y_{m+1})).
\end{equation} 
In particular, $t_m'\geq0$ and
$$
t_{m+1}:=\sqrt{t_m'}+\sum_{i=1}^{m+1}y_i\geq\sum_{i=1}^{m+1}y_i.
$$
Using the definition of $P_{m+1}$ and \eqref{claro} we have 
\begin{equation*}
\begin{split}
P_{m+1}(y_1,\dots,y_m,y_{m+1},t_{m+1})&
=P_m\big(a_1(y_1,y_2),\dots,a_m(y_1,\dots,y_{m+1}),\big(t_{m+1}-\textstyle\sum_{i=1}^{m+1}y_i\big)^2\big)\\
&=P_m\big(a_1(y_1,y_2),\dots,a_m(y_1,\dots,y_{m+1}),t_m'\big)\\
&=-A_m(a_1(y_1,y_2),\dots,a_m(y_1,\dots,y_{m+1}))\\
&=-A_{m+1}(y_1,\dots,y_m),
\end{split}
\end{equation*}
as required.
\end{proof}

\subsection{Modified Pecker's polynomials}\label{peck3}
Fix $m\ge2$ and denote 
$$
C_m:=\Big(1-\sqrt{\frac{m-1}{m}}\Big)\quad\text{and}\quad\expn(m):=(m+1)2^{m-1}-m^2+m. 
$$
Consider the polynomial 
\begin{equation}\label{qm}
Q_m(\y_1,\dots,\y_m,{\tt t}):=\Big(\frac{{\tt t}}{C_m}+\y_m\Big)^{\expn(m)} \hspace{-0.2cm}(\y_1\cdots \y_m)^{2^{m-1}}P_m\Big(\y_1,\dots,\y_{m-1},\frac{1}{\y_1\cdots \y_m},{\tt t}\Big).
\end{equation}
Then:
\begin{itemize}\em
\item[(i)] $Q_m(y_1,\dots,y_m,t^2)\geq0$ if some $y_i\le 0$. In addition, 
$$
\{Q_m(\y,\t^2)\leq0,\y_m=1\}\subset\{\y_1>0,\dots,\y_{m-1}>0,\y_m=1\}.
$$
\item[(ii)] If each $y_i>0$, the polynomial $Q_m(y_1,\dots,y_m,{\tt t})$ achieves the value $-1$ at some
$$
t\ge y_1+\cdots+y_{m-1}+\frac{1}{y_1\cdots y_m}.
$$
\end{itemize}
\begin{proof}
(i) As $\expn(m)$ is an even positive integer, the first factor of $Q_m$ is non-negative. By \ref{pecker} (i) and (vi)
$$
F_m(\y_1,\dots,\y_m,\t)=(\y_1\cdots \y_m)^{2^{m-1}}P_m\Big(\y_1,\dots,\y_{m-1},\frac{1}{\y_1\cdots \y_m},\t\Big)
$$
is a polynomial of degree $(m+1)2^{m-1}$. Consider the projection $\pi_{m+1}:\R^{m+1}\to\R^m,\ (y,t)\to y$. By \cite[Cor.1, p.308]{p} the hypersurface $\{F_m(\y,\t^2)=0\}\subset\R^{m+1}$ projects under $\pi_{m+1}$ onto the open orthant $\Qq:=\{\y_1>0,\dots,\y_m>0\}$. Thus, for each $t\in\R$ the polynomial $F_m(\y,t^2)$ has empty zero-set on $\R^{m}\setminus\Qq$. As $\R^m\setminus \Qq$ is connected and the origin $0\in\R^m\setminus\Qq$, we deduce $F_m(y,t^2)\cdot F_m(0,t^2)>0$ for every $y\in\R^m\setminus \Qq$. By \ref{pecker}(i) \& (vi) the polynomial $P_m$ is homogeneous and monic in each variable, so $F_m(0,t^2)=1>0$. Thus, $F_m(\y,t^2)>0$ on $\R^m\setminus\Qq$ and the first part of the statement follows.

If $y_m=1$, the first factor of $Q_m(\y,\t^2)$ is strictly positive. We have proved above that the factor $F_m$ is strictly positive on $(\R^m\setminus\Qq)\times\R$. Consequently,
$$
\{Q_m(\y,\t^2)\leq0,\y_m=1\}\subset\{F_m(\y,\t^2)\leq0,\y_m=1\}\subset\{\y_1>0,\dots,\y_{m-1}>0,\y_m=1\}.
$$

(ii) Fix $y:=(y_1,\ldots,y_m)\in\Qq:=\{\y_1>0,\ldots,\y_m>0\}$. By \ref{peck1}(iv) there exists
\begin{equation}\label{tm}
t_m\ge y_1+\cdots+y_{m-1}+\frac{1}{y_1\cdots y_m}
\end{equation}
such that
\begin{equation}\label{am}
P_m(y_1,\ldots,y_{m-1},\tfrac{1}{y_1\cdots y_m},t_m)=-A_m(y_1,\dots,y_{m-1},\tfrac{1}{y_1\cdots y_m}).
\end{equation}
By \ref{peck1}(iii) we can write
\begin{equation}\label{bm}
A_m(\y_1,\dots,\y_{m-1},\tfrac{1}{\y_1\cdots\y_m})=\frac{B_{m-1}(\y_1,\dots,\y_{m-1})}{\y_m}
\end{equation}
where $B_{m-1}\in\Z[\y_1,\ldots,\y_{m-1}]$ is a homogeneous polynomial of degree $2^{m-1}-m$ with non-negative coefficients. Consider the rational functions
\begin{align*}
&Q_{m,1}(\y_1,\dots,\y_m):=\Big(\y_1+\cdots+\y_m+\dfrac{1}{\y_1\cdots\y_m}\Big)^{2^{m-1}}(\y_1\cdots \y_m)^{2^{m-1}}\\
&Q_{m,2}(\y_1,\dots,\y_m):=\Big(\y_1+\cdots+\y_m+\dfrac{1}{\y_1\cdots\y_m}\Big)^{m(2^{m-1}-m)}B_{m-1}(\y_1,\dots,\y_{m-1})\\
&Q_{m,3}(\y_1,\dots,\y_m):=\Big(\y_1+\cdots+\y_m+\dfrac{1}{\y_1\cdots\y_m}\Big)^m{\dfrac{1}{\y_m}}
\end{align*}
We claim: \em $Q_{m,1}(y)>1$, $Q_{m,2}(y)\geq1$ and $Q_{m,3}(y)>1$\em.

The inequalities $Q_{m,1}(y)>1$ and $Q_{m,3}(y)>1$ are straightforward. We proceed with $Q_{m,2}(y)\geq1$. As $B_{m-1}$ is a homogeneous polynomial of degree $2^{m-1}-m$ whose coefficients are non-negative integers, we write
$$
B_{m-1}(\y_1,\ldots,\y_{m-1})=\sum_{|\nu|=2^{m-1}-m}a_\nu\y_1^{\nu_1}\cdots\y_{m-1}^{\nu_{m-1}}
$$
where $\nu:=(\nu_1,\ldots,\nu_{m-1})\in(\N\cup\{\orig\})^{m-1}$, $|\nu|=\nu_1+\cdots+\nu_{m-1}$ and $a_\nu\in\N\cup\{\orig\}$.

Fix $a_\nu\neq0$. By Lemma~\ref{ineq}(iii) and since $\nu_1+\cdots+\nu_{m-1}=2^{m-1}-m$, we have
\begin{align*}
\Big(y_1+\cdots+y_m+\dfrac{1}{y_1\cdots y_m}&\Big)^{m(2^{m-1}-m)}a_\nu y_1^{\nu_1}\cdots y_{m-1}^{\nu_{m-1}}\\
&=a_\nu\prod_{i=1}^{m-1}
\Big(\Big(y_1+\cdots+y_m+\dfrac{1}{y_1\cdots y_m}\Big)^m y_i\Big)^{\nu_i}\ge a_\nu\ge 1.
\end{align*}
Consequently, $Q_{m,2}(y)\ge 1$, as claimed.

By \eqref{tm} we have
$$
\frac{t_m}{C_m}+y_m\ge t_m+y_m\ge y_1+\cdots+y_m+\frac{1}{y_1\cdots y_m}.
$$
Therefore, by \eqref{am} and \eqref{bm}
\begin{align*}
Q_m(y,t_m)&=\Big(\frac{t_m}{C_m}+y_m\Big)^{\expn(m)}(y_1\cdots y_m)^{2^{m-1}}
P_m\big(y_1,\dots,y_{m-1},\tfrac{1}{y_1\cdots y_m},t_m\big)\\ 
&\le -Q_{m,1}(y) Q_{m,2}(y) Q_{m,3}(y)\le -1.
\end{align*}
By \ref{pecker}(vi) $\lim_{t\to+\infty}Q_m(y,t)\to+\infty$. Thus, there exists $t \ge t_m$ for which $Q_m(y,t)=-1$, as required.
\vspace{1mm}
\end{proof}

\subsection{Proof of Lemma~\ref{polq1}}
Consider the polynomial 
\begin{equation}\label{qg}
\polq{\setg{g}}(\x',\x_n):=g_{m+1}^{\expn_0}(\x')Q_{m+1}\Big(g_1(\x'),\dots,g_m(\x'),1,\frac{\x_n^2C_{m+1}}{g_{m+1}^2(\x')}\Big)
\end{equation}
where $Q_{m+1}$ is the polynomial constructed in \ref{peck3}, $C_{m+1}:=1-\sqrt{\frac{m}{m+1}}$ and $\expn_0:=2\expn(m+1)+2^{m+1}$ is large enough to guarantee that $\polq{\setg{g}}$ is a polynomial. 

(i) We have to show 
$$
\{\polq{\setg{g}}\leq0\}\subset\vspan{\seta{\setg{g}}}{\ven}\cap\Big\{|\x_n|>\max\Big\{g_{m+1},\frac{g_{m+1}}{\sqrt{g_1\cdots g_m}}\Big\}\Big\}.
$$

By \ref{peck3}(i) 
$$
\{\polq{\setg{g}}\leq0\}=\Big\{Q_{m+1}\Big(g_1,\ldots,g_m,1,\frac{\x_n^2C_{m+1}}{g_{m+1}^2}\Big)\leq0\Big\}\subset\{g_1>0,\ldots,g_m>0\}=\vspan{\seta{\setg{g}}}{\ven}. 
$$
We check now 
\begin{equation}\label{req1}
\{\polq{\setg{g}}\leq0\}\subset\Big\{|\x_n|>\max\Big\{g_{m+1},\frac{g_{m+1}}{\sqrt{g_1\cdots g_m}}\Big\}\Big\}.
\end{equation}

Fix $(x',0)\in\seta{\setg{g}}$. As $\polq{\setg{g}}(x',\x_n)=\polq{\setg{g}}(x',-\x_n)$ and the leading coefficient of $\polq{\setg{g}}(x',\x_n)$ with respect to $\x_n$ is positive, $\lim_{x_n\to\pm\infty}\polq{\setg{g}}(x',x_n)=+\infty$. 

By \ref{pecker}(iii) the univariate polynomial $\polq{\setg{g}}(x',\x_n)$ has $2^m$ real roots. As it defines an even polynomial function and by \ref{peck1}(i) none of its roots is zero, $2^{m-1}$ of them are positive and $2^{m-1}$ are negative. Let $r>0$ be the smallest of the positive roots of $\polq{\setg{g}}(x',\x_n)$. We have 
$$
P_{m+1}\Big(g_1(x'),\dots,g_m(x'),\frac{1}{g_1(x')\cdots g_m(x')},\frac{r^2C_{m+1}}{g_{m+1}^2(x')}\Big)=0
$$
and each $g_i(x')>0$. By \ref{peck1}(i) and Lemma~\ref{ineq}(i)
$$
\frac{r^2C_{m+1}}{g_{m+1}^2(x')}\ge C_{m+1}\Big(\sum_{i=1}^m g_i(x')+\frac{1}{g_1(x')\cdots g_m(x')}\Big)\ge C_{m+1}>0.
$$
Thus, $r^2\ge g_{m+1}^2(x')$, so $r\ge g_{m+1}(x')$. In addition,
$$
\frac{r^2C_{m+1}}{g_{m+1}^2(x')}\ge\frac{C_{m+1}}{g_1(x')\cdots g_m(x')}\quad\leadsto\quad r\ge\frac{g_{m+1}(x')}{\sqrt{g_1(x')\cdots g_m(x')}}.
$$
Consequently,
\begin{equation}\label{r}
r\ge\max\Big\{g_{m+1}(x'),\frac{g_{m+1}(x')}{\sqrt{g_1(x')\cdots g_m(x')}}\Big\}.
\end{equation}

\subsubsection{}\label{refi1}We claim: \em if $\polq{\setg{g}}(x',x_n)\leq0$ and $x_n\ge 0$, then $x_n\geq r$\em. As $r$ is the smallest positive real root of the univariate polynomial $\polq{\setg{g}}(x',\x_n)$, it is enough to show: $\polq{\setg{g}}(x',0)>0$. 

This follows from \ref{peck1}(ii) because $g_i(x')>0$ for $i=1,\dots,m$ and 
\begin{equation*}
\begin{split}
\polq{\setg{g}}(x',0)&=g_{m+1}^{\expn_0}(x')Q_{m+1}(g_1(x'),\dots,g_m(x'),1,0)\\
&=g_{m+1}^{\expn_0}(x')(g_1(x')\cdots g_m(x'))^{2^m}P_{m+1}\Big(g_1(x'),\dots,g_m(x'),\frac{1}{g_1(x')\cdots g_m(x')},0\Big)>0.
\end{split}
\end{equation*}
By \eqref{r} and \ref{refi1} the inclussion \eqref{req1} holds.

(ii) If $(x',0)\in\seta{\setg{g}}$, we have $g_1(x')>0,\ldots,g_m(x')>0$. By \ref{peck3}(ii) there exists 
$$
t_0\geq g_1(x')+\cdots+g_m(x')+\frac{1}{g_1(x')\cdots g_m(x')}>0
$$
such that $Q_{m+1}(g_1(x'),\dots,g_m(x'),1,t_0)=-1$. Define $t_1:=\sqrt{\frac{t_0}{C_{m+1}}}g_{m+1}(x')$ and observe 
$$
\polq{\setg{g}}(x',t_1)=g_{m+1}^{\expn_0}(x')Q_{m+1}(g_1(x'),\dots,g_m(x'),1,t_0)=-g_{m+1}^{\expn_0}(x')<-1.
$$
As $\polq{\setg{g}}(x',t_1)<-1$, we know by \ref{refi1} that $t_1>r$. We have $\polq{\setg{g}}(x',t_1)<-1<\polq{\setg{g}}(x',r)=0$, so there exists $r\leq t\leq t_1$ such that $\polq{\setg{g}}(x',t)=-1$.

(iii) Statements (i) and (ii) provide the first part of (iii) whereas \ref{refi1} issues the second part of (iii), as required
\qed

\subsection{Consequences of Lemma~\ref{polq1}}\label{s3d}\setcounter{paragraph}{0}

Let $\setg{g}:=(g_1,\ldots,g_{m+1})\in\R[\x']^{m+1}$ be an admissible tuple of polynomials and let us consider the corresponding polynomial $\polq{\setg{g}}\in\R[\x]$ introduced in Lemma~\ref{polq1} and the associated semialgebraic sets $\seta{\setg{g}}$ and $\sets{\setg{g}}$. The latter semialgebraic set was introduced in Lemma~\ref{polq1}, where we also proved some key properties of $\sets{\setg{g}}$.

\begin{thm}\label{atico2b}
Let $h\in\R[\x']$ be positive semidefinite on $\seta{\setg{g}}$ and let $P\in\R[\x]$ be strictly greater than $1$ on $\sets{\setg{g}}$. Assume in addition $g_{m+1}>h$ on $\seta{\setg{g}}$. Define 
$$
f:=(f_1,\ldots,f_n):\R^n\to\R^n,\ x:=(x',x_n)\mapsto(x',x_n(1+P(x)\polq{\setg{g}}(x))^2+h(x')(P(x)\polq{\setg{g}}(x))^2).
$$
We have:
\begin{itemize}
\item[(i)] $\vspan{\seta{\setg{g}}}{\ven}\cap\{\x_n\geq h(\x')\}\subset f(\sets{\setg{g}})\subset\vspan{\seta{\setg{g}}}{\ven}\cap\{2\x_n\geq h(\x')\}$. In particular, if $h=0$, we have $f(\sets{\setg{g}})=\vspanp{\seta{\setg{g}}}{\ven}$.
\item[(ii)] Whenever $x_n\ge0$, $P(x)\polq{\setg{g}}(x)\ge 0$ and $h(x')\ge0$, the inequality $f_n(x)\ge x_n$ holds. In particular, this happens if $x\in(\vspanp{\seta{\setg{g}}}{\ven}\setminus\sets{\setg{g}})\cap\{P\geq0\}$.
\item[(iii)] If $h(x')\geq0$ and $P(x',x_n)\geq0$ for $x_n$ large enough, then $\lim_{x_n\to+\infty}f_n(x',x_n)=+\infty$.
\end{itemize}
\end{thm}

\begin{figure}[!ht]
\begin{center}
\begin{tikzpicture}[xscale=1]

\draw[dashed] (-5,1) -- (-5,7) (-2,7) -- (-2,1); 
\draw[draw=none,fill=gray!50,opacity=1] (-4.8,7) parabola bend (-3.5,5) (-2.2,7);
\draw[line width=1pt](-4.8,7) parabola bend (-3.5,5) (-2.2,7);
\draw(-3.5,6) node{\small$\sets{\setg{g}}$};
\draw(-0.5,4.3) node{\small$f$};
\draw[line width=1pt] (-5,1)--(-2,1);
\draw[fill=white,draw] (-5,1) circle (0.75mm);
\draw[fill=white,draw] (-2,1) circle (0.75mm);
\draw(-3.5,1.3) node{\small$\seta{\setg{g}}$};

\draw[line width=1pt] (1,1)--(4,1);

\draw[line width=1pt,->](-1,4)--(0,4);
\draw[draw=none,fill=gray!50,opacity=1] (1,2.25) parabola bend (1,2.25) (2.5,2.75) parabola bend (4,3.25) (4,3.25) -- (4,7) -- (1,7) -- (1,2.25);

\draw[draw=none,fill=gray!50,opacity=1] (6,1) -- (9,1) -- (9,7) -- (6,7) -- (6,1);
\draw[line width=1pt] (1,2.25) parabola bend (1,2.25) (2.5,2.75) parabola bend (4,3.25) (4,3.25);

\draw[line width=1pt] (6,1) -- (9,1);

\draw[line width=0.75pt] (0.5,1.8) -- (4.5,2.4);
\draw[line width=0.75pt] (0.5,2.5) -- (4.5,4);

\draw[dashed] (1,1) -- (1,7);
\draw[dashed] (4,1) -- (4,7);

\draw[dashed] (6,1) -- (6,7);
\draw[dashed] (9,1) -- (9,7);

\draw[fill=white,draw] (1,2.25) circle (0.75mm);
\draw[fill=white,draw] (4,3.25) circle (0.75mm);

\draw[fill=white,draw] (6,1) circle (0.75mm);
\draw[fill=white,draw] (9,1) circle (0.75mm);
\draw[fill=white,draw] (1,1) circle (0.75mm);
\draw[fill=white,draw] (4,1) circle (0.75mm);

\draw(2.5,0.5) node{(a)};
\draw(7.5,0.5) node{(b)};

\draw(2.5,1.3) node{\small$\seta{\setg{g}}$};
\draw(7.5,1.3) node{\small$\seta{\setg{g}}$};

\draw(2.5,6) node{\small$f(\sets{\setg{g}})$};
\draw(7.5,6) node{\small$f(\sets{\setg{g}})$};

\draw(2.5,3.5) node[rotate=20]{\small${\tt x}_n=h$};
\draw(3,2.4) node[rotate=8]{\small$2{\tt x}_n=h$};

\end{tikzpicture}
\end{center}
\caption{Theorem~\ref{atico2b}: (a) Case $h\neq0$.\quad (b) Case $h=0$.}\label{fig2}
\end{figure}
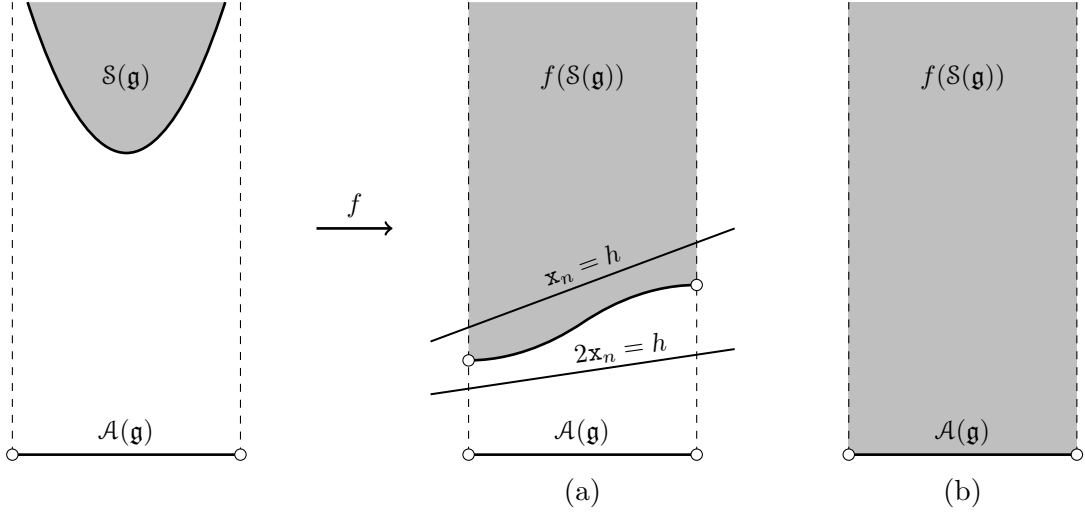

\begin{proof}
(i) The polynomial map $f$ preserves vertical lines. 

\subsubsection{}We prove first: $\vspan{\seta{\setg{g}}}{\ven}\cap\{\x_n\geq h(\x')\}\subset f(\sets{\setg{g}})$.

\subsubsection{}\label{degqg}Pick a point $(x',0)\in\seta{\setg{g}}$. We claim: \em the polynomial $\polq{\setg{g}}(x',\x_n)$ has degree $\ell_0:=2\ell(m+1)+2^{m+1}$ and its leading coefficient is strictly positive\em.

By \ref{pecker}(i) and (vi), \eqref{qm} and \eqref{qg} the degree of $\polq{\setg{g}}(x',\x_n)$ is $\ell_0$ and its leading coefficient is $(C_{m+1}g_1\cdots g_m)^{2^m}>0$.

\subsubsection{}By Lemma~\ref{polq1}(ii) there exist points $p:=(x',t)$ and $q:=(x',r)$ such that $t>r\geq g_{m+1}(x')>1$, $\polq{\setg{g}}(x',t)=-1$ and $\polq{\setg{g}}(x',r)=0$. In particular, $\vspanp{p}{\ven}\subset\sets{\setg{g}}$, so $P$ is strictly greater than $1$ on $\vspanp{p}{\ven}$. Consider the polynomial $\phi_{x'}(\x_n):=1+P(x',\x_n)\polq{\setg{g}}(x',\x_n)$. We have
$$
\phi_{x'}(r)=1,\quad\phi_{x'}(t)=1+P(x',t)Q_\setg{g}(x',t)<0\quad\text{and}\quad\lim_{x_n\to+\infty}\phi_{x'}(x_n)=+\infty.
$$
Consequently, there exists $s\in{]r,t[}$ such that $\phi_{x'}(s)=0$, so ${[0,+\infty[}\subset\phi_{x'}({[s,+\infty[})$. Consider also the polynomial 
$$
\varphi_{x'}(\x_n):=f_n(x',\x_n)=\x_n\phi_{x'}^2(\x_n)+h(x')(\phi_{x'}(\x_n)-1)^2
$$ 
and observe that 
$$
\varphi_{x'}(s)=h(x')\quad\text{and}\quad\lim_{x_n\to+\infty}\varphi_{x'}(x_n)=+\infty.
$$ 
Thus, ${[h(x'),+\infty[}\subset\varphi_{x'}({[s,+\infty[})$ and
$$
\{x'\}\times\{\x_n\geq h(x')\}\subset f(\vspanp{p}{\ven})\subset f(\sets{\setg{g}}).
$$
We conclude $\vspan{\seta{\setg{g}}}{\ven}\cap\{\x_n\geq h(\x')\}\subset f(\sets{\setg{g}})$.

\subsubsection{}Let us check next: $f(\sets{\setg{g}})\subset\vspan{\seta{\setg{g}}}
{\ven}\cap\{2\x_n\geq h(\x')\}$. 

Pick a point $(x',x_n)\in\sets{\setg{g}}$. By Lemma~\ref{polq1}(iii) $(x',0)\in\seta{\setg{g}}$ and $x_n>g_{m+1}(x')>h(x')$. Consider the polynomial $\psi_{x'}(\x_n):=P(x',\x_n)\polq{\setg{g}}(x',\x_n)$, so the last component of $f$ can be rewritten as $f_n(x',x_n)=\x_n(1+\psi_{x'}(\x_n))^2+h(x')\psi^2_{x'}(\x_n)$. As $x_n>h(x')\geq0$,
\begin{align*}
f_n(x',x_n)&=x_n+\psi_{x'}^2(x_n)(x_n+h(x'))+2\psi_{x'}(x_n) x_n\\
&=\Big(\sqrt{x_n+h(x')}\psi_{x'}(x_n)+\frac{x_n}{\sqrt{x_n+h(x')}}\Big)^2+x_n-\Big(
\frac{x_n}{\sqrt{x_n+h(x')}}\Big)^2\\
&\geq x_n-\Big(
\frac{x_n}{\sqrt{x_n+h(x')}}\Big)^2=\frac{x_nh(x')}{x_n+h(x')}=\frac{h(x')}{1+\frac{h(x')}{x_n}}\ge\frac{h(x')}{2}.
\end{align*}
Consequently, $f(\sets{\setg{g}})\subset\vspan{\seta{\setg{g}}}{\ven}\cap\{2\x_n\geq h(\x')\}$.

(ii) The statement follows from the required inequalities and the definition of the coordinate function $f_n(x)$.

(iii) Pick $x'\in\R^{n-1}$ such that $h(x')\ge 0$ and $P(x',x_n)\geq0$ for $x_n$ large enough. If $(x',0)\notin\seta{\setg{g}}$, then $\polq{\setg{g}}(x',\x_n)$ is positive on $\{\x_n>0\}$. By (ii) $f_n(x'x_n)\geq x_n$ if $x_n$ is large enough, hence
$$
\lim_{x_n\to+\infty}f_n(x',x_n)=+\infty.
$$
If $(x',0)\in\seta{\setg{g}}$, the polynomial $\polq{\setg{g}}(x',\x_n)$ has degree $\ell_0:=2\ell(m+1)+2^{m+1}$ and its leading coefficient is strictly positive (see \ref{degqg}). Consequently, $\lim_{x_n\to+\infty}f_n(x',x_n)=+\infty$, as required.
\end{proof}

Figure~\ref{fig2} illustrates the action of the polynomial map $f$ in Theorem~\ref{atico2b} on the semialgebraic set $\sets{\setg{g}}$.

\subsection{Lower dimensional semialgebraic sets}\label{lower}
Fix $1\leq d\leq n-2$ and write $\y:=(\x_1,\ldots,\x_d)$, $\z:=(\x_{d+1},\ldots,\x_{n-1})$ and $\x':=(\y,\z)$, so that $\x:=(\x_1,\ldots,\x_n)=(\x',\x_n)=(\y,\z,\x_n)$ and we identify $\R^n\equiv\R^d\times\R^{n-1-d}\times\R$. Let $g_1,\ldots,g_r\in\R[\y]$ and let $\veps>0$. Denote $I_\veps:={]{-\veps},\veps[}$ and set $m:=r+2(n-1-d)$. Given $g_{m+1}\in\R[\x']$ such that $g_{m+1}>1$ on $\R^{n-1}$, consider the admissible tuple
$$
\hat{\setg{g}}_\veps:=(g_1,\ldots,g_r,\x_{d+1}+\veps,\ldots,\x_{n-1}+\veps,\veps-\x_{d+1},\ldots,\veps-\x_{n-1}, g_{m+1})
$$
and the polynomial $\polq{\hat{\setg{g}}_\veps}$ constructed in Lemma~\ref{polq1}. Consider also the associated semialgebraic sets $\seta{\hat{\setg{g}}_\veps}$, $\sets{\hat{\setg{g}}_\veps}$ and 
$$
\setd{\hat{\setg{g}}_\veps}:=\seta{\hat{\setg{g}}_\veps}\cap\{\x_{d+1}=0,\ldots,\x_n=0\}=\{g_1>0,\ldots,g_r>0,\x_{d+1}=0,\ldots,\x_n=0\}\subset\R^n.
$$
Recall that $\vcon{\delta}:=\{(v',v_n)\in\R^n:\ \|v'\|\leq\delta v_n\}$ is the vertical cone of radius $\delta>0$ and given a set $T\subset\R^n$ the set $\acon{\delta}{T}:=T+\vcon{\delta}$ is the vertical cone of radius $\delta>0$ over $T$.

\begin{thm}\label{coneS}
Let $\veps,\delta>0$ and assume $g_{m+1}\geq 1+\veps\frac{\sqrt{n-d-1}}{\delta}$. Let $P\in\R[\x]$ be $>1$ on $\sets{\hat{\setg{g}}_\veps}$. For each $k\geq1$ consider the polynomial map
$$
f_k:\R^n\to\R^n,\ x:=(x',x_n)=(y,z,x_n)\mapsto
(y,A(x)z,B_k(x)x_n),
$$
where $A:=(1+P^2\polq{{\hat{\setg{g}}_\veps}})^2$ and $B_k:=\frac{A+A^k}{2}$. We have: 
\begin{itemize}
\item[(i)] $\lim_{x_n\to+\infty}B_k(x',x_n)x_n=+\infty$ for each $x'\in\R^{n-1}$.
\item[(ii)] $\sets{\hat{\setg{g}}_\veps}\subset\acon{\delta}{\setd{\hat{\setg{g}}_\veps}}.$
\item[(iii)] $\vspanp{\setd{\hat{\setg{g}}_\veps}}{\ven}\subset f_k(\sets{\hat{\setg{g}}_\veps})\subset\acon{2\delta}{\setd{\hat{\setg{g}}_\veps}}$.
\item[(iv)] For each $\Delta>2\delta$ there exists $k_0\geq1$ such that if $k\geq k_0$ and $x\in\acon{\Delta}{\setd{\hat{\setg{g}}_\veps}}\setminus\sets{\hat{\setg{g}}_\veps}$, then $f_k(x)\in\acon{2\delta}{\{x\}}$. 
\end{itemize}
\end{thm}

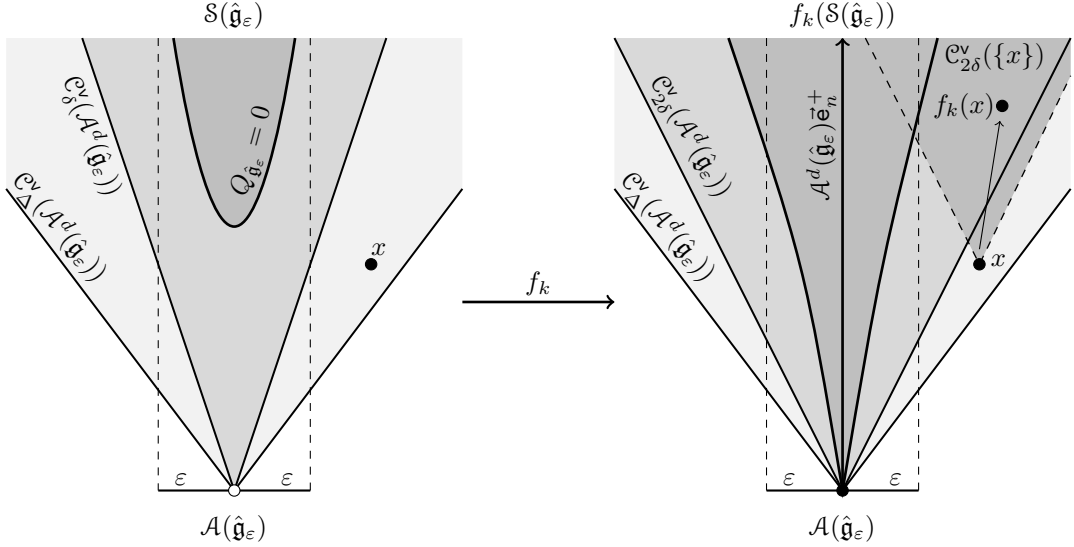
\begin{figure}[!ht]
\begin{center}
\begin{tikzpicture}[scale=1]

\draw[draw=none,fill=gray!10,opacity=1] (0,7) -- (0,5) -- (3,1) -- (6,5) -- (6,7) -- (0,7);
\draw[draw=none,fill=gray!30,opacity=1] (1,7) -- (3,1) -- (5,7) -- (1,7);
\draw[draw=none,fill=gray!50,opacity=1] (2.2,7) parabola bend (3,4.5) (3.8,7) -- (0,7);

\draw[dashed] (2,7) --(2,1) (4,1)--(4,7);
\draw[thick](2,1)--(4,1);

\draw[line width=0.75pt] (1,7) -- (3,1) -- (5,7);
\draw[line width=1pt] (2.2,7) parabola bend (3,4.5) (3.8,7);
\draw[line width=0.75pt] (0,5) -- (3,1) -- (6,5);

\draw[fill=white] (3,1) circle (0.75mm);
\draw[fill=black] (4.8,4) circle (0.75mm);
\draw[line width=1pt,->] (6,3.5) -- (8,3.5);

\draw[draw=none,fill=gray!10,opacity=1] (8,7) -- (8,5) -- (11,1) -- (14,5) -- (14,7) -- (8,7);
\draw[draw=none,fill=gray!30,opacity=1] (8,7) -- (11,1) -- (14,7) -- (8,7);
\draw[draw=none,fill=gray!50,opacity=1] (9.5,7) ..controls(10.5,4).. (11,1) ..controls(11.5,4).. (12.25,7) (11.8,7) -- (10.2,7);
\draw[draw=none,fill=gray!50,opacity=0.5] (11.3,7) -- (12.8,4) -- (14,6.5) -- (14,7) -- (11.3,7);

\draw[line width=0.75pt] (8,7) -- (11,1) -- (14,7);
\draw[line width=1pt] (9.5,7) ..controls(10.5,4).. (11,1) ..controls(11.5,4).. (12.25,7) (11.8,7);
\draw[line width=0.75pt] (8,5) -- (11,1) -- (14,5);
\draw[line width=1pt,->] (11,1) -- (11,7);
\draw[line width=0.5pt, dashed] (11.3,7) -- (12.8,4) -- (14,6.5);

\draw[fill=black] (11,1) circle (0.75mm);
\draw[fill=black] (12.8,4) circle (0.75mm);
\draw[fill=black] (13.1,6.1) circle (0.75mm);
\draw[->](12.8,4.2)--(13.07,5.9);
\draw[dashed] (10,7) --(10,1) (12,1)--(12,7);
\draw[thick](10,1)--(12,1);

\draw(3.7,1.15) node{\small$\varepsilon$} (2.3,1.15) node{\small$\varepsilon$};
\draw(11.7,1.15) node{\small$\varepsilon$} (10.3,1.15) node{\small$\varepsilon$};

\draw(3,0.5) node{\small$\seta{\hat{\setg{g}}_\veps}$};
\draw(1.125,5.7) node[rotate=-70]{\small$\acon{\delta}{\setd{\hat{\setg{g}}_\veps}}$};
\draw(0.7,4.5) node[rotate=-51]{\small$\acon{\Delta}{\setd{\hat{\setg{g}}_\veps}}$};
\draw(3,7.3) node{\small$\sets{\hat{\setg{g}}_\veps}$};
\draw(3.25,5.5) node[rotate=75]{\small$\polq{\hat{\setg{g}}_\veps}=0$};
\draw(4.9,4.2) node{\small${x}$};

\draw(11,0.5) node{\small$\seta{\hat{\setg{g}}_\veps}$};
\draw(11,7.3) node{\small$f_k(\sets{\hat{\setg{g}}_\veps})$};
\draw(10.72,5.6) node[rotate=90]{\small$\vspanp{\setd{\hat{\setg{g}}_\veps}}{\ven}$};
\draw(13.05,4.05) node{\small${x}$};
\draw(12.6,6.1) node{\small$f_k({x})$};
\draw(13,6.75) node{\small$\acon{2\delta}{\{x\}}$};
\draw(9,5.7) node[rotate=-60]{\small$\acon{2\delta}{\setd{\hat{\setg{g}}_\veps}}$};
\draw(8.7,4.5) node[rotate=-51]{\small$\acon{\Delta}{\setd{\hat{\setg{g}}_\veps}}$};

\draw(7,3.75) node{\small$f_k$};

\end{tikzpicture}
\end{center}
\caption{Behavior of the polynomial map $f_k$ (Theorem~\ref{coneS}).}\label{fig3}
\end{figure}

\begin{proof}
(i) Pick $x'\in\R^{n-1}$. If $(x',0)\notin\seta{\hat{\setg{g}}_\veps}$, then by Lemma \ref{polq1}(i) $\polq{\hat{\setg{g}}_\veps}(x',\x_n)$ is positive on $\{\x_n>0\}$ and $\lim_{x_n\to+\infty}B_k(x',x_n)x_n=+\infty$. If $(x',0)\in\seta{\hat{\setg{g}}_\veps}$, the polynomial $\polq{\setg{g}}(x',\x_n)$ has positive degree and its leading coefficient is strictly positive (see \ref{degqg}). In addition, by Lemma~\ref{polq1}(iii) $\pi_n(\sets{\hat{\setg{g}}_\veps})=\seta{\hat{\setg{g}}_\veps}$, so $\lim_{x_n\to+\infty}B_k(x',x_n)x_n=+\infty$.

(ii) Pick a point $x:=(x',x_n):=(y,z,x_n)\in\sets{\hat{\setg{g}}_\veps}$. By Lemma~\ref{polq1}(iii) we may write $x=(y,z,r_n+t_n)$ where $r_n>0$, $t_n\geq0$ and $\polq{\hat{\setg{g}}_\veps}(y,z,r_n)=0$. By Lemma~\ref{polq1}(i) $(y,z,r_n)\in\vspan{\seta{\hat{\setg{g}}_\veps}}{\ven}\cap\{\x_n\geq g_{m+1}\}$, so $(y,0,0)\in\setd{\hat{\setg{g}}_\veps}$ and $z\in I_\veps^{n-d-1}$. Thus, $\|z\|\le\veps\sqrt{n-d-1}$. We claim: $(0,z,x_n)\in\vcon{\delta}$.

As $(y,z,r_n)\in\{\x_n\geq g_{m+1}\}$, we deduce
$$
\delta x_n\ge\delta r_n\ge\delta g_{m+1}(y,z)\ge\delta\frac{\veps\sqrt{n-d-1}}{\delta}=\veps\sqrt{n-d-1}\ge\|z\|=\|(0,z)\|,
$$
so $(0,z,x_n)\in\vcon{\delta}$ and
$$
(y,z,x_n)=(y,0,0)+(0,z,x_n)\in\acon{\delta}{\setd{\hat{\setg{g}}_\veps}}.
$$

(iii) We show first: $\vspanp{\setd{\hat{\setg{g}}_\veps}}{\ven}\subset f_k(\sets{\hat{\setg{g}}_\veps})$.

Pick a point $p:=(y,0,0)\in\setd{\hat{\setg{g}}_\veps}$. By Lemma~\ref{polq1}(ii) there exist values $0<g_{m+1}(x')\leq r<t$ such that $\polq{\hat{\setg{g}}_\veps}(y,0,r)=0$ and $\polq{\hat{\setg{g}}_\veps}(y,0,t)=-1$. As $(y,0,r)\in\{\polq{\hat{\setg{g}}_\veps}=0,\x_n>0\}$, we deduce $(y,0,r),(y,0,t)\in\sets{\hat{\setg{g}}_\veps}$. Define $\phi_y(\x_n):=A(y,0,\x_n)=1+P^2(y,0,\x_n)\polq{\hat{\setg{g}}_\veps}(y,0,\x_n)$ and observe
\begin{equation}\label{nonzero}
\phi_y(r)=1\quad\text{and}\quad\phi_y(t)<0.
\end{equation}
Thus, there exists $s\in{]r,t[}$ such that $\phi_y(s)=0$. If we set $q:=(y,0,s)\in\sets{\hat{\setg{g}}_\veps}$, then $A(q)=0$ and $B_k(q)=0$, so $f_k(q)=p$. In addition, $\vspanp{q}{\ven}\subset\sets{\hat{\setg{g}}_\veps}$. As $f_k(q)=p$, the polynomial map $f_k$ preserve vertical lines and by (i) $\lim_{x_n\to+\infty}B_k(y,0,x_n)x_n=+\infty$, we deduce 
$$
\vspanp{p}{\ven}\subset f_k(\vspanp{q}{\ven})\subset f_k(\sets{\hat{\setg{g}}_\veps}),
$$ 
hence $\vspanp{\setd{\hat{\setg{g}}_\veps}}{\ven}\subset f_k(\sets{\hat{\setg{g}}_\veps})$.

\subsubsection{}\label{pick}We prove next: $f_k(\sets{\hat{\setg{g}}_\veps})\subset\acon{2\delta}{\setd{\hat{\setg{g}}_\veps}}$. 
Pick a point $x:=(y,z,x_n)\in\sets{\hat{\setg{g}}_\veps}$ and let us check: $(0,z,x_n)\in\vcon{\delta}$. 

By (ii) $(y,z,x_n)\in\sets{\hat{\setg{g}}_\veps}\subset\acon{\delta}{\setd{\hat{\setg{g}}_\veps}}$, so we write $(y,z,x_n)=(y_0,0,0)+(y_1,z,x_n)$ where $(y_0,0,0)\in\setd{\hat{\setg{g}}_\veps}$ and $(y_1,z,x_n)\in\vcon{\delta}$. Consequently, 
$$
\delta x_n\ge\|(y_1,z)\|\ge\|z\|=\|(0,z)\| 
$$
and $(0,z,x_n)\in\vcon{\delta}$. 

\subsubsection{}We show next: $(0,A(x)z,B_k(x)x_n)\in\vcon{2\delta}$. 

Observe that
$$
\frac{A(x)}{2B_k(x)}=\frac{A(x)}{A(x)+A^k(x)}=\frac{1}{1+A^{k-1}(x)}\leq1.
$$
As $(0,z,x_n)\in\vcon{\delta}$,
$$
2\delta B_k(x)x_n\geq 2B_k(x)\|(0,z)\|\geq A(x)\|(0,z)\|=\|(0,A(x)z)\|,
$$
hence $(0,A(x)z,B_k(x)x_n)\in\vcon{2\delta}$.

As $(y,0,0)\in\setd{\hat{\setg{g}}_\veps}$, we conclude 
$$
f_k(x)=f_k(y,z,x_n)=(y,0,0)+(0,A(x)z,B_k(x)x_n)\in\acon{2\delta}{\setd{\hat{\setg{g}}_\veps}}.
$$

(iv) Take a point $x:=(x',x_n):=(y,z,x_n)\in\acon{\Delta}{\setd{\hat{\setg{g}}_\veps}}\setminus\sets{\hat{\setg{g}}_\veps}$. We claim: $(0,z,x_n)\in\vcon{\Delta}$. 

Write $x=(y_1+y_2,z,x_n)$ where $(y_1,0,0)\in\setd{\hat{\setg{g}}_\veps}$ and $(y_2,z,x_n)\in\vcon{\Delta}$. This implies that $\Delta x_n\ge\|(y_2,z)\|\ge\|z\|=\|(0,z)\|$, hence $(0,z,x_n)\in\vcon{\Delta}$. 

As $x_n\ge 0$ and $x\notin\sets{\hat{\setg{g}}_\veps}$, we deduce $Q_{\hat{\setg{g}}_\veps}(x)>0$. We have
$$
f_k(x)-x=(0,(A(x)-1)z,(B_k(x)-1)x_n).
$$
Let us write
\begin{align*}
&A-1=2P^2\polq{\hat{\setg{g}}_\veps}+P^4\polq{\hat{\setg{g}}_\veps}^2,\\
&B_k-1=\frac{1}{2}\Big(2P^2\polq{\hat{\setg{g}}_\veps}+P^4\polq{\hat{\setg{g}}_\veps}^2
+\sum_{\ell=1}^{2k}\binom{2k}{\ell}(P^2\polq{\hat{\setg{g}}_\veps})^\ell\Big).
\end{align*}
Consequently, on $\{\polq{\hat{\setg{g}}_\veps}>0\}$
\begin{multline*}
\frac{B_k-1}{A-1}=\frac{2+2k+(1+\binom{2k}{2})P^2\polq{\hat{\setg{g}}_\veps}
+\sum_{\ell=3}^{2k}\binom{2k}{\ell}(P^2\polq{\hat{\setg{g}}_\veps})^{\ell-1}}{4+2P^2\polq{\hat{\setg{g}}_\veps}}\\
\geq\frac{2+2k+(1+k(2k-1))P^2\polq{\hat{\setg{g}}_\veps}}{4+2P^2\polq{\hat{\setg{g}}_\veps}}\geq\frac{k+1}{2}.
\end{multline*}
Let $k_0\geq1$ be such that $k_0+1\geq\frac{\Delta}{\delta}$. For $k\geq k_0$
$$
\frac{B_k-1}{A-1}\geq\frac{k+1}{2}\geq\frac{\Delta}{2\delta}.
$$
By \ref{pick} $(0,z,x_n)\in\vcon{\delta}\subset\vcon{\Delta}$. Thus, $\Delta x_n\ge\|(0,z)\|$, so for $k\geq k_0$
$$
2\delta(B_k(x)-1)x_n\ge(B_k(x)-1)\frac{2\delta}{\Delta}\|(0,z)\|\ge(A(x)-1)\|(0,z)\|=\|(0,(A(x)-1)z)\|,
$$
because $A(x)-1\geq0$ (recall that $Q_{\hat{\setg{g}}_\veps}(x)>0$). Therefore, $(0,(A(x)-1)z,(B_k(x)-1)x_n)\in\vcon{2\delta}$, so $f_k(x)-x\in\vcon{2\delta}$, as required.
\end{proof}

Figure~\ref{fig3} illustrates the behavior of the polynomial map $f_k$ for $k$ large enough.

\section{Convex polyhedra as polynomial images of $\R^n$}\label{s3}

The purpose of this section is to prove Theorem~\ref{main1}. We prove first this result for \em pointed cones\em, that is, unbounded convex polyhedra $\pol\subset\R^n$ with only one vertex $p$. In such case $\pol=\{p\}+\conv{\pol}{}$.

\begin{proof}[Proof of Theorem \em \ref{main1} \em for pointed cones]
Assume $\pol$ is a pointed cone with vertex $p$ and denote $\Cc_p:=\pol$ for the sake of clearness. Let $H$ be a hyperplane such that $\Cc_p\cap H=\{p\}$. Let $H'$ be a hyperplane parallel to $H$ such that $\p:=H'\cap\Cc_p$ is a bounded convex polyhedron of dimension $n-1$ (see \cite[Lem. 3.2]{fu3}). Assume $p$ is the origin, $H:=\{\x_n=0\}$ and $H':=\{\x_n=1\}$. By \cite[Thm. 1.2]{fgu1} there exists a regular map $g:=(g_1,\ldots,g_{n-1},1):\R^{n-1}\to\R^{n-1}\times\{1\}$ such that $g(\R^{n-1})=\Pp$. Write $g_i:=\frac{h_i}{h_0}$ where $h_0,h_i\in\R[\x']$ and $h_0$ is strictly positive on $\R^{n-1}$. Consider the polynomial map
$$
f:\R^n\to\R^n,\ (x',x_n)\mapsto x_n^2h(x')
$$
where $h:=(h_1,\ldots,h_{n-1},h_0)$. We claim: $f(\R^n)=\Cc_p$.

Pick a point $y\in\Cc_p$ and consider the vector $\vec{v}:=\overrightarrow{\orig y}$ and the ray $\vspanp{\orig}{\vec{v}}$. Observe that $\vspanp{\orig}{\vec{v}}\subset\Cc_p$ and the intersetion $H'\cap \vspanp{\orig}{\vec{v}}=:\{z\}\subset\p$ is a singleton. Thus, there exist $x'\in\R^{n-1}$ such that $g(x')=z$ and $\lambda\geq0$ such that $y=\lambda z$. Denote $x_n:=\sqrt{\frac{\lambda}{h_0(x')}}$ and observe that
$$
f(x',x_n)=\frac{\lambda}{h_0(x')}h(x')=\lambda g(x')=\lambda z=y.
$$ 
Consequently, $\Cc_p\subset f(\R^n)$. Conversely, if $x\in\R^n$, then 
$$
f(x)=x_n^2h(x')=x_n^2h_0(x')\frac{h(x')}{h_0(x')}=x_n^2h_0(x')g(x').
$$
As $x_n^2h_0(x')\geq0$ and $g(x')\in\Cc_p$, we conclude $f(x)\in\Cc_p$ because $\Cc_p$ is a cone with vertex the origin. Thus, $f(\R^n)=\Cc_p$, as required.
\end{proof}

We divide the proof of Theorem~\ref{main1} for the general case into three parts. The rest of the section is devoted to prove them. As a degenerate convex polyhedron $\pol\subset\R^n$ can be written in suitable coordinates as $\pol=\p\times\R^k$ where $\p$ is a non-degenerate convex polyhedron, it is enough to approach the non-degenerate case. If $k\geq n-1$, then $\pol$ is either $\R^n$ or a half-space, so it is trivially a polynomial image of $\R^n$. Thus, we assume in addition $n\geq2$. Let $\pol\subset\R^n$ be an $n$-dimensional non-degenerate convex polyhedron whose recession cone has dimension $n$. Let $X$ be the union of the affine subspaces of $\R^n$ spanned by the faces of $\pol$ of dimension $n-2$. We will prove the following statements.

\begin{prop}\label{main11}
There exists a polynomial map $h:\R^n\to\R^n$ such that $h(\R^n)=\pol\setminus X$.
\end{prop}

\begin{prop}\label{main12}
There exists a polynomial map $g:\R^n\to\R^n$ such that $g(\pol\setminus X)=\pol$.
\end{prop}

\begin{cor}\label{main13}
There exists a polynomial map $f:\R^{n+1}\to\R^n$ such that $f(\R^{n+1})=\Int(\pol)$.
\end{cor}

\subsection{Proof of Proposition~\ref{main11}}
Take a point $p\in\Int(\pol)$. Consider the pointed cone $\Cc_p:=\{p\}+\conv{\pol}{}\subset\Int(\pol)$. We have already proved that $\Cc_p$ is a polynomial image of $\R^n$, so it is enough to show that $\pol\setminus X$ is a polynomial image of $\Cc_p$. The idea here is to use $\Cc_p$ as a seed to fill the polyhedron $\pol$ by means of a sequence of polynomial maps whose images make $\Cc_p$ grow until we obtain $\pol\setminus X$. We start by placing the polyhedron $\pol$ in a convenient position (using affine changes of coordinates) in order to make our arguments clearer.

Denote the facets of $\pol$ with $\Ff_1$,\dots, $\Ff_r$. By Lemma~\ref{proj2} there exist $\vec{v}_1,\dots,\vec{v}_s\in\R^n$ such that 
\begin{equation}\label{x}
\bigcap_{k=1}^s\vspan{X}{\vec{v}_k}=X\quad\text{and}\quad\vec{v}_l
\in\Int(\conv{\pol}{})\setminus\bigcup_{k=1}^{l-1}
\vspan{\vec{X}}{\vec{v_k}}.
\end{equation}

\subsubsection{}\label{paso1}
Assume $\pol$ is placed in $\R^n$ so that $\pol\subset\{\x_n\ge 0\}$ and $\vec{v}_1=\ven\in\Int(\conv{\pol}{})$. Thus, $\pol$ has no vertical facets. Let $Z\subset\R^n$ denote a finite union of non-vertical hyperplanes $W_\ell:=\{w_\ell=0\}$, where $w_\ell$ denotes a linear equation of $W_\ell$ such that $w_\ell(\vec{\tt e}_n)>0$ for each $\ell$. This type of sets will be useful for the inductive process. Choose a facet $\Ff_i$ of $\pol$ and let $h_i(\x',\x_n)=h_i(\x',0)+\x_n=0$ be a non-zero linear equation for the (non-vertical) hyperplane spanned by $\Ff_i$. The affine change of coordinates
$$
\phi_i:\R^n\to\R^n,\ (x',x)\mapsto(x',h_i(x',0)+x_n)
$$ 
maps $\Ff_i$ onto $\pi_n(\Ff_i)\subset\{\x_n=0\}$ and keeps the vector $\ven$ invariant. To lighten the presentation we preserve the notations for all our geometric objects after applying the affine change of coordinates $\phi_i$. Write
$$
\Ff_i:=\{g_{i,1}\ge0,\dots,g_{i,m}\ge0,\x_n=0\}
$$
where each $g_{i,j}$ is a non-zero linear polynomial. Let $P_i:=P_{i,Z}$ be the square of the product of non-zero linear equations of the hyperplanes containing the facets of $\pol$, the hyperplanes containing the facets of $\Cc_p$ and the hyperplanes $W_\ell$. By Proposition~\ref{polq2} there exists $g_{i,m+1}\in\R[\x']$ strictly greater than 1 on $\R^{n-1}$ such that
$$
\{\x_n\ge g_{i,m+1}\}\subset\Cc_p\cap\{P_i>1\}.
$$
Let $\setg{g}_i:=(g_{i,1},\dots,g_{i,m},g_{i,m+1})$. By Lemma~\ref{polq1} there exist a polynomial $\polq{\setg{g}_i}$ such that the semialgebraic sets $\seta{\setg{g}_i}=\Int(\Ff_i)$ and $\sets{\setg{g}_i}=\vspan{\{\polq{\setg{g}}\le 0, \x_n>0\}}{\ven}$ satisfy 
\begin{equation}\label{inclus}
\sets{\setg{g}_i}\subset\vspan{\seta{\setg{g}_i}}{\ven}\cap\{\x_n\geq g_{i,m+1}\}\subset\vspan{\seta{\setg{g}_i}}{\ven}\cap\Cc_p\cap\{P_i>1\}. 
\end{equation}
Consider now the polynomial map
$$
f_i:=f_{i,Z}:=(f_{i1},\ldots,f_{in}):\R^n\to\R^n,\ x:=(x',x_n)\mapsto(x',x_n(1+P_i(x)
\polq{\setg{g}_i}(x))^2).
$$
\subsubsection{}\label{paso1m}We claim:
\begin{itemize}
\item[(i)] $f_i(\vspanp{\Ff_j}{\ven}\setminus X)
=\vspanp{\Ff_j}{\ven}\setminus X$ and $f_i(\vspanp{\Int(\Ff_j)}{\ven})=\vspanp{\Int(\Ff_j)}{\ven}$ for $j=1,\ldots,r$. In addition, $f_i(\pol\setminus X)=\pol\setminus X$.
\item[(ii)] $f_i(\Cc_p)=\vspanp{\Int(\Ff_i)}{\ven}\cup\Cc_p$. 
\item[(iii)] $f_i|_{Z}=\id_{Z}$.
\end{itemize}

Let us prove the previous statements:

(i) By Proposition~\ref{proj} $\pi_n|_{\partial\pol}:\partial\pol\to\R^{n-1}\times\{0\}$ is a semialgebraic homeomorphism, so $\pol=\bigcup_{j=1}^r\vspanp{\Ff_j}{\ven}$ and $\pol\setminus X=\bigcup_{j=1}^r\vspanp{\Ff_j}{\ven}\setminus X$. Thus, once we prove the first part of the statement we will have in addition the second.

Pick a point $x:=(x',x_n)\in\Ff_j\subset\partial\pol$ for some $j=1,\ldots,r$. If $j\neq i$, then $\pi_n(x)\notin\Int(\Ff_i)=\seta{\setg{g}_i}$, so $\polq{\setg{g}_i}$ is by Lemma~\ref{polq1} strictly positive on $\vspanp{x}{\ven}$. Thus, for each $(x',t)\in\vspanp{x}{\ven}$ we have
$$
f_{in}(x',t)=t(1+P_i(x',t)\polq{\setg{g}_i}(x',t))^2\ge t\ge x_n.
$$
As $f_i(x)=x$, we have $f_i(\vspanp{x}{\ven})=\vspanp{x}{\ven}$ and $f_i(\Int(\vspanp{x}{\ven}))=\Int(\vspanp{x}{\ven})$. Consequently, 
$$
f_i(\vspanp{\Ff_j}{\ven}\setminus X)
=\vspanp{\Ff_j}{\ven}\setminus X\quad\text{and}\quad f_i(\vspanp{\Int(\Ff_j)}{\ven})=\vspanp{\Int(\Ff_j)}{\ven}.
$$

Assume now $j=i$. As $\partial\Ff_i=\Ff_i\cap\bigcup_{j\neq i}\Ff_j$, we have $\vspanp{(\partial\Ff_i)}{\ven}\setminus X=\bigcup_{j\neq i}\vspanp{(\Ff_i\cap\Ff_j)}{\ven}\setminus X$. As $f_i$ preserves vertical lines, 
$$
\vspanp{(\partial\Ff_i)}{\ven}\setminus X=\bigcup_{j\neq i}f_i(\vspanp{(\Ff_i\cap\Ff_j)}{\ven}\setminus X)=f_i\Big(\bigcup_{j\neq i}\vspanp{(\Ff_i\cap\Ff_j)}{\ven}\setminus X\Big)=f_i(\vspanp{(\partial\Ff_i)}{\ven}\setminus X).
$$
To finish it is enough to check $f_i(\vspanp{\Int(\Ff_i)}{\ven})=\vspanp{\Int(\Ff_i)}{\ven}$. By Theorem~\ref{atico2b}(i) and \eqref{inclus}
$$
\vspanp{\Int(\Ff_i)}{\ven}=\vspanp{\seta{\setg{g}_i}}{\ven}=f_i(\sets{\setg{g}_i})\subset f_i(\vspanp{\seta{\setg{g}_i}}{\ven})=f_i(\vspanp{\Int(\Ff_i)}{\ven})\subset\vspanp{\Int(\Ff_i)}{\ven}.
$$
The latter inclusion follows because $f_i$ preserves vertical lines and $f_i(\{\x_n\geq0\})\subset\{\x_n\geq0\}$.

(ii) As $\ven\in\Int(\conv{\pol}{})=\Int(\conv{\Cc_p}{})$, the restriction map $\pi_n|_{\partial\Cc_p}:\partial\Cc_p\to\R^{n-1}\times\{0\}$ is by Proposition~\ref{proj} a semialgebraic homeomorphism. Consequently, $\Cc_p=\vspanp{\partial\Cc_p}{\ven}$. Pick $x:=(x',x_n)\in\partial\Cc_p$. If $y:=\pi_n(x)\notin\seta{\setg{g}_i}$, then $\polq{\setg{g}_i}(x',t)>0$ for $t\geq x_n$ by Lemma \ref{polq1}, so
$$
f_{in}(x',t)=t(1+P_i(x',t)\polq{\setg{g}_i}(x',t))^2\ge t\ge x_n
$$
for $(x',t)\in\vspanp{x}{\ven}$. As $f_i(x)=x$, we deduce $f_i(\vspanp{x}{\ven})=\vspanp{x}{\ven}$. Thus, 
$$
f_i(\Cc_p\setminus\vspanp{\seta{\setg{g}_i}}{\ven})=\Cc_p\setminus\vspanp{\seta{\setg{g}_i}}{\ven}=\Cc_p\setminus\vspanp{\Int(\Ff_i)}{\ven}. 
$$
By Theorem~\ref{atico2b}(i), \eqref{inclus} and (i)
$$
\vspanp{\Int(\Ff_i)}{\ven}=\vspanp{\seta{\setg{g}_i}}{\ven}=f_i(\sets{\setg{g}_i})\subset f_i(\vspanp{\seta{\setg{g}_i}}{\ven}\cap\Cc_p)=f_i(\vspanp{\Int(\Ff_i)}{\ven}\cap\Cc_p)\subset\vspanp{\Int(\Ff_i)}{\ven}.
$$
Consequently,
\begin{multline*}
f_i(\Cc_p)=f_i(\Cc_p\setminus\vspanp{\seta{\setg{g}_i}}{\ven})\cup f_i(\vspanp{\seta{\setg{g}_i}}{\ven}\cap\Cc_p)\\
=(\Cc_p\setminus\vspanp{\Int(\Ff_i)}{\ven})\cup\vspanp{\Int(\Ff_i)}{\ven}=\Cc_p\cup\vspanp{\Int(\Ff_i)}{\ven}.
\end{multline*}

(iii) This is immediate because $P_i$ vanishes identically on the hyperplanes contained in $Z$.

Figure~\ref{fig4} illustrates how the polynomial map $f_1$ acts on the cone $\Cc_p$.

\begin{figure}[!ht]
\begin{center}
\begin{tikzpicture}[scale=1]

\draw[draw=none,fill=gray!20,opacity=1] (0,4) -- (1,1) -- (2.5,0) -- (4,0) -- (7,4);
\draw[draw=none,fill=gray!20,opacity=1] (8,4) -- (9,1) -- (10.5,0) -- (12,0) -- (15,4);
\draw[draw=none,fill=gray!50,opacity=1] (10.3333,4) -- (10.5,3.3333) -- (10.5,0) -- (12,0) -- (12,2.9) -- (12.87,4) -- (10.5,4);

\draw[draw=none,fill=gray!50,opacity=1] (2.3333,4) -- (3,1.5) -- (4.87,4);

\draw[line width=0.75pt,dashed] (10.5,0) -- (10.5,3.3333);
\draw[line width=0.75pt,dashed] (12,0) -- (12,2.9);

\draw[line width=1pt,dashed,gray] (0,4) -- (1,1) -- (2.5,0) -- (4,0) --(7,4);
\draw[line width=1pt,dashed,gray] (8,4) -- (9,1) -- (10.5,0)(12,0) -- (15,4);

\draw[line width=1pt] (2.3333,4) -- (3,1.5) -- (4.87,4);
\draw[line width=1pt] (10.3333,4) -- (11,1.5) -- (12.87,4);
\draw[draw=none,fill=gray!50,opacity=1] (10.5,3.5) --(10.5,0)--(12,0) -- (12,4);\draw[line width=1pt](10.5,0) -- (12,0);
\draw[dashed] (10.5,3.5) --(10.5,0)(12,0) -- (12,2.9);

\draw[fill=white] (1,1) circle (0.75mm);
\draw[fill=white] (2.5,0) circle (0.75mm);
\draw[fill=white] (4,0) circle (0.75mm);
\draw[fill=gray,draw=gray] (3,1.5) circle (0.75mm);
\draw[fill=white] (9,1) circle (0.75mm);
\draw[fill=white,draw] (10.5,0) circle (0.75mm);
\draw[fill=white,draw] (12,0) circle (0.75mm);

\draw(3.5,3.5) node{\small$\Cc_p$};
\draw[line width=1pt,->] (6,2) -- (8,2);
\draw(7,2.25) node{\small$f_1$};
\draw(0.8,3.5) node{\small$\pol$};
\draw(8.8,3.5) node{\small$\pol$};
\draw(11.4,3.5) node{\small$f_1(\Cc_p)$};
\draw(3.25,1.5) node{\small$p$};
\draw(3.25,0.25) node{\small$\seta{\setg{g}_1}$};
\draw(11.25,0.25) node{\small$\seta{\setg{g}_1}$};

\end{tikzpicture}
\end{center}
\caption{Behavior of the polynomial map $f_1$ over $\Cc_p$.}\label{fig4}
\end{figure}

\subsubsection{}\label{paso2m} Fix $1\leq k\leq s$ and consider $\pol$ placed in $\R^n$ (by means of an affine change of coordinates $\psi_k$) so that $\pol\subset\{\x_n\ge 0\}$ and $\vec{v}_k=\ven$. We preserve the names for all our geometric objects after applying the change of coordinates $\psi_k$. Set $X_l:=\vspan{X}{\vec{v}_l}$ and $Z_{k-1}:=\bigcup_{l=1}^{k-1}X_l$, which are unions of hyperplanes because each $(n-2)$-affine subspace in $X$ is parallel to none of the vectors $\vec{v}_j$. In addition, each hyperplane in $Z_{k-1}$ is not parallel to $\vec{v}_k$ (see \ref{x}). For each $i=1,\dots,r$ consider the affine change of coordinates $\phi_i$ described in \ref{paso1} and the polynomial map $f_{i,k}:=f_{i,Z}$ introduced in \ref{paso1m} taking $Z:=\phi_i(Z_{k-1})$ and $Z_0=\varnothing$.

Define the polynomial map
$$
F_k:=\hat{f}_{r,k}\circ\cdots\circ\hat{f}_{1,k}\quad\text{where $\hat{f}_{i,k}:=\phi_i^{-1}\circ f_{i,k}\circ\phi_i$.}
$$
We claim: 
\begin{itemize}
\item[(i)] $\Cc_p\cup(\pol\setminus X_k)= F_k(\Cc_p)$ for $1\le k\le s$;
\item[(ii)] $\Cc_p\cup(\pol\setminus\bigcap_{j=1}^kX_j)\subset F_k(\Cc_p\cup(\pol\setminus\bigcap_{j=1}^{k-1}X_j))$ and $F_k(\pol\setminus X)=\pol\setminus X$ for $1\le k\le s$.
\end{itemize}
To prove (i) we use recursively \ref{paso1m}. Indeed,
\begin{align*}
F_k(\Cc_p)&=(\hat{f}_{r,k}\circ\cdots\circ
\hat{f}_{2,k})(\hat{f}_{1,k}(\Cc_p))=(\hat{f}_{r,k}\circ\cdots\circ
\hat{f}_{2,k})(\Cc_p\cup\vspanp{\Int(\Ff_1)}{\ven})\\
&=\cdots=\Cc_p\cup
\vspanp{\Int(\Ff_r)}{\ven}\cup\cdots\cup\vspanp{\Int(\Ff_1)}{\ven}=\Cc_p\cup(\pol\setminus\vspanp{X}{\ven})\\
&=\Cc_p\cup(\pol\setminus X_k).
\end{align*}
Figure~\ref{fig5} shows the action of polynomial map $F_1$ on the cone $\Cc_p$.
\begin{figure}[!ht]
\begin{center}
\begin{tikzpicture}[scale=1]

\draw[draw=none,fill=gray!20,opacity=1] (0,4) -- (1,1) -- (2.5,0) -- (4,0) -- (7,4);
\draw[draw=none,fill=gray!50,opacity=1] (8,4) -- (9,1) -- (10.5,0) -- (12,0) -- (15,4);
\draw[draw=none,fill=gray!50,opacity=1] (2.3333,4) -- (3,1.5) -- (4.87,4);

\draw[dashed,line width=1.5pt,color=white] (9,1) -- (9,4);
\draw[dashed,line width=1.5pt,color=white] (10.5,0) -- (10.5,3.3333);
\draw[dashed,line width=1.5pt,color=white] (12,0) -- (12,4);

\draw[line width=1pt,dashed,gray] (0,4) -- (1,1) -- (2.5,0) -- (4,0) -- (7,4);
\draw[line width=1pt] (8,4) -- (9,1) -- (10.5,0) -- (12,0) -- (15,4);
\draw[line width=1pt] (2.3333,4) -- (3,1.5) -- (4.87,4);
\draw[color=gray,line width=1pt] (10.3333,4) -- (11,1.5) -- (12.87,4);

\draw[draw=none,fill=gray!50,opacity=1] (10.3333,4) -- (11,1.5) -- (12.87,4);

\draw[fill=white] (1,1) circle (0.75mm);
\draw[fill=white] (2.5,0) circle (0.75mm);
\draw[fill=white] (4,0) circle (0.75mm);
\draw[fill=gray] (3,1.5) circle (0.75mm);
\draw[fill=gray,draw=gray] (11,1.5) circle (0.75mm);
\draw[fill=white,draw] (9,1) circle (0.75mm);
\draw[fill=white,draw] (10.5,0) circle (0.75mm);
\draw[fill=white,draw] (12,0) circle (0.75mm);

\draw(3.3,3.5) node{\small$\Cc_p$};
\draw(11.4,3.5) node{\small$F_1(\Cc_p)$};
\draw[line width=1pt,->] (6,2) -- (8,2);
\draw(7,2.25) node{\small$F_1$};
\draw(0.6,3.5) node{\small$\pol$};
\draw(8.6,3.5) node{\small$\pol$};
\draw(3.25,1.5) node{\small$p$};
\draw(11.25,1.5) node{\small$p$};

\end{tikzpicture}
\end{center}
\caption{Behavior of the polynomial map $F_1$ over $\Cc_p$.}\label{fig5}
\end{figure}

We show now (ii). As each polynomial map $\hat{f}_i$ appearing in the definition of $F_k$ satisfies by \ref{paso1m} (i) $\hat{f}_i(\pol\setminus X)=\pol\setminus X$, we have $F_k(\pol\setminus X)=\pol\setminus X$. By \ref{paso1m}(iii) $\hat{f}_{i,k}|_{Z_{k-1}}=\id_{Z_{k-1}}$ for $i=1,\ldots,r$. As $X_i\subset Z_{k-1}$ for $i=1,\ldots,k-1$, we have $\hat{f}_{i,k}|_{\bigcap_{j=1}^{k-1}X_j}=\id_{\bigcap_{j=1}^{k-1}X_j}$, so $F_k(Y)\setminus\bigcap_{j=1}^{k-1}X_j\subset F_k(Y\setminus\bigcap_{j=1}^{k-1}X_j)$ for each $Y\subset\R^n$. As $X\subset\bigcap_{j=1}^{k-1}X_j$, we deduce by \ref{paso1m} 
\begin{align*}
F_k\Big(\Cc_p\cup\Big(\pol\setminus\bigcap_{j=1}^{k-1}X_j\Big)\Big)&=F_k(\Cc_p)\cup F_k\Big(\Big((\pol\setminus X)\setminus\bigcap_{j=1}^{k-1}X_j\Big)\Big)\\
&\supset(\Cc_p\cup(\pol\setminus X_k))\cup\Big(F_k(\pol\setminus X)\setminus\bigcap_{j=1}^{k-1}X_j\Big)\\
&=(\Cc_p\cup(\pol\setminus X_k))\cup\Big((\pol\setminus X)\setminus\bigcap_{j=1}^{k-1}X_j\Big)\\
&=\Cc_p\cup\Big(\pol\setminus\bigcap_{j=1}^kX_j\Big).
\end{align*}

\subsubsection{}\label{paso3m}
Let us finish the proof of Proposition~\ref{main11}. Define for $k=1,\dots,s$ the polynomial map
$$
\hat{F}_k:=\psi_k^{-1}\circ F_k\circ\psi_k:\R^n\to\R^n.
$$
As $\Cc_p$ is a pointed cone, we have already constructed a polynomial map $h_0:\R^n\to \R^n$ such that $h_0(\R^n)=\Cc_p$. We claim: \em the polynomial map
$$
h:=\hat{F}_s\circ\cdots\circ\hat{F}_1\circ h_0
$$
satisfies $h(\R^n)=\pol\setminus X$\em. It is enough to show: $(\hat{F}_s\circ\cdots\circ\hat{F}_1)(\Cc_p)=\pol\setminus X$. 

Using recursively \ref{paso2m} we deduce
\begin{align*}
&\Cc_p\cup(\pol\setminus X_1)= \hat{F}_1(\Cc_p)\subset\pol\setminus X\\
&\Cc_p\cup(\pol\setminus (X_1\cap X_2))\subset(\hat{F}_2\circ\hat{F}_1)(\Cc_p)\subset\pol\setminus X\\
&\hskip 0.55cm\vdots\\
&\Cc_p\cup\Big(\pol\setminus\bigcap_{j=1}^sX_j\Big)\subset(\hat{F}_s\circ\cdots\circ\hat{F}_1)(\Cc_p)
\subset\pol\setminus X.
\end{align*}
To illustrate this process Figure~\ref{fig6} shows how $\hat{F}_2$ acts on $\hat{F}_1(\Cc_p)$. As $\pol\setminus(\bigcap_{j=1}^sX_j)=\pol\setminus X$ and $\Cc_p\subset\pol\setminus X$,
$$
\pol\setminus X=\Cc_p\cup\Big(\pol\setminus\Big(\bigcap_{j=1}^sX_j\Big)\Big)\subset(\hat{F}_s\circ\cdots\circ\hat{F}_1)(\Cc_p)\subset\pol\setminus X,
$$
so $h(\pol\setminus X)=\pol\setminus X$, as required.\qed

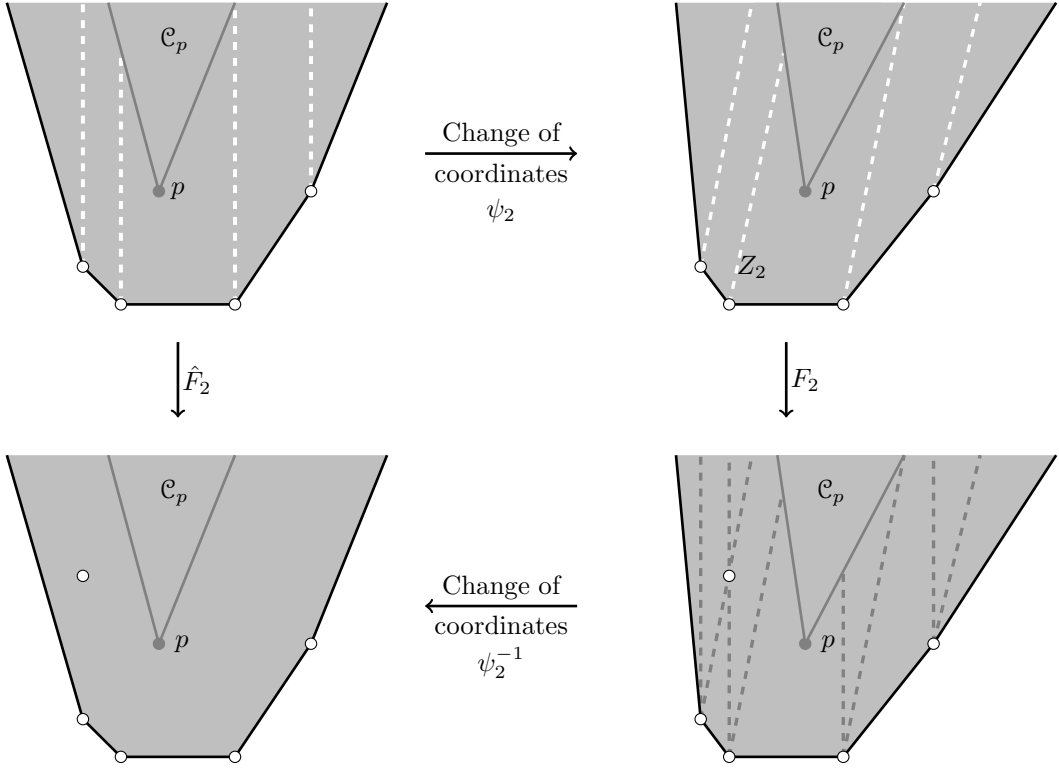
\begin{figure}[!ht]
\begin{center}
\begin{tikzpicture}[scale=1]


\draw[draw=none,fill=gray!50,opacity=1] (1,10) -- (2,6.5) -- (2.5,6) -- (4,6) -- (5,7.5) -- (6,10);
\draw[draw=none,fill=gray!50,opacity=1] (2.3333,10) -- (3,7.5) -- (4,10);

\draw[dashed,line width=1.5pt,color=white] (2,6.5) -- (2,10);
\draw[dashed,line width=1.5pt,color=white] (2.5,6) -- (2.5,9.3333);
\draw[dashed,line width=1.5pt,color=white] (4,6) -- (4,10);
\draw[dashed,line width=1.5pt,color=white] (5,7.5) -- (5,10);

\draw[line width=1pt] (1,10) -- (2,6.5) -- (2.5,6) -- (4,6) -- (5,7.5) -- (6,10);
\draw[line width=1pt, draw=gray] (2.3333,10) -- (3,7.5) -- (4,10);

\draw[fill=white,draw] (2,6.5) circle (0.75mm);
\draw[fill=white,draw] (2.5,6) circle (0.75mm);
\draw[fill=white,draw] (4,6) circle (0.75mm);
\draw[fill=white,draw] (5,7.5) circle (0.75mm);
\draw[fill=gray,draw=gray] (3,7.5) circle (0.75mm);

\draw(3.2,9.5) node{\small$\Cc_p$};
\draw[line width=1pt,->] (6.5,8) -- (8.5,8);
\draw(7.5,8.25) node{\small{Change of}};
\draw(7.5,7.75) node{\small{coordinates}};
\draw(7.5,7.25) node{\small{$\psi_2$}};

\draw(3.25,7.5) node{\small$p$};

\draw[draw=none,fill=gray!50,opacity=1] (9.8,10) -- (10.125,6.5) -- (10.5,6) -- (12,6) -- (13.1875,7.5) -- (14.8,10);
\draw[draw=none,fill=gray!50,opacity=1] (10.8333,10) -- (11.1875,7.5) -- (12.5,10);

\draw[dashed,line width=1.25pt,color=white] (10.125,6.5) -- (10.8,10);
\draw[dashed,line width=1.25pt,color=white] (10.5,6) -- (11.216666,9.3333);
\draw[dashed,line width=1.25pt,color=white] (12,6) -- (12.8,10);
\draw[dashed,line width=1.25pt,color=white] (13.1875,7.5) -- (13.8,10);

\draw[line width=1pt] (9.8,10) -- (10.125,6.5) -- (10.5,6) -- (12,6) -- (13.1875,7.5) -- (14.8,10);

\draw[fill=gray,draw=gray] (11.5,7.5) circle (0.75mm);
\draw[fill=white,draw] (10.125,6.5) circle (0.75mm);
\draw[fill=white,draw] (10.5,6) circle (0.75mm);
\draw[fill=white,draw] (12,6) circle (0.75mm);
\draw[fill=white,draw] (13.1875,7.5) circle (0.75mm);

\draw[draw=gray,line width=1pt,fill=gray!50] (11.1333,10) -- (11.5,7.5) -- (12.8,10);

\draw(11.8375,9.5) node{\small$\Cc_p$};
\draw(11.8,7.5) node{\small$p$};
\draw(10.8,6.5) node{\small$Z_2$};

\draw[line width=1pt,->] (11.25,5.5) -- (11.25,4.5);
\draw(11.5,5) node{\small{$F_2$}};
\draw(3.5,5) node{\small{$\hat{F}_2$}};
\draw[line width=1pt,->] (3.25,5.5) -- (3.25,4.5);


\draw[draw=none,fill=gray!50,opacity=1] (1,4) -- (2,0.5) -- (2.5,0) -- (4,0) -- (5,1.5) -- (6,4);

\draw[line width=1pt] (1,4) -- (2,0.5) -- (2.5,0) -- (4,0) -- (5,1.5) -- (6,4);
\draw[line width=1pt, draw=gray] (2.3333,4) -- (3,1.5) -- (4,4);

\draw[fill=white,draw] (2,0.5) circle (0.75mm);
\draw[fill=white,draw] (2.5,0) circle (0.75mm);
\draw[fill=white,draw] (4,0) circle (0.75mm);
\draw[fill=white,draw] (5,1.5) circle (0.75mm);

\draw[fill=white,draw] (2,2.4) circle (0.75mm);

\draw[line width=1pt,<-] (6.5,2) -- (8.5,2);
\draw(7.5,2.25) node{\small{Change of}};
\draw(7.5,1.75) node{\small{coordinates}};
\draw(7.5,1.25) node{\small{$\psi_2^{-1}$}};

\draw[draw=none,fill=gray!50,opacity=1] (9.8,4) -- (10.125,0.5) -- (10.5,0) -- (12,0) -- (13.1875,1.5) -- (14.8,4);

\draw[draw=gray,line width=1.25pt,dashed] (10.125,0.5) -- (10.125,4);
\draw[draw=gray,line width=1.25pt,dashed] (10.5,0) -- (10.5,4);
\draw[draw=gray,line width=1.25pt,dashed] (12,0) -- (12,3);
\draw[draw=gray,line width=1.25pt,dashed] (13.1875,1.5) -- (13.1875,4);

\draw[draw=gray,line width=1.25pt,dashed] (10.125,0.5) -- (10.8,4);
\draw[draw=gray,line width=1.25pt,dashed] (10.5,0) -- (11.2,3.5);
\draw[draw=gray,line width=1.25pt,dashed] (12,0) -- (12.8,4);
\draw[draw=gray,line width=1.25pt,dashed] (13.1875,1.5) -- (13.8,4);

\draw[line width=1pt] (9.8,4) -- (10.125,0.5) -- (10.5,0) -- (12,0) -- (13.1875,1.5) -- (14.8,4);

\draw[fill=gray,draw=gray] (11.5,1.5) circle (0.75mm);
\draw[fill=gray,draw=gray] (3,1.5) circle (0.75mm);
\draw[fill=white,draw] (10.125,0.5) circle (0.75mm);
\draw[fill=white,draw] (10.5,0) circle (0.75mm);
\draw[fill=white,draw] (12,0) circle (0.75mm);
\draw[fill=white,draw] (13.1875,1.5) circle (0.75mm);

\draw[fill=white,draw] (10.5,2.4) circle (0.75mm);

\draw[draw=none,fill=gray!50,opacity=1] (11.1333,4) -- (11.5,1.5) -- (12.8,4);
\draw[line width=1pt, draw=gray] (11.1333,4) -- (11.5,1.5) -- (12.8,4);

\draw(11.8375,3.5) node{\small$\Cc_p$};
\draw(11.8,1.5) node{\small$p$};
\draw(3.3,1.5) node{\small$p$};
\draw(3.2,3.5) node{\small$\Cc_p$};

\end{tikzpicture}
\end{center}
\caption{Behavior of the polynomial map $\hat{F}_2$.}\label{fig6}
\end{figure}

\subsection{Proof of Proposition~\ref{main12}}
Let $\Ee$ be a face of $\pol$ of dimension $d\le n-2$. We write $x:=(x',x_n):=(y,z,x_n)\in\R^d\times\R^{n-d-1}\times\R$. Assume $\pol\cap\{\x_n=0\}=\Ee$, $\pol\subset\{\x_n\ge0\}$ and $\ven\in\Int(\conv{\pol}{})$. Write $\Int(\Ee):=\{g_1>0,\ldots,g_r>0\}\times\{0\}\subset\R^{d}\times\origs$ where each $g_i\in\R[\y]:=\R[\x_1,\ldots,\x_d]$. By Proposition~\ref{conpols} there exist positive numbers $\delta,\Delta$ such that $\vcon{2\delta}\setminus\origs\subset\Int(\conv{\pol}{})$ and $\acon{2\delta}{\Ee}\subset\pol\subset\acon{\Delta}{\Ee}$. As $\ven\in\Int(\conv{\pol}{})$, the hyperplanes spanned by the facets of $\pol$ are non-vertical. Let $P\in\R[\x]$ be the product of linear equations of these hyperplanes, so $\partial\pol\subset\{P=0\}$. Fix $\veps>0$. By Proposition~\ref{polq2} there exists $g\in\R[\x']$ such that $g>1+\veps\frac{\sqrt{n-d-1}}{\delta}$ on $\R^{n-1}$ and $\{\x_n\geq g\}\subset\{P>1\}$. Denote $I_\veps:={]{-\veps},+\veps[}$ and consider the admissible tuple of polynomials
\begin{align*}
\hat{\setg{g}}_\veps&=(g_1,\dots,g_r,\x_{d+1}+\veps,\dots,\x_{n-1}+\veps,\veps-\x_{d+1},\dots,\veps-\x_{n-1},g).
\end{align*}
If we write $m:=r+2(n-1-d)$, then $\hat{\setg{g}}_\veps$ consists of $m+1$ polynomials. Rename $g_{m+1}:=g$. The admisible tuple $\hat{\setg{g}}_\veps$ has associated a polynomial $Q_{\hat{\setg{g}}_\veps}$ constructed in Lemma~\ref{polq1} and semialgebraic sets $\seta{\hat{\setg{g}}_\veps}$, $\sets{\hat{\setg{g}}_\veps}$ and $\setd{\hat{\setg{g}}_\veps}$ provided in \ref{lower}. Observe that $\Int(\Ee)=\setd{\hat{\setg{g}}_\veps}$. By Lemma~\ref{polq1} $\sets{\hat{\setg{g}}_\veps}\subset\{\x_n\geq g_{m+1}\}\subset\{P>1\}$. For each $k\geq1$ consider the polynomial map
$$
f_k:\R^n\to\R^n,\ x:=(x',x_n)=(y,z,x_n)\mapsto(y,A(x)z,B_k(x)x_n)
$$
where $A:=(1+P^2\polq{{\hat{\setg{g}}_\veps}})^2$ and $B_k:=\frac{A+A^k}{2}$. Note that $\veps,\delta,\Delta>0$, $g_{m+1}$, $\hat{\setg{g}}_\veps$ and $P$ satisfy the hypotheses of Theorem~\ref{coneS}. Let $k_0\geq1$ be the positive integer constructed in Theorem~\ref{coneS}(iv).

\subsubsection{Main claim:}\label{clue}
\em Let $\Tt$ be a semialgebraic set such that $\pol\setminus X\subset\Tt\subset\pol$. For $k\geq k_0$
\begin{equation}\label{fkt}
f_k(\Tt)=\Tt\cup\Int(\Ee)\subset\pol.
\end{equation}
\em 
 
To show \eqref{fkt} we prove the following facts for $k\geq k_0$:
\begin{description}
\item[Fact 1] $f_k(\Tt)\setminus\Int(\Tt)=(\Tt\setminus\Int(\Tt))\cup\Int(\Ee)$.
\item[Fact 2] $\Int(\Tt)\subset f_k(\Int(\Tt))$.
\end{description}
Once they are proved and since $\Int(\Tt)=\Int(\pol)$, we conclude
$$
f_k(\Tt)=(f_k(\Tt)\setminus\Int(\pol))\cup(f_k(\Tt)\cap\Int(\pol))=(\Tt\setminus\Int(\pol))\cup\Int(\Ee)\cup\Int(\pol)=\Tt\cup\Int(\Ee)
$$
and equality $f_k(\Tt)=\Tt\cup\Int(\Ee)$ follows.

\subsubsection{\bf Proof of Fact 1}\label{3d2}
We show first: $(\Tt\setminus\Int(\Tt))\cup\Int(\Ee)\subset f_k(\Tt)\setminus\Int(\Tt)$.

As $\vcon{\delta}\setminus\origs\subset\Int(\conv{\pol}{})$, Theorem~\ref{coneS}(ii) provides
$$
\sets{\hat{\setg{g}}_\veps}\subset\acon{\delta}{\setd{\hat{\setg{g}}_\veps}}\cap\{\x_n>0\}\subset\Int(\Ee)+\Int(\conv{\pol}{})\subset\Int(\pol)=\Int(\Tt)\subset\Tt.
$$
By Theorem~\ref{coneS}(iii) the inclusion $\Int(\Ee)\subset f_k(\sets{\hat{\setg{g}}_\veps})\subset f_k(\Tt)$ holds for $k\geq1$, hence 
\begin{equation}\label{inte1}
\Int(\Ee)=\Int(\Ee)\setminus\Int(\Tt)\subset f_k(\Tt)\setminus\Int(\Tt).
\end{equation}
As $f_k|_{\Tt\setminus\Int(\Tt)}=\id_{\Tt\setminus\Int(\Tt)}$ because $\partial\Tt\subset\{P=0\}$,
\begin{equation}\label{inte2}
\Tt\setminus\Int(\Tt)=f_k(\Tt\setminus\Int(\Tt))\subset f_k(\Tt)\quad\leadsto\quad\Tt\setminus\Int(\Tt)\subset f_k(\Tt)\setminus\Int(\Tt)
\end{equation}
and the inclusion $(\Tt\setminus\Int(\Tt))\cup\Int(\Ee)\subset f_k(\Tt)\setminus\Int(\Tt)$ follows from \eqref{inte1} and \eqref{inte2}.

To prove $f_k(\Tt)\setminus\Int(\Tt)\subset(\Tt\setminus\Int(\Tt))\cup\Int(\Ee)$, pick a point $x\in\Tt$ such that $f_k(x)\not\in\Int(\Tt)$. If $x\in\Tt\setminus\Int(\Tt)$, then $f_k(x)=x\in\Tt\setminus\Int(\Tt)$ because $f_k|_{\Tt\setminus\Int(\Tt)}=\id_{\Tt\setminus\Int(\Tt)}$. If $x\in\Int(\Tt)$, then $x\in\sets{\hat{\setg{g}}_\veps}$. 

Otherwise, as $\Int(\Tt)\subset\acon{\Delta}{\Int(\Ee)}$, we deduce by Theorem~\ref{coneS}(iv) that $f_k(x)\in\acon{2\delta}{\{x\}}\subset\{x\}+\Int(\conv{\pol}{})\subset\Int(\pol)=\Int(\Tt)$, which is a contradiction. 

As $x\in\sets{\hat{\setg{g}}_\veps}$, we have by Theorem~\ref{coneS}(iii)
$$
f_k(x)\in f_k(\sets{\hat{\setg{g}}_\veps})\subset\acon{2\delta}{\Int(\Ee)}\subset\Int(\Ee)+\Int(\conv{\pol}{})\subset\Int(\Ee)\cup\Int(\Tt),
$$
so $f_k(x)\in\Int(\Ee)$. Therefore, the inclusion $f_k(\Tt)\setminus\Int(\Tt)\subset(\Tt\setminus\Int(\Tt))\cup\Int(\Ee)$ holds, as required.\qed

\subsubsection{\bf Proof of Fact 2}\label{3d5}
As $\Int(\Tt)=\Int(\pol)$, we have to check: $\Int(\pol)\subset f_k(\Int(\pol))$. Its proof is long and requires a topological argument based on a result by Janiszewski \cite{ja}. 

Pick a point $x_0:=(x'_0,x_{0n}):=(y_0,z_0,x_{0n})\in\Int(\pol)$, so $x_{0n}>0$. If $z_0=0$, consider the intersection $\vspan{x_0}{\ven}\cap\pol=\vspanp{x_1}{\ven}$, where the point $x_1:=(y_0,0,r)\in\partial\pol$ must satisfy $0\le r< x_{0n}$ and $\vspan{x_0}{\ven}\cap\Int(\pol)=\Int(\vspanp{x_1}{\ven})$. As $A|_{\partial\pol}=1$, we have $B_k(y_0,0,r)r=\frac{A(y_0,0,r)+A^k(y_0,0,r)}{2}r=r$. By Theorem~\ref{coneS}(i) $\lim_{x_n\to+\infty}B_k(y_0,0,x_n)x_n=+\infty$. As $r<x_{0n}$, there exists $s>r$ such that $B_k(y_0,0,s)s=x_{0n}$, so $f_k(y_0,0,s)=(y_0,0,x_{0n})=x_0$. Note that $(y_0,0,s)\in\Int(\pol)$, so $x_0\in f_k(\Int(\pol))$. 

\paragraph{}\label{3d3} By Fact 1 for $\Tt=\pol$ we have $f_k(\pol)\setminus\Int(\pol)=(\pol\setminus\Int(\pol))\cup\Ee\subset\pol$, so $f_k(\pol)\subset\pol$.

\paragraph{}\label{claim4} We assume next that $z_0\neq0$ and let us prove: \em there exists $x_1\in\Int(\pol)$ such that $f_k(x_1)=x_0$ for each $k\geq k_0$\em. 

The proof of \ref{claim4} is conducted in several steps.

\paragraph{} Consider the 2-dimensional plane $\Pi$ determined by the points $(y_0,0,0)$, $(y_0,0,x_{0n})$ and $x_0$. Let us show: \em $f_k(\p)\subset\p$ where $\p:=\pol\cap\Pi$\em.

As $f_k(x)=f_k(y,z,x_n)=(y,A(x)z,B_k(x)x_n)$ for $x:=(y,z,x_n)$, we have $f_k(\Pi)\subset\Pi$. Since $f_k(\pol)\subset\pol$, it holds
$$
f_k(\p)=f_k(\pol\cap\Pi)\subset f_k(\pol)\cap f_k(\Pi)\subset\pol\cap\Pi=\p.
$$ 

\paragraph{}Set coordinates $(u,v)$ in $\Pi$ with respect to the affine reference 
$$
\Rr:=\{O:=(y_0,0,0);\vec{w}_1=(0,\tfrac{z_0}{\|z_0\|},0),\vec{w}_2=(0,0,1)\}. 
$$
Observe that $O+u\vec{w}_1+v\vec{w}_2=(y_0,\tfrac{z_0}{\|z_0\|}u,v)$ and 
\begin{align*}
f_k(O+u\vec{w}_1+v\vec{w}_2)&=\big(y_0,A(y_0,\tfrac{z_0}{\|z_0\|}u,v)\tfrac{z_0}{\|z_0\|}u,B_k(y_0,\tfrac{z_0}{\|z_0\|}u,v)v\big)\\
&\equiv\Big(\alpha(u,v)u,\frac{\alpha(u,v)+\alpha(u,v)^k}{2}v\Big)_{\Rr},
\end{align*}
where $\alpha(\u,\vv):=A(y_0,\tfrac{z_0}{\|z_0\|}\u,\vv)\in\R[\u,\vv]$. Consider the polynomial map
$$
G_k:=(G_{k1},G_{k2}):\R^2\to\R^2,\ (u,v)\mapsto\Big(\alpha(u,v)u,\frac{\alpha(u,v)+\alpha(u,v)^k}{2}v\Big).
$$
Note that $x_0\equiv(\|z_0\|,x_{0n})_{\Rr}=:(a,b)_{\Rr}$, so $a,b>0$. Consider the algebraic curve 
$$
Y_a:=\{\alpha(\u,\vv)\u-a=0\}=G_{k1}^{-1}(a)\subset\{\u>0\}. 
$$

\paragraph{}We claim: \em $r:=\max\{\veps\sqrt{n-d-1},a\}\geq u_0$ for each $(u_0,v_0)\in Y_a$\em. 

If $u_0>a$, then $\alpha(u_0,v_0)<1$. As $P^2$ is the square of a polynomial, $\polq{\hat{\setg{g}}_\veps}(y_0,\tfrac{z_0}{\|z_0\|}u_0,v_0)<0$. By Lemma~\ref{polq1} we have $(y_0,\tfrac{z_0}{\|z_0\|}u_0,0)\in\seta{\hat{\setg{g}}_\veps}$. In particular, $\tfrac{z_0}{\|z_0\|}u_0\in I_\veps^{n-d-1}$, so 
$$
u_0\leq\|(\veps,\overset{(n-d-1)}{\ldots},\veps)\|=\veps\sqrt{n-d-1}\leq r.
$$ 

\paragraph{}Consider the convex polygon $\p_0:=\p\cap\{0\le\u\le r+1\}$ and the singleton $\{q\}:=\partial\p_0\cap\{\u=a\}$. Write $q:=(a,c)_{\Rr}$. Let us check: $Y_a\cap\partial\Pp_0=\{q\}$.

As $\Pi$ meets $\Int(\pol)$, we have $\partial\Pp=\partial\pol\cap\Pi$. As $Y_a\cap(\{\u=0\}\cup\{\u=r+1\})=\varnothing$, then 
$$
Y_a\cap\partial\Pp_0\subset Y_a\cap\partial\Pp\subset Y_a\cap\partial\pol\subset Y_a\cap\{P=0\}\subset Y_a\cap\{\alpha=1\}=Y_a\cap\{\u=a\}.
$$
Thus, $Y_a\cap\partial\Pp_0=Y_a\cap\partial\Pp_0\cap\{\u=a\}\subset\{q\}$. As $q\in\partial\Pp_0\setminus(\{\u=0\}\cup\{\u=r+1\})\subset\partial\pol$, we have $\alpha(a,c)=A(q)=1$. As $q\in\{\u=a\}\cap\{\alpha(\u,\vv)=1\}$, we conclude $q\in Y_a$, so $Y_a\cap\partial\Pp_0=\{q\}$.

\paragraph{}\label{4b37} Given a connected topological space $T$ and different points $p,q\in T$, we say that $K\subset T$ separates $p$ and $q$ if these points belong to different connected components of $T\setminus K$. Given $S\subset\R^2$, we say that $S$ is `upperly unbounded' to refer that it is unbounded in the direction of the second coordinate. We claim: \em There exists an upperly unbounded connected component $Z$ of $Y_a\cap\p_0$ passing through $q$ such that $Z\setminus\{q\}\subset\Int(\pol)$\em. 

To prove this claim we will make use of Janiszewski's Theorem (see \cite{ja} or \cite[Thm. A]{bing}): \em If $K_1$ and $K_2$ are compact subsets of the plane $\R^2$ whose intersection is connected, a pair of points that is separated by neither $K_1$ nor $K_2$ is neither separated by their union $K_1\cup K_2$\em. The proof of our claim is conducted in several steps:

\begin{figure}[!ht]
\begin{center}
\begin{tikzpicture}[scale=1]

\draw[draw=none,fill=gray!30,opacity=0.5] (0,6.5) -- (1,0.75) -- (2.5,0) -- (4,0) -- (5.75,1.16666666666) -- (6,6.5);
\draw[draw=none,fill=gray!80,opacity=0.5] (1.5,6) -- (1.5,0.5) -- (2.5,0) -- (4,0) -- (5.5,1) -- (5.5,4) -- (5.5,6);
\draw[line width=1pt] (0,6.5) -- (1,0.75) -- (2.5,0) -- (4,0) -- (5.75,1.16666666666) -- (6,6.5);
\draw[line width=1pt] (0.1,6) -- (5.98,6);
\draw[line width=1pt] (1.5,0.5) -- (1.5,6);
\draw[line width=1pt] (5.5,1) -- (5.5,6);

\draw[line width=0.75pt,dashed] (-0.5,6) -- (6.5,6);
\draw[line width=0.75pt,dashed] (-0.5,1.5) -- (6.5,1.5);
\draw[line width=0.75pt,dashed] (3.25,0) -- (3.25,6.5);
\draw[line width=0.75pt,dashed] (1.5,0) -- (1.5,6.5);
\draw[line width=0.75pt,dashed] (5.5,0) -- (5.5,6.5);

\draw[line width=1.5pt] (2,6.5) .. controls (2,0) and (2.5,0) .. (4.5,6.5);
\draw[line width=1.5pt] (3.25,0) .. controls (3.2,0.4) and (4,0.6) .. (4.8,1.1) .. controls (5.1,1.3) and (4.9,1.332) .. (4.8,1.35) .. controls (4.3,1.44) and (3.3,1.3) .. (3.25,0);

\draw[->] (0,-0.1) -- (0,6.75);
\draw[->] (-0.1,0) -- (6.5,0);
\draw (-0.1,2) -- (0.1,2);
\draw (-0.1,3) -- (0.1,3);
\draw (-0.1,4) -- (0.1,4);

\draw[fill=black] (1,0.75) circle (0.75mm);
\draw[fill=black] (2.5,0) circle (0.75mm);
\draw[fill=gray,draw] (3.25,0) circle (0.75mm);
\draw[fill=black] (4,0) circle (0.75mm);
\draw[fill=black] (5.75,1.16666666666) circle (0.75mm);
\draw[fill=gray,draw] (1.5,0.5) circle (0.75mm);
\draw[fill=gray,draw] (5.5,1) circle (0.75mm);
\draw[fill=gray,draw] (1.5,6) circle (0.75mm);
\draw[fill=gray,draw] (5.5,6) circle (0.75mm);

\draw(7.2,1.5) node{\small${\tt v}=\frac{M}{4}$}; 
\draw(7.2,6) node{\small${\tt v}=M$};
\draw(3.25,6.6) node{\small${\tt u}=a$};
\draw(3.5,-0.25) node{\small$q$};
\draw(-0.3,0) node{\small$0$};
\draw(-0.3,2) node{\small$\frac{M}{3}$};
\draw(-0.3,3) node{\small$\frac{M}{2}$};
\draw(-0.35,4) node{\small$\frac{2M}{3}$};
\draw(1.5,-0.25) node{\small${\tt u}=0$};
\draw(5.5,-0.25) node{\small${\tt u}=r+1$};
\draw(4.1,1) node{\small$Z_1$};
\draw(2.5,5) node{\small$Z_\ell$};
\draw(1,5.5) node{\small$\Pp$};
\draw(5,5) node{\small$\Pp_0'$};
\draw(4.5,4) node{\small$Y_a\cap\Pp_0$};

\end{tikzpicture}
\end{center}
\caption{Description of the fake situation.}\label{fig7}
\end{figure}
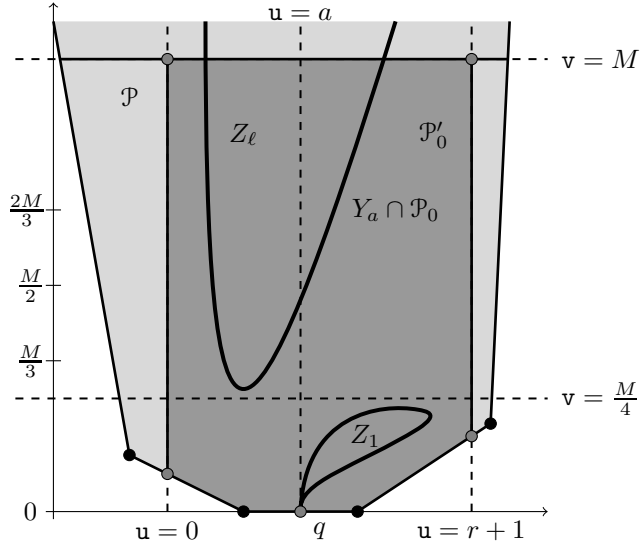

\noindent\em Step \em1. \em The line $\{\u=0\}\subset\{\alpha(\u,\vv)\u-a<0\}$ and the line $\{\u=r+1\}\subset\{\alpha(\u,\vv)\u-a>0\}$\em. 

The first inclusion is clear. To prove the second denote $\zeta(\u,\vv):=\alpha(\u,\vv)\u-a$ and observe that if $\zeta(r+1,v)\leq0$, then $\alpha(r+1,v)<1$. As $P^2$ is the square of a polynomial, we deduce $\polq{\hat{\setg{g}}_\veps}(y_0,\tfrac{z_0}{\|z_0\|}(r+1),v)<0$. By Lemma~\ref{polq1} we have 
$$
\big(y_0,\tfrac{z_0}{\|z_0\|}(r+1),0\big)\in\seta{\hat{\setg{g}}_\veps}. 
$$
In particular, $\tfrac{z_0}{\|z_0\|}(r+1)\in I_\veps^{n-d-1}$ and 
$$
r+1\leq\|(\veps,\overset{(n-d-1)}{\ldots},\veps)\|=\veps\sqrt{n-d-1}\leq r,
$$ 
a contradiction. Consequently, $\{\u=r+1\}\subset\{\alpha(\u,\vv)\u-a>0\}$.

\noindent\em Step \em2. Let $M>0$ be such that all the vertices of $\p_0$ and all the upperly bounded connected components of $Y_a\cap\Pp_0$ are contained in $\{\vv<\tfrac{M}{4}\}$. Consider the compact convex polygon $\p_0':=\p_0\cap\{\vv\leq M\}\subset[0,r+1]\times[0,M]$. Let $Z_1,\ldots,Z_\ell$ be the connected components of $Y_a\cap\Pp_0$. Suppose that none of them meets both $\partial\p_0$ and $\{\vv=M\}$ (see Figure~\ref{fig7}). Assume $Z_1,\ldots,Z_s$ are the upperly bounded connected components of $Y_a\cap\Pp_0$ and $Z_{s+1},\ldots,Z_\ell$ are the remaining ones, so they meet $\{\vv=M\}$ but they do not meet $\partial\p_0$. Define $K_1:=(\partial\p_0\cap\{\vv\leq \tfrac{M}{4}\})\cup\bigcup_{i=1}^s(Z_i\cap\p_0')$ and $K_2:=\partial\p_0'\cup
\bigcup_{i=s+1}^\ell(Z_i\cap\p_0')$ (see Figures~\ref{fig:pos1} and \ref{fig:pos2}). Observe that $K_1\cap K_2=\partial\p_0\cap\{\vv\leq \tfrac{M}{4}\}$ is connected. Consider the positive real number
$$
\epsilon:=\min\big\{\dist\big(K_1,\partial\p_0'\cap\big\{\vv\geq \tfrac{M}{3}\big\}\big),\dist\big(\textstyle\bigcup_{i=s+1}^\ell(Z_i\cap\p_0'),\partial\p_0'\cap\big\{\vv\leq\tfrac{2M}{3}\big\}\big)
\big\}.
$$
Let $0<\rho<\tfrac{\epsilon}{2}$ be such that $q_1:=(\rho,\tfrac{M}{2})\in\{\zeta<0\}$ and $q_2:=(r+1-\rho,\tfrac{M}{2})\in\{\zeta>0\}$. 

We have $\zeta(q_1)<0$ and $\zeta(q_2)>0$. Consequently, \em $K_1\cup K_2=(Y_a\cap\p_0')\cup\partial\p_0'$ separates the points $q_1$ and $q_2$\em.

\noindent
\begin{figure}[ht]
\begin{minipage}[t]{0.49\textwidth}
\begin{tikzpicture}[scale=1]

\draw[draw=none,fill=gray!30,opacity=0.5] (0,6.5) -- (1,0.75) -- (2.5,0) -- (4,0) -- (5.75,1.16666666666) -- (6,6.5);
\draw[draw=none,fill=gray!80,opacity=0.5] (1.5,6) -- (1.5,0.5) -- (2.5,0) -- (4,0) -- (5.5,1) -- (5.5,4) -- (5.5,6);
\draw[draw=none,fill=gray!90,opacity=0.75] (1.5,2) -- (1.85,2) -- (1.85,5.65) -- (5.15,5.65) -- (5.15,2) -- (5.5,2) -- (5.5,6) -- (1.5,6) -- (1.5,2);
\draw[line width=1pt] (0,6.5) -- (1,0.75) -- (2.5,0) -- (4,0) -- (5.75,1.16666666666) -- (6,6.5);
\draw[line width=1pt] (0.1,6) -- (5.98,6);
\draw[line width=1pt] (1.5,0.5) -- (1.5,6);
\draw[line width=1pt] (5.5,1) -- (5.5,6);

\draw[line width=0.75pt,dashed] (-0.5,6) -- (6.5,6);
\draw[line width=0.75pt,dashed] (-0.5,1.5) -- (6.5,1.5);
\draw[line width=0.75pt,dashed] (3.25,0) -- (3.25,6.5);
\draw[line width=0.75pt,dashed] (1.5,0) -- (1.5,6.5);
\draw[line width=0.75pt,dashed] (5.5,0) -- (5.5,6.5);

\draw[line width=1.5pt] (3.25,0) .. controls (3.2,0.4) and (4,0.6) .. (4.8,1.1) .. controls (5.1,1.3) and (4.9,1.332) .. (4.8,1.35) .. controls (4.3,1.44) and (3.3,1.3) .. (3.25,0);
\draw[line width=1.5pt] (1.5,1.5) -- (1.5,0.5) -- (2.5,0) -- (4,0) -- (5.5,1) -- (5.5,1.5);

\draw[->] (0,-0.1) -- (0,6.75);
\draw[->] (-0.1,0) -- (6.5,0);
\draw (-0.1,2) -- (0.1,2);
\draw (-0.1,3) -- (0.1,3);
\draw (-0.1,4) -- (0.1,4);

\draw[fill=black] (1,0.75) circle (0.75mm);
\draw[fill=black] (2.5,0) circle (0.75mm);
\draw[fill=gray,draw] (3.25,0) circle (0.75mm);
\draw[fill=black] (4,0) circle (0.75mm);
\draw[fill=black] (5.75,1.16666666666) circle (0.75mm);
\draw[fill=black] (1.675,3) circle (0.5mm);
\draw[fill=black] (5.325,3) circle (0.5mm);
\draw[fill=gray,draw] (1.5,0.5) circle (0.75mm);
\draw[fill=gray,draw] (5.5,1) circle (0.75mm);
\draw[fill=gray,draw] (1.5,6) circle (0.75mm);
\draw[fill=gray,draw] (5.5,6) circle (0.75mm);

\draw(3.25,6.6) node{\small${\tt u}=a$};
\draw(3.5,-0.25) node{\small$q$};
\draw(-0.3,0) node{\small$0$};
\draw(-0.3,2) node{\small$\frac{M}{3}$};
\draw(-0.3,3) node{\small$\frac{M}{2}$};
\draw(-0.35,4) node{\small$\frac{2M}{3}$};
\draw(1.5,-0.25) node{\small${\tt u}=0$};
\draw(5.5,-0.25) node{\small${\tt u}=r+1$};
\draw(4.1,1) node{\small$Z_1$};
\draw(1,5.5) node{\small$\Pp$};
\draw(4.75,5) node{\small$\Pp_0'$};
\draw(2.1,1) node{\small$K_1$};
\draw(4.1,5.625) node{\small$W_2$};
\draw(1.25,3) node{\small$q_1$};
\draw(5,3) node{\small$q_2$};

\end{tikzpicture}
\captionof{figure}{Positions of $K_1$ and $W_2$}\label{fig:pos1}

\end{minipage}
\hfill
\begin{minipage}[t]{0.49 \textwidth}
\begin{tikzpicture}[scale=1]

\draw[draw=none,fill=gray!30,opacity=0.5] (0,6.5) -- (1,0.75) -- (2.5,0) -- (4,0) -- (5.75,1.16666666666) -- (6,6.5);
\draw[draw=none,fill=gray!80,opacity=0.5] (1.5,6) -- (1.5,0.5) -- (2.5,0) -- (4,0) -- (5.5,1) -- (5.5,4) -- (5.5,6);
\draw[draw=none,fill=gray!90,opacity=0.75] (1.5,4) -- (1.85,4) -- (1.85,0.7) -- (2.5,0.35) -- (4,0.35) -- (5.15,1.1) -- (5.15,4) -- (5.5,4) -- (5.5,1) -- (4,0) -- (2.5,0) -- (1.5,0.5) -- (1.5,4);
\draw[line width=1pt] (0,6.5) -- (1,0.75) -- (2.5,0) -- (4,0) -- (5.75,1.16666666666) -- (6,6.5);
\draw[line width=1pt] (0.1,6) -- (5.98,6);
\draw[line width=1pt] (1.5,0.5) -- (1.5,6);
\draw[line width=1pt] (5.5,1) -- (5.5,6);

\draw[line width=0.75pt,dashed] (-0.5,6) -- (6.5,6);
\draw[line width=0.75pt,dashed] (-0.5,1.5) -- (6.5,1.5);
\draw[line width=0.75pt,dashed] (3.25,0) -- (3.25,6.5);
\draw[line width=0.75pt,dashed] (1.5,0) -- (1.5,6.5);
\draw[line width=0.75pt,dashed] (5.5,0) -- (5.5,6.5);

\draw[line width=1.5pt] (2,6.5) .. controls (2,0) and (3,0) .. (4.5,6.5);
\draw[line width=1.5pt] (1.5,6) -- (1.5,0.5) -- (2.5,0) -- (4,0) -- (5.5,1) -- (5.5,6) -- (1.5,6);

\draw[->] (0,-0.1) -- (0,6.75);
\draw[->] (-0.1,0) -- (6.5,0);
\draw (-0.1,2) -- (0.1,2);
\draw (-0.1,3) -- (0.1,3);
\draw (-0.1,4) -- (0.1,4);

\draw[fill=black] (1,0.75) circle (0.75mm);
\draw[fill=black] (2.5,0) circle (0.75mm);
\draw[fill=gray,draw] (3.25,0) circle (0.75mm);
\draw[fill=black] (4,0) circle (0.75mm);
\draw[fill=black] (5.75,1.16666666666) circle (0.75mm);
\draw[fill=black] (1.675,3) circle (0.5mm);
\draw[fill=black] (5.325,3) circle (0.5mm);
\draw[fill=gray,draw] (1.5,0.5) circle (0.75mm);
\draw[fill=gray,draw] (5.5,1) circle (0.75mm);
\draw[fill=gray,draw] (1.5,6) circle (0.75mm);
\draw[fill=gray,draw] (5.5,6) circle (0.75mm);

\draw(3.25,6.6) node{\small${\tt u}=a$};
\draw(3.5,-0.25) node{\small$q$};
\draw(-0.3,0) node{\small$0$};
\draw(-0.3,2) node{\small$\frac{M}{3}$};
\draw(-0.3,3) node{\small$\frac{M}{2}$};
\draw(-0.35,4) node{\small$\frac{2M}{3}$};
\draw(1.5,-0.25) node{\small${\tt u}=0$};
\draw(5.5,-0.25) node{\small${\tt u}=r+1$};
\draw(2.5,5) node{\small$Z_\ell$};
\draw(1,5.5) node{\small$\Pp$};
\draw(5,5) node{\small$\Pp_0'$};
\draw(4.1,3.75) node{\small$K_2$};
\draw(2.75,0.4) node{\small$W_1$};
\draw(1.25,3) node{\small$q_1$};
\draw(4.9,3) node{\small$q_2$};

\end{tikzpicture}
\captionof{figure}{Positions of $K_2$ and $W_1$}\label{fig:pos2}

\end{minipage}
\end{figure}

\noindent\em Step \em3. Let us check: \em neither $K_1$ nor $K_2$ separates the points $q_1$ and $q_2$\em. 

The points $q_1,q_2$ belong to both open connected subsets
\begin{align*}
&W_1:=\big\{p\in\Int(\p_0'):\ 0<\dist\big(p,\big(\partial\p_0'\cap\big\{\vv\leq\tfrac{2M}{3}\big\}\big)\big)<\tfrac{\epsilon}{2}\big\},\\
&W_2:=\big\{p\in\Int(\p_0'):\ 0<\dist\big(p,\big(\partial\p_0'\cap\big\{\vv\geq\tfrac{M}{3}\big\}\big)\big)<\tfrac{\epsilon}{2}\big\}
\end{align*}
of $\Int(\p_0')$ whereas $K_1\cap W_2=\varnothing$ and $K_2\cap W_1=\varnothing$. Thus, $q_1,q_2$ are separated neither by $K_1$ nor by $K_2$, which contradicts Janiszewski's Theorem. 

\noindent\em Step \em4. Consequently, \em there exists a connected component $Z_j$ of $Y_a\cap\p_0$ that meets both $\partial\p_0$ and the line $\{\vv=M\}$\em, as shown in Figure~\ref{fig8}. As all the upperly bounded connected components are contained in $\{\vv<\tfrac{M}{4}\}$, we deduce $Z_j$ is upperly unbounded. In addition, $Y_a\cap\partial\Pp_0=\{q\}$, so $q\in Z_j$ and 
$$
Z_j\setminus\{q\}\subset Z_j\setminus\partial\p_0\subset\p\setminus\partial\p=\Int(\p)\subset\Int(\pol), 
$$
as claimed in \ref{4b37}. 

\paragraph{} We are ready to finish the proof of \ref{claim4}. By \cite[2.9.10]{bcr} $Z_j$ is the union of a finite set ${\mathfrak F}$ and finitely many Nash paths $\Gamma_i$ that are Nash diffeomorphic to ${]0,1[}$. We may assume that $\Gamma_1$ is upperly unbounded. Let $q'\in\cl(\Gamma_1)\setminus\Gamma_1$ and let $\gamma_0:[0,1]\to Z_j$ be a semialgebraic path such that $\gamma_0(0)=q$ and $\gamma_0(1)=q'$. Let $\gamma_1:{[1,+\infty[}\to\Gamma_1\cup\{q'\}$ be a semialgebraic parameterization such that $\gamma_1(1)=q'$ and define
$$
\beta:=(\beta_1,\beta_2):{[0,+\infty[}\to Z_j,\ t\mapsto
\begin{cases}
\gamma_0(t)&\text{if $t\in[0,1]$,}\\
\gamma_1(t)&\text{if $t\in{[1,+\infty[}$.}
\end{cases}
$$
As $\Gamma_1$ is upperly unbounded, $\lim_{t\to\infty}\beta_2(t)=+\infty$. We have
$$
G_k\circ\beta(t)=\Big(a,\Big(\Big(\frac{a}{\beta_1(t)}\Big)+\Big(\frac{a}{\beta_1(t)}\Big)^k\Big)\frac{\beta_2(t)}{2}\Big)
$$
As $0\leq \beta_1(t)\leq r+1$ for $t\in{[0,+\infty[}$, we have 
$$
0<\frac{a}{r+1}+\frac{a^k}{(r+1)^k}\leq\Big(\frac{a}{\beta_1(t)}\Big)+\Big(\frac{a}{\beta_1(t)}\Big)^k
$$
for $t\in{[0,+\infty[}$. Consequently,
\begin{equation}\label{limit}
\lim_{t\to\infty}\Big(\Big(\Big(\frac{a}{\beta_1(t)}\Big)+\Big(\frac{a}{\beta_1(t)}\Big)^k\Big)\frac{\beta_2(t)}{2}\Big)=+\infty.
\end{equation}
As $x_0=(a,b)_\Rr\in\Int(\p_0)\cap \vspanp{q}{\ven}$, $q=(a,c)_\Rr\in\partial\p_0$ and $\vec{\tt e}_n\in\conv{\p_0}{}$, we have $c<b$. As $q=G_k(q)=(G_k\circ\beta)(0)$ and using \eqref{limit}, there exists $t_0\in{]0,+\infty[}$ such that $(G_k\circ\beta)(t_0)=(a,b)_{\Rr}=x_0$, so there exists $x_1:=\beta(t_0)\in Z_j\setminus\{q\}\subset\Int(\pol)$ such that $G_k(x_1)=x_0$ and \ref{claim4} holds.

\begin{figure}[!ht]
\begin{center}
\begin{tikzpicture}[scale=1]

\draw[draw=none,fill=gray!30,opacity=0.5] (0,6.5) -- (1,0.75) -- (2.5,0) -- (4,0) -- (5.75,1.16666666666) -- (6,6.5);
\draw[draw=none,fill=gray!80,opacity=0.5] (1.5,6) -- (1.5,0.5) -- (2.5,0) -- (4,0) -- (5.5,1) -- (5.5,4) -- (5.5,6);
\draw[line width=1pt] (0,6.5) -- (1,0.75) -- (2.5,0) -- (4,0) -- (5.75,1.16666666666) -- (6,6.5);
\draw[line width=1pt] (0.1,6) -- (5.98,6);
\draw[line width=1pt] (1.5,0.5) -- (1.5,6);
\draw[line width=1pt] (5.5,1) -- (5.5,6);

\draw[line width=0.75pt,dashed] (-0.5,6) -- (6.5,6);
\draw[line width=0.75pt,dashed] (3.25,0) -- (3.25,6.5);
\draw[line width=0.75pt,dashed] (1.5,0) -- (1.5,6.5);
\draw[line width=0.75pt,dashed] (5.5,0) -- (5.5,6.5);

\draw[line width=1.5pt] (3.25,0) .. controls (3.2,1) and (4,2) .. (4.8,3.25) .. controls (4.95,3.5) and (4.9,3.7) .. (4.8,3.75) .. controls (4.3,4) and (3.3,3) .. (5,6.5);

\draw[->] (-0.1,0) -- (6.5,0);

\draw[fill=black] (1,0.75) circle (0.75mm);
\draw[fill=black] (2.5,0) circle (0.75mm);
\draw[fill=gray,draw] (3.25,0) circle (0.75mm);
\draw[fill=black] (4,0) circle (0.75mm);
\draw[fill=black] (5.75,1.16666666666) circle (0.75mm);
\draw[fill=gray,draw] (1.5,0.5) circle (0.75mm);
\draw[fill=gray,draw] (5.5,1) circle (0.75mm);
\draw[fill=gray,draw] (1.5,6) circle (0.75mm);
\draw[fill=gray,draw] (5.5,6) circle (0.75mm);

\draw(7.1,6) node{\small${\tt v}=M$};
\draw(3.25,6.6) node{\small${\tt u}=a$};
\draw(3.5,-0.25) node{\small$q$};
\draw(1.5,-0.25) node{\small${\tt u}=0$};
\draw(5.5,-0.25) node{\small${\tt u}=r+1$};
\draw(4.1,2.75) node{\small$Z_1$};
\draw(1,5.5) node{\small$\Pp$};
\draw(5,5) node{\small$\Pp_0'$};

\end{tikzpicture}
\end{center}
\caption{Description of the authentic situation.}\label{fig8}
\end{figure}
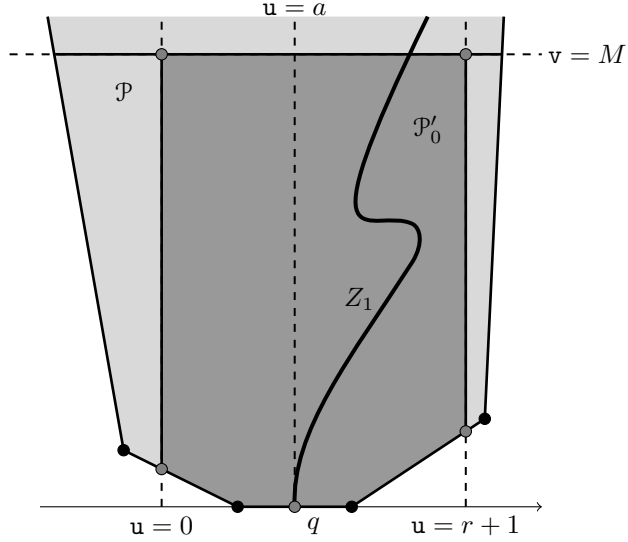

\subsubsection{Conclusion of proof for Proposition~\em\ref{main12}} Let $\Ee_1,\ldots,\Ee_m$ be all the faces of $\pol$ of dimension $\leq n-2$. We use the symbol $\sqcup$ to stress unions that involve only pairwise disjoint sets. We have $\pol\cap X=\bigsqcup_{i=1}^m\Int(\Ee_i)$. By \ref{clue} there exist polynomial maps $g_\ell:\R^n\to\R^n$ such that
$$
g_\ell\Big((\pol\setminus X)\sqcup\bigsqcup_{i=1}^{\ell-1}\Int(\Ee_i)
\Big)=(\pol\setminus X)\sqcup\bigsqcup_{i=1}^\ell\Int(\Ee_i)
$$
for $\ell=1,\ldots,m$. Consequently, $g:=(g_m\circ\cdots\circ g_1):\R^n\to\R^n$ satisfies $g(\pol\setminus X)=(\pol\setminus X)\sqcup(\pol\cap X)=\pol$, as required.
\qed

\subsection{Proof of Corollary \ref{main13}}
By Propositions \ref{main11} and \ref{main12} there exists a polynomial map $f_0:\R^n\to\R^n$ such that $f_0(\R^n)=\pol$. Let $x':=(x_1,\ldots,x_{n-1})$ and consider the polynomial map
\begin{align*}
f_1:\R^{n+1}&\to\R^{n+1},\\
x:=(x',x_n,x_{n+1})&\mapsto(x',x_{n+1}(x_nx_{n+1}-1),(x_nx_{n+1}-1)^2+x_n^2),
\end{align*}
whose image is $\{\x_{n+1}>0\}$ (see \cite[Ex. 1.4 (iv)]{fg1}). Assume that $\ven\in\conv{\pol}{}$ and let
$$
f_2:\R^{n+1}\to\R^n,\ (x_1,\ldots,x_n,x_{n+1})\mapsto f_0(x_1,\ldots,x_n)+x_{n+1}\ven.
$$
We have $f_2(\{\x_{n+1}>0\})=\Int(\pol)$, so $\Int(\pol)$ is the image of the polynomial map $f:=(f_2\circ f_1):\R^{n+1}\to\R^n$, as required.
\qed

\section{Interiors of convex polyhedra as polynomial images of $\R^n$}\label{s4}

In this section we prove Theorem~\ref{main2}. Each degenerate convex polyhedron $\pol\subset\R^n$ can be written in suitable coordinates as $\pol=\p\times\R^k$ where $\p$ is a non-degenerate convex polyhedron and $\Int(\pol)=\Int(\p)\times\R^k$. If $k\geq n-1$, then $\Int(\pol)$ is either $\R^n$ or an open half-space. The second case is a polynomial image of $\R^n$ by \cite[Ex. 1.4 (iv)]{fg1}. Thus, we will prove Theorem~\ref{main2} assuming in addition that the convex polyhedron $\pol$ is non-degenerate and has dimension $n\geq2$.

The general strategy is the following. By Proposition~\ref{main11} we know that if $\pol\subset\R^n$ is an unbounded non-degenerate convex polyhedron with $n$-dimensional recession cone $\conv{\pol}{}$ and $X$ is the union of the affine subspaces of $\R^n$ spanned by the faces of $\pol$ of dimension $n-2$, then $\pol\setminus X$ is a polynomial image of $\R^n$. For each unbounded facet $\Ff$ of $\pol$ we devise a procedure to `erase' it from $\pol\setminus X$ in two steps: (1) first we `push' $\Int(\Ff)$ `inside' $\Int(\pol)$ to obtain a polynomial image $\Ss$ of $\R^n$ contained in $\pol\setminus (X\cup\Ff)$ and (2) we fill the gap left between $\Ss$ and $\Ff$ to express $\pol\setminus(X\cup \Ff)$ as a polynomial image of $\R^n$. We `erase' all the facets of the initial image $\pol\setminus X$ to obtain $\Int(\pol)$ as a polynomial image of $\R^n$. To lighten the proof of Theorem~\ref{main2} we develop next some preliminary work. In the following we write $x'':=(x_1,\ldots,x_{n-2})$, $x':=(x'',x_{n-1})$ and $x:=(x',x_n)$.

\subsection{Preliminary construction}
We first introduce the type of polynomial maps that will allow us to push the interior of a given facet $\Ff$ of a convex polyhedron $\pol$ inside $\Int(\pol)$.

\begin{lem}\label{allplanos2}
Let $\pol\subset\R^n$ be a convex polyhedron of dimension $n$ and let $\Ff$ be a facet of $\pol$. Assume $\Ff\subset\{\x_{n-1}=0\}$, $\pol\subset\{\x_{n-1}\geq0\}$ and $\vem\in\conv{\pol}{}$. Let $\Tt$ be a semialgebraic set obtained by removing the interior of some facets of $\pol$ from $\pol\setminus X$ and let $F\in\R[\x]$ be a polynomial such that $\{F=0\}\cap\Int(\Ff)=\varnothing$ and $F$ is identically zero on the facets of $\pol$ different from $\Ff$. Consider the semialgebraic set $\Rr:=\{F=0\}\cap\{\x_{n-1}>0\}$ and the polynomial map $f_0:\R^n\to\R^n,\ x\mapsto x+F^2(x)\vem$. Then
\begin{itemize}
\item[(i)] $f_0(\Int(\vspanp{x_0}{\vem}))=\Int(\vspanp{x_0}{\vem})$ and $f_0(x_0)=x_0$ for each $x_0\in\{F=0\}$.
\item[(ii)] $\Tt\cap\vspanp{\Rr}{\vem}\subset f_0(\Tt)\subset\Tt\setminus\Ff$.
\end{itemize}
\end{lem}
\begin{proof}
(i) Write $x_0:=(x_{01},\ldots,x_{0n})$. Consider the continuous function 
$$
\psi:\R\to\R,\ t\mapsto t+F^2(x_0+t\vec{\tt e}_{n-1}).
$$
As $\psi(0)=0$ and $\psi(t)\geq t$ for each $t\geq0$, we have $\psi({]0,+\infty[})={]0,+\infty[}$, so $f_0(\Int(\vspanp{x_0}{\vem}))=\Int(\vspanp{x_0}{\vem})$ and $f_0(x_0)=x_0$.

(ii) Observe that $f_0(\vspanp{x}{\vem})\subset\vspanp{x}{\vem}$ for each $x\in\R^n$. As $\vem\in\conv{\pol}{}$ and $\tt\subset\pol$, we deduce $f_0(\Tt)\subset\Tt$. In addition, if $x:=(x'',x_{n-1},x_n)\in\Tt$ is such that $f_0(x)=x+F^2(x)\vem\in\Ff$, then $x_{n-1}=0$ and $F(x)=0$, so 
$$
x\in(\pol\setminus X)\cap\{\x_{n-1}=0\}\cap\{F=0\}=\Int(\Ff)\cap\{F=0\}=\varnothing,
$$
which is a contradiction. Thus, $f_0(\Tt)\subset\Tt\setminus\Ff$. 

Let us prove now $\Tt\cap\vspanp{\Rr}{\vem}\subset f_0(\Tt)$. Pick a point $x\in\Tt\cap\vspanp{\Rr}{\vem}$ and write $x=y+\lambda\vem$ where $y\in\Rr$ and $\lambda\ge 0$. Observe that $\vspan{x}{\vem}\cap\pol=\vspan{z}{\vem}$ where either $z$ belongs to a facet of $\pol$ different from $\Ff$ or $z:=(x'',0,x_n)\in\Int(\Ff)$. In the first case $F(z)=0$, so by (i) $f_0(z)=z$ and $x\in\vspanp{z}{\vem}\cap\Tt=f_0(\vspanp{z}{\vem}\cap\Tt)$. In the second case, $\vspanp{y}{\vem}\subset\vspanp{z}{\vem}\cap\Tt$. As $F(y)=0$, we have by (i) $f_0(y)=y$ and $x\in\vspanp{y}{\vem}=f_0(\vspanp{y}{\vem})\subset f_0(\Tt)$, as required. 
\end{proof}

In order to take advantage of Lemma~\ref{allplanos2} we need a polynomial $h\in\R[\x]$ with some added specific characteristics, that we proceed to describe below.

\subsection{Pushing an open facet inside the interior of a convex polyhedron}\label{preparat} Assume now that a convex, unbounded polyhedron $\pol$ with $n$-dimensional recession cone is placed in $\R^n$ so that $\vec{\tt e}_{n-1},\ven\in\conv{\pol}{}$, $\Ff=\{\x_{n-1}=0\}\cap\pol\subset\{\x_n>0\}$ and $\pol\subset\{\x_{n-1}\geq0\}$. Observe that $\ven\in\conv{\Ff}{}$. Denote the facets of $\pol$ with $\Ff_1,\ldots,\Ff_e$ and let $h_i=0$ be a non-zero linear equation of the hyperplane $H_i$ spanned by $\Ff_i$. Suppose $\pol=\{h_1\geq0,\ldots,h_e\geq0\}$,
\begin{itemize}
\item $\Ff_1,\ldots,\Ff_s$ are non-vertical and among them $\Ff_1,\ldots,\Ff_r$ are those non-vertical facets of $\pol$ that meet $\Ff$, 
\item $\Ff_{s+1},\ldots,\Ff_e$ are vertical and $\Ff_e=\Ff$.
\end{itemize}
As $\ven\in\conv{\pol}{}$, we may assume $\vec{h}_i(\ven)=1$ for $i=1,\ldots,r$, so that $h_i=h_i(\x',0)+\x_n$. Define
\begin{align}
b_i&:=h_i-\x_{n-1}=h_i(\x',0)-\x_{n-1}+\x_n,\\
b_i'&:=h_i-2\x_{n-1}=h_i(\x',0)-2\x_{n-1}+\x_n.
\end{align}
The hyperplanes $B_i:=\{b_i=0\}$ and $B_i':=\{b_i'=0\}$ separate by Lemma~\ref{sep} the facets $\Ff_i$ and $\Ff$ and meet $\Int(\pol)$. Consider now the affine change of coordinates
\begin{equation}\label{iso}
\phi_i:\R^n\to\R^n,\ x:=(x',x_n)\mapsto(x',x_n+h_i(x',0)),
\end{equation}
which satisfies $B_i^*:=\phi_i(B_i)=\{\x_n-\x_{n-1}=0\}$ and $B_i'^*:=\phi_i(B_i')=\{\x_n-2\x_{n-1}=0\}$.

\subsubsection{}Denote the union of all the facets of $\pol$ that do not meet $\Ff$ with $\Gg$. By Lemma~\ref{dist} and Corollary~\ref{dist2} there exists $\veps_0\in\R$ such that 
\begin{align}
0<\veps_0<\min\{1,
\dist(\{\x_{n-1}=0\},\Gg)\}\quad\text{and}\quad
\pol\cap\{\x_{n-1}\le \veps_0\}\subset\{\x_n>0\}.
 \end{align} 
As each $B_i\cap\pol\cap\{\x_{n-1}\le\veps_0\}\subset\{\x_n>0\}$, there exists by Lemma~\ref{dist} $\delta>0$ such that $B_i\cap\pol\cap\{\x_{n-1}\le\veps_0\}\subset\{\x_n>\delta\}$ for $i=1,\ldots,r$. Set $\veps:=\min\{\veps_0,\frac{\delta}{2}\}>0$. 

\subsubsection{}\label{bibi}
Define $\pol_0:=\pol\cap\{\x_{n-1}\leq\veps\}$ and observe that $\Int(\pol_0)=\Int(\pol)\cap\{\x_{n-1}<\veps\}$. Consider the family of hyperplanes containing the non-vertical facets of $\pol$ together with all hyperplanes $B_i'$. By Proposition~\ref{polq2} there exists a polynomial $G_i$ such that 
\begin{equation}\label{ajustefino3}
\{\x_n\ge G_i\}\subset\bigcap_{j=1}^r\{b_j'\circ\phi_i^{-1}>1\}\cap\bigcap_{k=1}^s\{h_k\circ\phi_i^{-1}>1\}\subset\{b_i'\circ\phi_i^{-1}>1\}=\{\x_n>2\x_{n-1}+1\}.
\end{equation}
Define $\Bb_i:=B_i\cap\Int(\pol_0)$. We claim: 
\begin{equation}\label{bibieq}
\vspan{\Bb_i}{\ven}\cap\Int(\pol)\subset\vspanp{\Bb_i}{\ven}\cup(\{b_i'\leq0\}\cap\Int(\pol_0))\subset\Int(\pol_0).
\end{equation} 

As $\ven\in\conv{\pol}{}$ and $\Bb_i\subset\Int(\pol_0)$, we have 
\begin{equation}\label{bibieq2}
\vspanp{\Bb_i}{\ven}\subset\Int(\pol_0). 
\end{equation}
In addition, $\vspanp{\Bb_i}{(-\ven)}\subset\{b_i\leq0\}\subset\{b_i'\leq0\}$. Consequently,
\begin{equation*}
\begin{split}
\vspan{\Bb_i}{\ven}\cap\Int(\pol)&=(\vspanp{\Bb_i}{\ven}\cup\vspanp{\Bb_i}{(-\ven)})\cap\Int(\pol_0)\\
&=\vspanp{\Bb_i}{\ven}\cup(\vspanp{\Bb_i}{(-\ven)}\cap\Int(\pol_0))\subset\vspanp{\Bb_i}{\ven}\cup(\{b_i'\leq0\}\cap\Int(\pol_0))\subset\Int(\pol_0).
\end{split}
\end{equation*}

\subsubsection{}Write $\pi_n(\Bb_i)=\{g_{i,1}>0,\dots,g_{i,m}>0\}$ where each $g_{i,j}\in\R[\x']$ is a polynomial of degree one. We may assume $g_{i,1}=\x_{n-1}$. Consider the admissible tuple $\setg{g}_i:=(g_{i,1},\dots,g_{i,m},g_{i,m+1})$ where $g_{i,m+1}\in\R[\x']$ is a polynomial satisfying
\begin{equation}\label{ajustefino4}
g_{i,m+1}>\max\Big\{G_i,1+|h_i(\x',0)|\sqrt{|g_{i,1}\cdots g_{i,m}|},\ i=1,\dots,r\Big\}
\end{equation}
and the associated semialgebraic sets $\seta{\setg{g}_i}=\pi_n(\Bb_i)$ and $\sets{\setg{g}_i}\subset\vspanp{\Bb_i}{\ven}$. In addition, by \eqref{ajustefino3} we have $g_{i,m+1}\geq G_i\geq2\x_{n-1}$. 

\subsubsection{}\label{bii}We claim: \em $h_i(\x',0)<0$ on $\seta{\setg{g}_i}\subset\pi_n(B_i\cap\pol_0)$\em.

Pick a point $x:=(x',x_n)\in B_i\cap\pol_0$. Then $h_i(x',0)=x_{n-1}-x_n<\veps-\delta<0$.

\subsubsection{} By the choice of $\veps>0$ the non-vertical facets of $\pol_0$ are $\Ff_{i0}:=\Ff_i\cap\{\x_{n-1}\le\veps\}$ for $i=1,\ldots,r$ and all of them meet the facet $\Ff$ of $\pol$. By Lemma~\ref{allplanos0}
\begin{equation}\label{7688}
\vspan{\Int(\pol)}{\ven}\cap\{\x_{n-1}<\veps\}=\vspan{\Int(\pol_0)}{\ven}=\bigcup_{i=1}^r\vspan{(B_i\cap\Int(\pol_0))}{\ven}=\bigcup_{i=1}^r\vspan{\Bb_i}{\ven}=\bigcup_{i=1}^r\vspan{\seta{\setg{g}_i}}{\ven}.
\end{equation}

\subsubsection{}\label{allplanos}
Denote $\setst{\setg{g_i}}:=\phi_i(\sets{\setg{g}_i})
=\{(x',x_n+h_i(x',0)):\ (x',x_n)\in\sets{\setg{g_i}}\}$ for $i=1,\ldots,r$. \em Then there exists a polynomial $P\in\R[\x'',\x_n]$ with empty zero-set such that the zero-set $\Gamma$ of the polynomial $R(\x):=\x_{n-1}P(\x'',\x_n)-1$
satisfies $\Gamma\subset\{0<\x_{n-1}<\veps\}$, $\vspanp{\Gamma}{\vem}\subset\{R\geq0\}$ and 
\begin{equation}\label{si}
\sets{\setg{g}_i}\subset\setst{\setg{g}_i}\subset\vspanp{\Gamma}{\vem}\cap\{R>1\}.
\end{equation}
\em 

\begin{proof}
The inclusion $\sets{\setg{g}_i}\subset\setst{\setg{g}_i}$ holds because by \ref{bii} $h_i(x',0)\le 0$ on $\seta{\setg{g}_i}$. Write $g_{i,j}:=\qq{\vec{a}_{ij}}{(\x',1)}$ where $\vec{a}_{ij}\in\R^n$. Pick $M_0>1$ such that $\|\vec{a}_{ij}\|\leq M_0$ for each pair $(i,j)$. We have
$$
|g_{i,j}(x')|=|\qq{\vec{a}_{ij}}{(x',1)}|\leq\|\vec{a}_{ij}\|\|(x',1)\|\leq M_0\sqrt{\|x'\|^2+1}.
$$
If $x_{n-1}\leq\veps$ and $M:=M_0\sqrt{1+\frac{1}{\veps^2}}$, then
\begin{equation}\label{cota}
|g_{i,j}(x')|\leq M_0\sqrt{\|x'\|^2+1}\leq M_0\sqrt{\|x''\|^2+\veps^2+1}\leq M\sqrt{\|x''\|^2+\veps^2}\leq\frac{M}{\veps}(\|x''\|^2+\veps^2).
\end{equation}
Pick $x:=(x',x_n)\in\setst{\setg{g}_i}$, then $(x',x_n-h_i(\x',0))\in\sets{\setg{g_i}}\subset\{\x_n>0\}$. By Lemma~\ref{polq1}(iii) we have $(x',0)\in\seta{\setg{g}_i}$, so $h_i(x',0)<0$. By By Lemma~\ref{polq1}(i) and \eqref{ajustefino4}
$$
x_n-h_i(x',0)\geq\frac{g_{m+1}(x')}{\sqrt{x_{n-1}g_{i,2}(x')\cdots g_{i,m}(x')}}\geq\frac{1}{\sqrt{x_{n-1}}}\cdot\frac{1}{\sqrt{g_{i,2}(x')\cdots g_{i,m}(x')}}-h_i(x',0).
$$
As $0<\veps<1$, we deduce by \eqref{cota}
$$
x_n^2+1\geq x_n\geq\frac{1}{\sqrt{x_{n-1}}}\cdot\frac{1}{\sqrt{g_{i,2}(x')\cdots g_{i,m}(x')}}\geq\frac{\veps^{\frac{m}{2}}}{\sqrt{x_{n-1}}(\sqrt{M}\sqrt{\|x''\|^2+\veps^2})^{m-1}}.
$$
Consequently,
\begin{equation}\label{eq:gamma}
x_{n-1}\geq\frac{\veps^{m}}{M^{m-1}(x_n^2+1)^2(\|x''\|^2+\veps^2)^{m-1}}
\end{equation}
for each point $(x',x_n)\in\setst{\setg{g}_i}$. Define
$$
P:=3\frac{M^{m-1}(\x_n^2+1)^2(\|\x''\|^2+\veps^2)^{m-1}}{\veps^{m}}
$$
and observe that by \eqref{eq:gamma} each $\setst{\setg{g}_i}\subset\vspanp{\Gamma}{\vem}$ where
$$
\Gamma:=\Big\{\x_{n-1}=\frac{1}{P}\Big\}.
$$
In addition, $\Gamma\subset\{0<\x_{n-1}<\veps\}$, $\vspanp{\Gamma}{\vem}\subset\{R\geq0\}$ and
$$
\setst{\setg{g}_i}\subset\Big\{\x_{n-1}\geq\frac{3}{P}\Big\}=\{\x_{n-1}P-1\geq2\}\subset\{R>1\},
$$
as claimed.
\end{proof}

\subsubsection{}\label{Fpm} 
Let $F:=R\prod_{j=1}^rb_j'\prod_{k=1}^{e-1}h_k\in\R[\x]$ be the product of the polynomial $R$, the linear equations $b_j'$ of the hyperplanes $B_j'$ and the linear equations $h_k$ of the hyperplanes $H_k$ spanned by the facets of $\pol$ except that of $\Ff$. It holds $\{F=0\}=\Gamma\cup\bigcup_{j=1}^rB_j'\cup\bigcup_{k=1}^{e-1}H_k$. As $B_j'$ is a separating hyperplane for $\Ff$ and $\Ff_j$, we have $\Ff\cap B_j'\subset\Ff\cap\Ff_j\subset\partial\Ff$. In addition, $\Gamma\subset\{0<\x_{n-1}<\veps\}$, so $\{F=0\}\cap\Int(\Ff)=\varnothing$. 

\subsubsection{}\label{p0clue}Let $\Int(\pol)\subset\Tt\subset\pol\setminus X$ be a semialgebraic set obtained by removing the interiors of some facets $\Ff_i$ of $\pol$ from $\pol\setminus X$ such that $\Ff_i\neq\Ff$. Define
\begin{align}
&\Pp:=\Tt\cap\Big(\vspanp{\Gamma}{\vem}\cup\bigcup_{j=1}^r\{b_j'\le 0\}\cup\bigcup_{i=1}^{e-1}\Ff_i\Big),\nonumber\\
&\Tt_0:=\Tt\cap\{\x_{n-1}\leq\veps\}\label{p0clueeq},\\
&\Pp_0:=\Tt_0\cap\Pp\nonumber.
\end{align}
We claim:\em
\begin{itemize}
\item[(i)] $\{\x_{n-1}\geq\veps\}\subset\vspanp{\Gamma}{\vem}$.
\item[(ii)] $\Tt_0\setminus\Ff=\Pp_0\cup\bigcup_{j=1}^r\vspanp{\Bb_j}{\ven}$.
\end{itemize}\em
\begin{proof}
(i) This inclusion follows from the fact that $\Gamma\subset\{0<\x_{n-1}<\veps\}$ can be understood as the graph over the hyperplane $\{\x_{n-1}=0\}$ of the regular function $\frac{1}{P}$, which depends on the variables $(\x'',\x_n)$.

(ii) Observe that $\Int(\Tt_0)=\Int(\pol_0)$. By the choice of $\veps$ the convex polyhedron $\pol_0$ satisfies the hypothesis of Lemma~\ref{allplanos0}, hence $\vspan{\Int(\Tt_0)}{\ven}=\vspan{\Int(\pol_0)}{\ven}=\bigcup_{j=1}^r\vspan{\Bb_j}{\ven}$. By \eqref{bibieq} we have
\begin{multline*}
\Int(\Tt_0)=\Int(\Tt_0)\cap(\vspan{\Int(\Tt_0)}{\ven})=\bigcup_{j=1}^r(\vspan{\Bb_j}{\ven}\cap\Int(\Tt_0))\\
\subset\bigcup_{j=1}^r(\vspanp{\Bb_j}{\ven}\cup(\{b_j'\leq0\}\cap\Int(\Tt_0)))\subset\p_0\cup\bigcup_{j=1}^r\vspanp{\Bb_j}{\ven}.
\end{multline*}
In addition, $\partial\Tt_0\setminus\Ff\subset\Pp\cap\Tt_0=\Pp_0$ (use (i) to guarantee that $\Tt_0\cap\{\x_{n-1}=\veps\}\subset\Pp_0$), so by \eqref{bibieq2}
$$
\Tt_0\setminus\Ff=\Int(\Tt_0)\cup(\partial\Tt_0\setminus\Ff)\subset\p_0\cup\bigcup_{j=1}^r\vspanp{\Bb_j}{\ven}\subset\Tt_0\setminus\Ff,
$$
as required.
\end{proof}

The interest of the semialgebraic set $\Pp$ comes from the following result, which is illustrated in Figure~\ref{fig9}.

\begin{figure}[!ht]
\begin{center}
\begin{tikzpicture}[scale=1]

\draw[draw=none,fill=gray!30,opacity=1] (0,4) -- (0,1) -- (2,0) -- (3,-0.25) -- (5,-0.25) -- (5,4);
\draw[draw=none,fill=gray!30,opacity=1] (8.1,4) .. controls (8.4,2.5) and (8.45,2.25) .. (8,1) -- (10,0) -- (11,-0.25) -- (13,-0.25) -- (13,4);

\draw[draw=none,fill=gray!70,opacity=1] (8,1) -- (10.25,4) -- (13,4) -- (13,-0.25) -- (11,-0.25) -- (10,0) -- (8,1);

\draw[draw=none,fill=gray!70,opacity=1] (8.2,4) .. controls (8.5,2.5) and (8.55,2.25) .. (8.75,2) .. controls (9.25,1.5) and (9.4,1) .. (9.45,0.3) -- (10,0)--(11,-0.25)--(13,-0.25)--(13,4)--
(0.2,4);
\draw[line width=1pt] (0,4) -- (0,1) -- (2,0) -- (3,-0.25);
\draw[line width=0.75pt,dashed] (8,4) -- (8,1);
\draw[line width=1pt] (8.1,4) .. controls (8.4,2.5) and (8.45,2.25) .. (8,1) -- (10,0) -- (11,-0.25);
\draw[line width=0.75pt,dashed] (1.5,0) -- (1.5,4);
\draw[line width=0.75pt,dashed] (9.5,0) -- (9.5,4);

\draw[line width=1pt] (0.2,4) .. controls (0.5,2.5) and (0.55,2.25) .. (0.75,2) .. controls (1.25,1.5) and (1.4,1) .. (1.45,0);
\draw[line width=1pt] (0,1) -- (2.25,4);

\draw[line width=0.75pt,dashed] (8.2,4) .. controls (8.5,2.5) and (8.55,2.25) ..(8.75,2) .. controls (9.25,1.5) and (9.4,1) .. (9.45,0);
\draw[line width=0.75pt,dashed] (8,1) -- (10.25,4); 

\draw[fill=white,draw] (0,1) circle (0.75mm);
\draw[fill=white] (2,0) circle (0.75mm);
\draw[fill=white,draw] (8,1) circle (0.75mm);
\draw[fill=white] (10,0) circle (0.75mm);

\draw[line width=1pt,->] (5.5,2) -- (7.5,2);
\draw(6.5,2.25) node{\small$f_0$};
\draw(4.5,3.7) node{\small$\Tt$};
\draw(-0.25,2.3) node{\small$\Ff$};
\draw(0.7,3) node{\small$\Gamma$};
\draw(8.7,3) node{\small$\Gamma$};
\draw(2.2,3.3) node{\small$B_i'$};
\draw(10.2,3.3) node{\small$B_i'$};
\draw(1.8,1.5) node[rotate=90]{\small${\tt x}_{n-1}=\varepsilon$};
\draw(9.8,1.5) node[rotate=90]{\small${\tt x}_{n-1}=\varepsilon$};
\draw(12,3.7) node{\small$f_0(\Tt)\supset\Pp$};
\end{tikzpicture}
\end{center}
\caption{Behavior of the polynomial map $f_0$ (Lemma~\ref{setp}).}\label{fig9}
\end{figure}
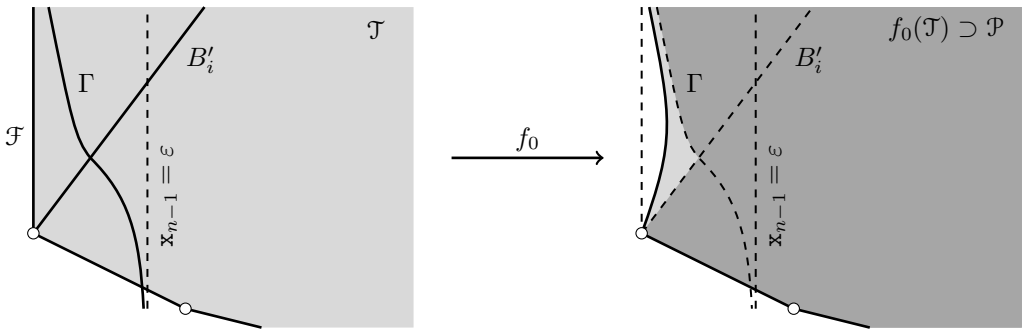

\begin{lem}\label{setp}
The polynomial map $f_0:\R^n\to\R^n,\ x\mapsto x+F^2(x)\vem$ satisfies $\Pp\subset f_0(\Tt)\subset\Tt\setminus\Ff$.
\end{lem}
\begin{proof}
The inclusion $f_0(\Tt)\subset\Tt\setminus\Ff$ follows from Lemma~\ref{allplanos2} and \ref{Fpm}. We prove next $\Pp\subset f_0(\Tt)$.

Pick $x:=(x'',x_{n-1},x_n)\in\Pp$ and consider the intersection $\Pp\cap\vspan{x}{\vem}$. This intersection consists of finitely many intervals of the line $\vspan{x}{\vem}$ whose endpoints belong to $\{F=0\}=\Gamma\cup\bigcup_{j=1}^rB_j'\cup\bigcup_{i=1}^{e-1}\Ff_i$, so they are fixed by $f_0$. As $\lim_{x_{n-1}\to+\infty}f_0(x'',x_{n-1},x_n)
=+\infty$, we have by Corollary~\ref{fun2}
$$
x\in\Pp\cap\vspan{x}{\vem}\subset f_0(\Pp\cap\vspan{x}{\vem})\subset f_0(\Tt).
$$
Thus, $\Pp\subset f_0(\Tt)$, as required.
\end{proof}

The image of the polynomial map $f_0$ is contained in $\Tt\setminus\Ff$ and contains $\Pp$. The semialgebraic set $\Pp$ leaves a `gap' inside $\Tt$ in a neighborhood of the facet $\Ff$. Our next goal is to construct another polynomial map to fill the gap that $f_0(\Tt)$ leaves inside $\Tt\setminus\Ff$.

\subsection{Filling the interior gap of the convex polyhedron} 

Let $\pol\subset\R^n$ be an unbounded convex polyhedron with recession cone $\conv{\pol}{}$ of dimension $n$. Let $X$ be the union of the affine subspaces of $\R^n$ spanned by the faces of $\pol$ of dimension $n-2$. 

\begin{prop}\label{detach}
Let $\Ff$ be one of the unbounded facets of $\pol$ and let $\Int(\pol)\subset\Tt\subset\pol\setminus X$ be a semialgebraic set obtained by removing the interiors of some facets $\Ff_i$ of $\pol$ from $\pol\setminus X$ such that $\Ff_i\neq\Ff$. Then there exists a polynomial map $F:\R^n\to\R^n$ such that $F(\Tt)=\Tt\setminus\Ff$.
\end{prop}

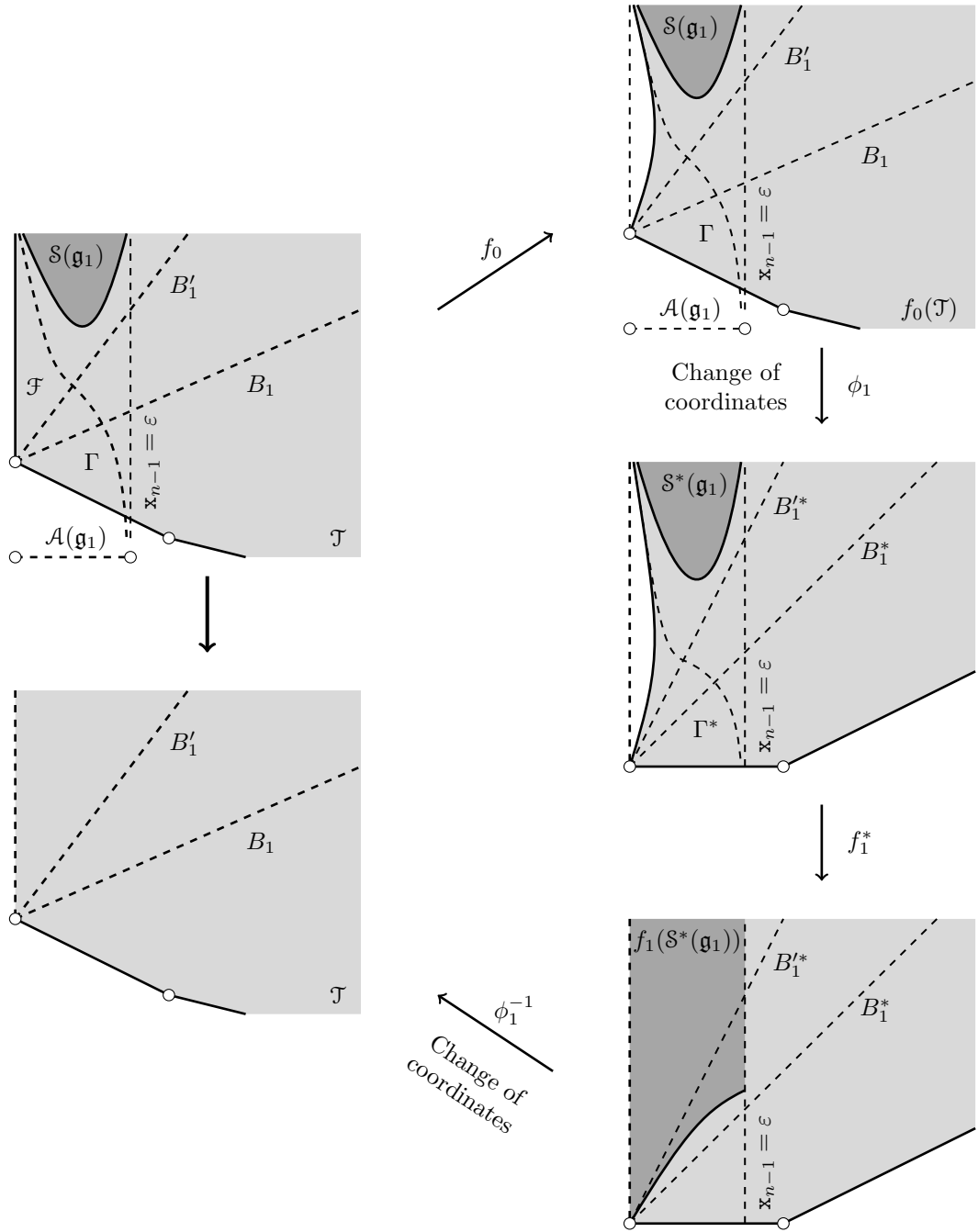
\begin{figure}[!ht]
\begin{center}
\begin{tikzpicture}[scale=1.1]


\draw[draw=none,fill=gray!30,opacity=1] (0,7) -- (0,4) -- (2,3) -- (3,2.75) -- (4.5,2.75) -- (4.5,7);
\draw[draw=none,fill=gray!30,opacity=1] (8.05,10) .. controls (8.4,8.5) and (8.45,8.25) .. (8,7) -- (10,6) -- (11,5.75) -- (12.5,5.75) -- (12.5,10);

\draw[draw=none,fill=gray!70,opacity=1] (0.1,7) .. controls (0.75,5.5) and (1,5.25) .. (1.45,7);
\draw[draw=none,fill=gray!70,opacity=1] (8.1,10) .. controls (8.75,8.5) and (9,8.25) .. (9.45,10);

\draw[line width=1pt] (0,7) -- (0,4) -- (2,3) -- (3,2.75);
\draw[line width=0.75pt,dashed] (8,10) -- (8,7);
\draw[line width=1pt] (8.05,10) .. controls (8.4,8.5) and (8.45,8.25) .. (8,7) -- (10,6) -- (11,5.75);
\draw[line width=0.75pt,dashed] (1.5,3) -- (1.5,7);
\draw[line width=0.75pt,dashed] (9.5,6) -- (9.5,10);

\draw[line width=1pt,dashed] (0.05,7) .. controls (0.4,5.5) and (0.45,5.25) .. (0.75,5) .. controls (1.25,4.5) and (1.4,4) .. (1.45,3);
\draw[line width=1pt,dashed] (0,4) -- (2.25,7);
\draw[line width=1pt,dashed] (0,4) -- (4.5,6);
\draw[line width=1pt,dashed] (0,2.75) -- (1.5,2.75);
\draw[line width=1pt] (0.1,7) .. controls (0.75,5.5) and (1,5.25) .. (1.45,7);
\draw[line width=1pt] (8.1,10) .. controls (8.75,8.5) and (9,8.25) .. (9.45,10);

\draw[line width=0.75pt,dashed] (8.05,10) .. controls (8.4,8.5) and (8.45,8.25) ..(8.75,8) .. controls (9.25,7.5) and (9.4,7) .. (9.45,6);
\draw[line width=0.75pt,dashed] (8,7) -- (10.25,10); 
\draw[line width=0.75pt,dashed] (8,7) -- (12.5,9);
\draw[line width=0.75pt,dashed] (8,5.75) -- (9.5,5.75);

\draw[fill=white,draw] (0,4) circle (0.75mm);
\draw[fill=white,draw] (2,3) circle (0.75mm);
\draw[fill=white,draw] (8,7) circle (0.75mm);
\draw[fill=white,draw] (10,6) circle (0.75mm);
\draw[fill=white,draw] (0,2.75) circle (0.75mm);
\draw[fill=white,draw] (1.5,2.75) circle (0.75mm);
\draw[fill=white,draw] (8,5.75) circle (0.75mm);
\draw[fill=white,draw] (9.5,5.75) circle (0.75mm);

\draw[line width=1pt,->] (5.5,6) -- (7,7);
\draw(6.2,6.8) node{\small$f_0$};
\draw(4.2,3) node{\small$\Tt$};
\draw(0.25,5) node{\small$\Ff$};
\draw(1,4) node{\small$\Gamma$};
\draw(9,7) node{\small$\Gamma$};
\draw(2.2,6.3) node{\small$B_1'$};
\draw(3.2,5) node{\small$B_1$};
\draw(10.2,9.3) node{\small$B_1'$};
\draw(11.2,8) node{\small$B_1$};
\draw(1.8,4) node[rotate=90]{\small${\tt x}_{n-1}=\varepsilon$};
\draw(9.8,7) node[rotate=90]{\small${\tt x}_{n-1}=\varepsilon$};
\draw(11.9,6) node{\small$f_0(\Tt)$};
\draw(0.8,3) node{\small$\seta{\setg{g}_1}$};
\draw(8.8,6) node{\small$\seta{\setg{g}_1}$};
\draw(0.8,6.7) node{\small$\sets{\setg{g_1}}$};
\draw(8.8,9.7) node{\small$\sets{\setg{g_1}}$};

\draw[line width=1pt,->] (10.5,5.5) -- (10.5,4.5);
\draw(11,5) node{\small$\phi_1$};
\draw(9.25,5.15) node{\small Change of};
\draw(9.25,4.8) node{\small coordinates};
\draw[line width=1.5pt,->] (2.5,2.5) -- (2.5,1.5);


\draw[draw=none,fill=gray!30,opacity=1] (8,-2) -- (8,-6) -- (10,-6) -- (12.5,-4.75)-- (12.5,-2);

\draw[draw=none,fill=gray!70,opacity=1] (8,-2) -- (8,-6) .. controls (8.75,-4.75) and (9,-4.5) .. (9.5,-4.25) -- (9.5,-2);

\draw[line width=1pt,dashed] (8,-2) -- (8,-6); 
\draw[line width=1pt] (8,-6) -- (10,-6) -- (12.5,-4.75);
\draw[line width=1pt] (8,-6) .. controls (8.75,-4.75) and (9,-4.5) .. (9.5,-4.25);

\draw[line width=0.75pt,dashed] (8,-6) -- (12,-2); 
\draw[line width=0.75pt,dashed] (8,-6) -- (10,-2); 
\draw[line width=0.75pt,dashed] (9.5,-6) -- (9.5,-2); 

\draw[fill=white,draw] (8,-6) circle (0.75mm);
\draw[fill=white,draw] (10,-6) circle (0.75mm);
\draw(10.1,-2.6) node{\small$B_1'^*$};
\draw(11.2,-3.2) node{\small$B_1^*$};
\draw(9.8,-5.2) node[rotate=90]{\small${\tt x}_{n-1}=\varepsilon$};
\draw(8.75,-2.3) node{\small$f_1(\setst{\setg{g_1}})$};


\draw[line width=1pt,->] (10.5,-0.5) -- (10.5,-1.5);
\draw(11,-1) node{\small$f_1^*$};

\draw[line width=1pt,<-] (5.5,-3) -- (7,-4);
\draw(6.5,-3.2) node{\small$\phi_1^{-1}$};

\draw(6,-4) node[rotate=-33.7]{\small Change of};
\draw(5.8,-4.29) node[rotate=-33.7]{\small coordinates};


\draw[draw=none,fill=gray!30,opacity=1] (8.05,4) .. controls (8.4,1.7) and (8.45,1.475) .. (8,0) -- (10,0) -- (12.5,1.25) -- (12.5,4);

\draw[draw=none,fill=gray!70,opacity=1] (8.1,4) .. controls (8.6,2.2) and (9.1,1.7) .. (9.45,4);

\draw[line width=1pt,dashed] (8,4) -- (8,0); 
\draw[line width=1pt] (8.05,4) .. controls (8.4,1.7) and (8.45,1.475) .. (8,0) -- (10,0) -- (12.5,1.25);
\draw[line width=1pt] (8.1,4) .. controls (8.6,2.2) and (9.1,1.7) .. (9.45,4);

\draw[line width=0.75pt,dashed] (8,0) -- (12,4); 
\draw[line width=0.75pt,dashed] (8,0) -- (10,4); 
\draw[line width=0.75pt,dashed] (9.5,0) -- (9.5,4);

\draw[line width=0.75pt,dashed] (8.05,4) .. controls (8.4,1.7) and (8.45,1.475) .. (8.75,1.375) .. controls (9.25,1.125) and (9.4,0.7) .. (9.45,0);

\draw[fill=white,draw] (8,0) circle (0.75mm);
\draw[fill=white,draw] (10,0) circle (0.75mm);
\draw(9,0.5) node{\small$\Gamma^*$};
\draw(10.1,3.4) node{\small$B_1'^*$};
\draw(11.2,2.8) node{\small$B_1^*$};
\draw(9.8,0.8) node[rotate=90]{\small${\tt x}_{n-1}=\varepsilon$};
\draw(8.85,3.7) node{\small$\setst{\setg{g_1}}$};

\draw[draw=none,fill=gray!30,opacity=1] (0,1) -- (0,-2) -- (2,-3) -- (3,-3.25) -- (4.5,-3.25) -- (4.5,1);

\draw[line width=1pt] (0,-2) -- (2,-3) -- (3,-3.25);
\draw[line width=1pt,dashed] (0,1) -- (0,-2);

\draw(4.2,-3) node{\small$\Tt$};
\draw[line width=1pt,dashed] (0,-2) -- (2.25,1);
\draw[line width=1pt,dashed] (0,-2) -- (4.5,0);
\draw[fill=white,draw] (0,-2) circle (0.75mm);
\draw[fill=white,draw] (2,-3) circle (0.75mm);

\draw(2.2,0.3) node{\small$B_1'$};
\draw(3.2,-1) node{\small$B_1$};

\end{tikzpicture}
\end{center}
\caption{Erasing an unbounded facet of a convex polyhedron.}\label{fig10}
\end{figure}

\begin{proof}
Assume first that $\pol$ is placed as described in \ref{preparat}, take into account all considerations developed thereafter and keep the used notations. We have constructed a polynomial map $f_0:\R^n\to\R^n$ such that $\Pp\subset f_0(\Tt)\subset\Tt\setminus\Ff$ (see Lemma~\ref{setp}). By \ref{p0clue}(i)
$$
\Tt_1:=\Tt\cap\{\x_{n-1}\ge\veps\}=\Pp\cap\{\x_{n-1}\ge\veps\}.
$$
Fix $1\leq i\leq r$ and consider the polynomial 
$$
P_{i0}:=\prod_{j=1}^r(b_j'\circ\phi_i^{-1})\prod_{k=1}^s(h_k\circ\phi_i^{-1}).
$$
By Lemma~\ref{polq1} and equations \eqref{ajustefino3} and \eqref{ajustefino4} we have $\setst{\setg{g_i}}=\phi_i(\sets{\setg{g}_i})\subset\{P_{i0}>1\}$. For each $T\subset\R^n$ we denote $T^*$ the set $\phi_i(T)$. It holds 
$$
\{P_{i0}=0\}=\bigcup_{j=1}^rB_j'^*\cup\bigcup_{k=1}^sH_k^*.
$$
Define $P_{i1}:=R\circ\phi_i^{-1}$. By \eqref{si} $\sets{\setg{g}_i}\subset\{R>1\}$, hence $\setst{\setg{g_i}}\subset\{P_{i1}>1\}$. 
Define $P_i:=(P_{i0}P_{i1})^2$ and note that 
$$
\sets{\setg{g}_i}\subset\setst{\setg{g_i}}\subset\{P_{i0}>1\}\cap\{P_{i1}>1\}\subset\{P_i>1\}. 
$$
Consider the polynomial maps
$$
f_i:=(f_{i1},\ldots,f_{in}):\R^n\to\R^n,\ (x',x_n)\mapsto(x',x_n(1+P_i(x)\polq{\setg{g}_i}(x))^2+2x_{n-1}(P_i(x)\polq{\setg{g}_i}(x))^2)
$$
and
\begin{equation}\label{hatfi}
\hat{f}_i:=\phi_i^{-1}\circ f_i\circ\phi_i.
\end{equation}
Note that $\setg{g}:=\setg{g}_i$, $P:=P_i$, $g_{m+1}:=g_{i,m+1}$ and $h:=2\x_{n-1}$ satisfy the hypotheses of Theorem~\ref{atico2b}.

\subsubsection{}\label{tau1}
We claim: \em each polynomial map $\hat{f}_i$ satisfies $\hat{f}_i(\Tt_1)=\Tt_1$\em. 

To prove that $\hat{f}_i(\Tt_1)=\Tt_1$ it is enough to show: $f_i(\Tt_1^*)=\Tt_1^*$. It holds $\Tt_1^*\subset\{\x_{n-1}\geq\veps,\x_n\geq0\}$.

As $\sets{\setg{g_i}}=\{\polq{\setg{g}_i}\leq0,\x_n\geq0\}\subset\{\x_{n-1}\leq\veps\}$, the polynomial $\polq{\setg{g}_i}$ is positive on $\Tt_1^*$, as well as $P_i$, which is a square, and $\x_{n-1}$. Thus, the inclusion $f_i(\Tt_1^*)\subset\Tt_1^*$ holds by Theorem~\ref{atico2b}(ii) because $\ven\in\conv{\pol}{}$. As the non-vertical facets of $\pol^*$ are contained in $\{P_i=0\}$ and by Theorem~\ref{atico2b}(iii) $\lim_{x_n\to\infty}f_i(x',x_n)=+\infty$ for each $x'\in\{\x_{n-1}\geq0\}$, we deduce by Corollary~\ref{fun2} $\Tt_1^*\subset f_i(\Tt_1^*)$.

\subsubsection{}\label{clpr} 
Let us study the behavior of $\hat{f}_i$ on $\Pp_0\cup\bigcup_{j=1}^{i-1}\vspanp{\Bb_j}{\ven}$. We claim:
\begin{equation}\label{induct}
\Pp_0\cup\bigcup_{j=1}^i\vspanp{\Bb_j}{\ven}\subset\hat{f}_i\Big(\Pp_0\cup\bigcup_{j=1}^{i-1}\vspanp{\Bb_j}{\ven}\Big)\subset\hat{f}_i(\Tt_0\setminus\Ff)\subset\Tt_0\setminus\Ff.
\end{equation}
By \ref{p0clue}(ii) to prove the previous chain of inclusions it is enough to show
\begin{equation}\label{induct2}
\Pp_0^*\cup\bigcup_{j=1}^i\vspanp{\Bb_j^*}{\ven}\subset f_i\Big(\Pp_0^*\cup\bigcup_{j=1}^{i-1}\vspanp{\Bb_j^*}{\ven}\Big)\quad\text{and}\quad f_i(\Tt_0^*\setminus\Ff^*)\subset\Tt_0^*\setminus\Ff^*.
\end{equation}

Pick a point $x:=(x',x_n)\in\Pp_0^*\cup\bigcup_{j=1}^{i-1}\vspanp{\Bb_j^*}{\ven}$ and consider the ray $\Pp_{0,x}^*:=\vspan{x}{\ven}\cap\Pp_0^*\subset\{\x_n\geq0\}$, which is a finite union of intervals inside the ray $\vspan{x}{\ven}\cap\{\x_n>0\}$ whose endpoints belong to 
$$
\Gamma^*\cup\bigcup_{j=1}^rB_j'^*\cup\bigcup_{k=1}^sH_k^*=\{P_i=0\},
$$
so they are fixed by $f_i$. In addition, by Theorem~\ref{atico2b}(iii)
$$
\lim_{t\to\infty}f_{in}(x',t)=+\infty,
$$
because $x'\in\{\x_{n-1}\geq0\}$, hence by Corollary~\ref{fun2} $\Pp_{0x}^*\subset f_i(\Pp_{0x}^*)$. Thus, 
\begin{equation}\label{eq:uno}
\Pp_0^*\cup\bigcup_{j=1}^{i-1}\vspanp{\Bb_j^*}{\ven}\subset f_i\Big(\Pp_0^*\cup\bigcup_{j=1}^{i-1}\vspanp{\Bb_j^*}{\ven}\Big).
\end{equation}

Observe that $\vspan{\Bb_i^*}{\ven}=\vspan{\Bb_i}{\ven}$ because $\pi_n(\Bb_i^*)=\pi_n(\Bb_i)=\seta{\setg{g}_i}$. In addition, by \eqref{iso} $B_i^*=\{\x_n-\x_{n-1}=0\}$ and $B_i'^*=\{\x_n-2\x_{n-1}=0\}$, so $\vspanp{B_i^*}{\ven}=\{\x_n-\x_{n-1}\geq0\}$ and $\vspanp{B_i'^*}{\ven}=\{\x_n-2\x_{n-1}\geq0\}$. By Theorem~\ref{atico2b}(i)
$$
\vspanp{\Bb_i^*}{\ven}\cap\vspanp{B_i'^*}{\ven}=\vspanp{\seta{\setg{g}_i}}{\ven}\cap\{\x_n\ge 2\x_{n-1}\}\subset f_i(\sets{\setg{g}_i}).
$$
As $\vspanp{B_i'}{(-\ven)}\cap\Tt=\{b_i'\leq0\}\cap\Tt\subset\Pp$, we have $\vspanp{\Bb_i^*}{\ven}\cap\vspanp{B_i'^*}{(-\ven)}\subset\Pp_0^*$. By \eqref{si}
$$
\sets{\setg{g}_i}\subset\vspanp{\Gamma^*}{\ven}\cap\vspanp{\seta{\setg{g}_i}}{\ven}\cap\{\x_n>0\}\subset\vspanp{\Gamma^*}{\ven}\cap\Tt^*\cap\{0<\x_{n-1}<\veps\}\subset\Pp_0^*.
$$
Consequently, by \eqref{p0clueeq} and \eqref{eq:uno}
$$
\vspanp{\Bb_i^*}{\ven}=(\vspanp{\Bb_i^*}{\ven}\cap\vspanp{B_i'^*}{(-\ven)})\cup(\vspanp{\Bb_i^*}{\ven}\cap\vspanp{B_i'^*}{\ven})\subset\Pp_0^*\cup f_i(\sets{\setg{g}_i})\subset f_i(\Pp_0^*).
$$
Therefore, $\Pp_0^*\cup\bigcup_{j=1}^i\vspanp{\Bb_j^*}{\ven}\subset f_i\Big(\Pp_0^*\cup\bigcup_{j=1}^{i-1}\vspanp{\Bb_j^*}{\ven}\Big)$.

Let us check next: $f_i(\Tt_0^*\setminus\Ff^*)\subset\Tt_0^*\setminus\Ff^*$.

By Theorem~\ref{atico2b}(i) and \ref{p0clue}(ii)
$$
f_i(\sets{\setg{g}_i})\subset\vspanp{\seta{\setg{g}_i}}{\ven}\cap\{\x_n\ge\x_{n-1}\}=\vspanp{\Bb_i^*}{\ven}\subset\Tt_0^*\setminus\Ff^*. 
$$
The polynomial $\polq{\setg{g}_i}$ is positive on $\{\x_n\geq0,\x_{n-1}\geq0\}\setminus\sets{\setg{g}_i}$, as well as $P_i$, which is a square, and $\x_{n-1}$. Thus, the inclusion $f_i(\Tt_0^*\setminus(\Ff^*\cup\sets{\setg{g}_i}))\subset\Tt_0^*\setminus\Ff^*$ holds by Theorem~\ref{atico2b}(ii) because $\ven\in\conv{\pol}{}$. We conclude $f_i(\Tt_0^*\setminus\Ff^*)\subset\Tt_0^*\setminus\Ff^*$.

\subsubsection{}\label{finproof}
Define $F:=\hat{f}_r\circ\cdots\circ\hat{f}_1\circ f_0$. By Lemma~\ref{setp} 
\begin{equation}\label{tau1eq}
\Pp_0\cup\Tt_1\subset\Pp\subset f_0(\Tt)\subset\Tt\setminus\Ff. 
\end{equation}
By \ref{tau1} and \eqref{induct}
$$
\hat{f}_i(\Tt\setminus\Ff)=\hat{f}_i(\Tt_1\cup(\Tt_0\setminus\Ff))\subset\Tt_1\cup(\Tt_0\setminus\Ff)=\Tt\setminus\Ff
$$
for $i=1,\ldots,r$. Thus, by \eqref{tau1eq}
\begin{equation}\label{tau1eq0}
F(\Tt)\subset(\hat{f}_r\circ\cdots\circ\hat{f}_1)(\Tt\setminus\Ff)\subset\Tt\setminus\Ff. 
\end{equation}

By \ref{p0clue}, \eqref{induct}, \eqref{tau1eq} and \eqref{tau1eq0} we deduce 
$$
\Tt_0\setminus\Ff=\Pp_0\cup\bigcup_{j=1}^r\vspanp{\Bb_j}{\ven}\subset(\hat{f}_r\circ\cdots\circ
\hat{f}_1)(\Pp_0)\subset F(\Tt)\subset\Tt\setminus\Ff.
$$
In addition, by \ref{tau1} and \eqref{tau1eq} we have $\Tt_1\subset F(\Tt)\subset\Tt\setminus\Ff$. Consequently,
$$
\Tt\setminus\Ff=(\Tt_0\setminus\Ff)\cup\Tt_1\subset F(\Tt)\subset\Tt\setminus\Ff,
$$
so $F(\Tt)=\Tt\setminus\Ff$, as required.
\end{proof}

Figure~\ref{fig10} shows the combined action of the polynomial map $f_0:\Tt\to f_0(\Tt)$ appearing in Lemma~\ref{setp} and the polynomial map $\hat{f}_1:f_0(\Tt)\to\Tt\setminus\Ff$ constructed in \eqref{hatfi}.

\subsection{Proof of Theorem~\ref{main2}}
By Theorem~\ref{main11} there exists a polynomial map $f_0:\R^n\to\R^n$ such that $f_0(\R^n)=\pol\setminus X$ where $X$ is the union of the affine subspaces of $\R^n$ spanned by the faces of $\pol$ of dimension $n-2$. Let $\Ff_1,\ldots,\Ff_m$ be the facets of $\pol$. By Proposition~\ref{detach} there exists a polynomial map $F_i:\R^n\to\R^n$ such that
$$
F_i\Big((\pol\setminus X)\setminus\bigcup_{j=1}^{i-1}\Ff_j\Big)=(\pol\setminus X)\setminus\bigcup_{j=1}^i\Ff_j.
$$
for $i=1,\ldots,m$. Consider the polynomial map $f:=(F_m\circ\cdots\circ F_1\circ f_0):\R^n\to\R^n$. Thus,
$$
f(\R^n)=(\pol\setminus X)\setminus\bigcup_{j=1}^r\Ff_j=\Int(\pol),
$$
as required.
\qed

\appendix
\section{Some basic inequalities}

Some useful inequalities concerning finite collections of positive numbers have been used in Section~\ref{s2}. We collect them in the following lemma for easy reference.

\begin{lem}\label{ineq} 
Let $y_1,\dots,y_m$ be positive real numbers and fix $1\leq i\leq m$. Then the following inequalities hold:
\begin{itemize}
\item[(i)] $y_1+\cdots+y_m+\dfrac{1}{y_1\cdots y_m}\ge m+1>1$.
\item[(ii)] $y_1+\cdots+y_m+\dfrac{1}{y_1\cdots y_m}\ge y_i+m\sqrt[m]{\dfrac{1}{y_i}}$.
\item[(iii)] $\Big(y_1+\cdots+y_m+\dfrac{1}{y_1\cdots y_m}\Big)^my_i>m^m\ge1$.
\end{itemize}
\end{lem}
\begin{proof} 
(i) Denote $z:=\prod_{i=1}^my_i$. It is enough to show
\begin{equation}\label{ineq1}
z\leq\frac{1+z\sum_{i=1}^my_i}{m+1}.
\end{equation}
Consider the positive real numbers $z_i:=y_iz$ for $i=1,\ldots,m$ and $z_{m+1}=1$. By the arithmetic-geometric inequality
$$
\sqrt[m+1]{\prod_{i=1}^{m+1}z_i}\leq\frac{\sum_{i=1}^{m+1}z_i}{m+1}.
$$
As $\prod_{i=1}^{m+1}z_i=z^{m+1}$ and $\sum_{i=1}^{m+1}z_i=1+z\sum_{i=1}^my_i$, inequality \eqref{ineq1} holds.

(ii) By the arithmetic-geometric inequality
$$
\sqrt[m]{\frac{1}{y_i}}=\sqrt[m]{\frac{1}{y_1\cdots y_m}\prod_{j\neq i}y_j}\leq\frac{{\displaystyle\frac{1}{y_1\cdots y_m}}+\sum_{j\neq i}y_j}{m},
$$
so the statement holds.

(iii) Using (ii) we have
$$
\Big(y_1+\cdots+y_m+\frac{1}{y_1\cdots y_m}\Big)^my_i\ge\Big(y_i+m\sqrt[m]{\frac{1}{y_i}}\Big)^my_i\ge y_i^{m+1}+m^m>m^m,
$$
as required.
\end{proof}


\begin{thebibliography}{FGU2}

\bibitem[AG1]{ag1} C. Andradas, JM. Gamboa: A note on projections of real algebraic varieties. \em Pacific J. Math\em. {\bf115} (1984), 1--11.

\bibitem[AG2]{ag2} C. Andradas, JM. Gamboa: On projections of real algebraic varieties. \em Pacific J. Math\em. {\bf121} (1986), 281--291.

\bibitem[Be]{ber1} M. Berger: Geometry. I \& II. \em Universitext\em. Springer-Verlag, Berlin (1987).

\bibitem[Bi]{bing} R.H. Bing: Generalizations of two theorems of Janiszewski. \em Bull. Amer. Math. Soc. \em {\bf 51} (1945), 954--960.

\bibitem[BCR]{bcr} J. Bochnak, M. Coste, M.F. Roy: Real algebraic geometry. {\em Ergeb. Math.} {\bf 36}. Springer-Verlag, Berlin (1998).

\bibitem[Fe]{f1} J.F. Fernando: On the one dimensional polynomial and regular images of $\R^n$. \em J. Pure Appl. Algebra \em {\bf 218} (2014), no. 9, 1745--1753.
 
\bibitem[FG1]{fg1} J.F. Fernando, J.M. Gamboa: Polynomial images of $\R^n$. \em J. Pure Appl. Algebra \em {\bf 179} (2003), no. 3, 241--254.

\bibitem[FG2]{fg2} J.F. Fernando, J.M. Gamboa: Polynomial and regular images of $\R^n$. \em Israel J. Math. \em {\bf 153} (2006), 61--92.

\bibitem[FGU1]{fgu1} J.F. Fernando, J.M. Gamboa, C. Ueno: On convex polyhedra as regular images of $\R^n$. \em Proc. London Math. Soc. \em (3) {\bf 103} (2011), 847--878. 

\bibitem[FGU2]{fgu2} J.F. Fernando, J.M. Gamboa, C. Ueno: The open quadrant problem: A topological proof. {\em A Mathematical tribute to Professor Jos\'e Mar\'ia Montesinos Amilibia.} Departamento de Geometr\'ia y Topolog\'ia. Facultad de Ciencias Matem\'aticas, UCM (2016) 137--350.

\bibitem[FU1]{fu1} J.F. Fernando, C. Ueno: On the set of points at infinity of a polynomial image of $\R^n$. {\em Discrete Comput. Geom.} {\bf 52} (2014), no. 4, 583--611.

\bibitem[FU2]{fu2} J.F. Fernando, C. Ueno: On complements of convex polyhedra as polynomial and regular images of $\R^n$. {\em Int. Math. Res. Not.} IMRN {\bf 2014}, no. 18, 5084--5123.

\bibitem[FU3]{fu3} J.F. Fernando, C. Ueno: On the complements of $3$-dimensional convex polyhedra as polynomial images of $\R^3$. {\em Internat. J. Math.} {\bf 25}, no. 7, 1450071, 18 pp, (2014).

\bibitem[FU4]{fu4} J.F. Fernando, C. Ueno: On complements of convex polyhedra as polynomial images of $\R^n$. {\em Preprint} RAAG (2015). {\tt arXiv:1412.5107}

\bibitem[FU5]{fu5} J.F. Fernando, C. Ueno: A short proof for the open quadrant problem. {\em J. Symbolic Comput.} {\bf79} (2017), no. 1, 57--64.

\bibitem[G]{g} J.M. Gamboa: Reelle Algebraische Geometrie, June, $10^{\text{th}}-16^{\text{th}}$ (1990), Oberwolfach.

\bibitem[GRS]{grs} M. J. Gonz\'alez-L\'opez, T. Recio, F. Santos: Parametrization of semialgebraic sets. {\em Symbolic computation, new trends and developments} (Lille, 1993). {\em Math. Comput. Simulation} {\bf42} (1996), no. 4--6, 353--362. 

\bibitem[HRR]{hrr} J. Heintz, T. Recio, M.-F. Roy: Algorithms in real algebraic geometry and applications to computational geometry. {\em Discrete and computational geometry} (New Brunswick, NJ, 1989/1990), 137--163, DIMACS Ser. {\em Discrete Math. Theoret. Comput. Sci.}, {\bf6}, Amer. Math. Soc., Providence, RI (1991).

\bibitem[KPS]{kps} K. Kubjas, P.A. Parrilo, B. Sturmfels: How to flatten a soccer ball. {\em Preprint} (2016). {\tt arXiv:1606.02253}

\bibitem[J]{ja} Z. Janiszewski: Sur les coupures du plan faites par les continus, \em Prace Matematyczno-Fizyczne \em {\bf 26} (1913), 11--63. 

\bibitem[Mo]{m2} T. S. Motzkin: The real solution set of a system of algebraic inequalities is the projection of a hypersurface in one more dimension. 1970 {\em Inequalities, II} (Proc. Second Sympos., U.S. Air Force Acad., Colo., 1967) pp. 251--254 Academic Press, New York.

\bibitem[NDS]{nds} J. Nie, J. Demmel, B. Sturmfels: Minimizing polynomials via sum of squares over the gradient ideal. \em Math. Program. \em {\bf106} (2006), no. 3, Ser. A, 587-606.

\bibitem[PS]{ps} P.A. Parrilo, B. Sturmfels: Minimizing polynomial functions. \em Algorithmic and quantitative real algebraic geometry \em (Piscataway, NJ, 2001), 83-99, \em DIMACS Ser. Discrete Math. Theoret. Comput. Sci.\em, {\bf60}, Amer. Math. Soc., Providence, RI, 2003.

\bibitem[P]{p} D. Pecker: On the Elimination of Algebraic Inequalities, \em Pacific J. Math.\em, {\bf146} (1990), 305--314.

\bibitem[R]{r} T.R. Rockafellar: Convex analysis. \em Princeton Mathematical Series\em, {\bf28}. Princeton University Press, Princeton, N.J (1970).

\bibitem[Sch]{sch} M. Schweighofer: Global optimization of polynomials using gradient tentacles and sums of squares. \em SIAM J. Optim. \em {\bf17} (2006), no. 3, 920--942.

\bibitem[S]{s} G. Stengle: A Nullstellensatz and a Positivstellensatz in semialgebraic geometry. {\em Math. Ann.} {\bf 207}, (1974) 87--97.

\bibitem[U]{u2} C. Ueno: On convex polygons and their complements as images of regular and polynomial maps of $\R^2$. \em J. Pure Appl. Algebra \em {\bf 216} (2012), no. 11, 2436--2448. 

\bibitem[VS]{vs} H.H. Vui, P.T. So'n: Global optimization of polynomials using the truncated tangency variety and sums of squares. {\em SIAM J. Optim.} {\bf19} (2008), no. 2, 941-951.

\bibitem[Z]{z} G.M. Ziegler: Lectures on Polytopes. \em Graduate Texts in Mathematics {\bf 152}\em. Springer-Verlag, New-York (1995).
\end{thebibliography}
\end{document}